\documentclass[11pt]{article}

\usepackage{amssymb,amsmath,bm}
\usepackage{amsthm}
\usepackage{empheq}
\usepackage{textcomp}
\usepackage{enumerate}      
\usepackage{graphicx}        
\usepackage{caption}
\usepackage{nicematrix}
\usepackage{tikz-cd}

\usepackage{mathrsfs}

\usepackage{url}

\renewcommand{\setminus}{{\smallsetminus}}

\usepackage{amssymb,amsmath,bm}
\usepackage{amsthm}
\usepackage{hyperref}
\usepackage{mathrsfs}
\usepackage{textcomp}
\usepackage{enumerate}
\usepackage{graphicx}
\usepackage{tikz}
\usepackage{verbatim}

\newtheorem{theorem}{Theorem}[section]
\newtheorem{lemma}[theorem]{Lemma}
\newtheorem{proposition}[theorem]{Proposition}
\newtheorem{definition}[theorem]{Definition}
\newtheorem{corollary}[theorem]{Corollary}
\newtheorem{conjecture}[theorem]{Conjecture}

\theoremstyle{remark}
\newtheorem{remark}[theorem]{Remark}

\theoremstyle{remark}

\numberwithin{equation}{section}

\def\Vol{\operatorname{Vol}}
\def \CC{\mathbb C}
\def \Re{\operatorname{Re}}

\def\SS{{\mathbb S}}
\def\CC{{\mathbb C}}
\def\RR{{\mathbb R}}
\def\ZZ{{\mathbb Z}}
\def\NN{{\mathbb N}}

\def\Vol{\operatorname{Vol}}

\def \Re{\operatorname{Re}}

\def \Li{\operatorname{Li_2}}

\def \Arg{\operatorname{Arg}}
\def \det{\operatorname{det}}

\title{Geometry of fundamental shadow link complements and applications to the 1-loop conjecture}
\author{Tushar Pandey and Ka Ho Wong}
\date{}

\begin{document}

\maketitle

\begin{abstract}
We construct a geometric ideal triangulation for every fundamental shadow link complement and solve the gluing equation explicitly in terms of the logarithmic holonomies of the meridians of the link for any generic character in the distinguished component of the $\mathrm{PSL}(2;\CC)$-character variety of the link complement. As immediate applications, we obtain a new formula for the volume of a hyperideal tetrahedron in terms of its dihedral angles, and a formula for the volume of hyperbolic 3-manifolds obtained by doing Dehn-fillings to some of the boundary components of fundamental shadow link complements. 
Moreover, by using these ideal triangulations, we verify the 1-loop conjecture proposed by Dimofte and Garoufalidis for every fundamental shadow link complement. By using the result of Kalelkar-Schleimer-Segerman \cite{KSS}, we also prove the topological invariance of the 1-loop invariant and show that the 1-loop invariant satisfies a  surgery formula. As a result, we prove the 1-loop conjecture for manifolds obtained by doing sufficiently long Dehn-fillings on boundary components of any fundamental shadow link complement. This verifies the 1-loop conjecture for a large class of hyperbolic 3-manifolds.
\end{abstract}

\tableofcontents

\section{Introduction}

The family of fundamental shadow links was introduced and studied by Costantino and Thurston in \cite{CT}. Each fundamental shadow link is a hyperbolic link in a connected sum of copies of $\SS^2 \times \SS^1$. This family of links is universal in the sense that every compact oriented 3-manifold with toroidal or empty boundary can be obtained from a suitable fundamental shadow link complement by doing an integral Dehn-filling to some of the boundary components \cite{CT}. According to Thurston's hyperbolic Dehn surgery theorem \cite{T}, most Dehn surgeries on a hyperbolic 3-manifold are actually hyperbolic. Thus, studying the hyperbolic geometry of fundamental shadow link complements may help us to understand the geometry of a large class of hyperbolic 3-manifolds. Especially, this provides a possible approach to compute geometric invariants of hyperbolic 3-manifolds, such as complex volume and the adjoint twisted Reidemeister torsion, by first computing the invariants for fundamental shadow link complements and then investigating how the invariants behave under doing hyperbolic Dehn-fillings on fundamental shadow link complements. By using this approach, in \cite{WY3}, T. Yang and the second author discovered a new and explicit formula for the adjoint twisted Reidemeister torsion of a large class of hyperbolic 3-manifolds in terms of the determinant of the associated Gram matrices (see Section \ref{TRT} for more details). This formula also has an immediate application in the study of the asymptotics of quantum invariants \cite{WY4}.

\subsection{Geometry of fundamental shadow link complements}

In the first part of this paper, we study the geometry of fundamental shadow link complements by constructing a geometric ideal triangulation of each fundamental shadow link complement and solving the gluing equations explicitly for any representation of the fundamental group of the link complement into $\mathrm{PSL}(2;\CC)$ near the holonomy representation of the complete hyperbolic structure. This is achieved by solving the gluing equation on each $D$-block explicitly in terms of the holonomies around the six ideal vertices. The problem of solving the gluing equations is then reduced to the problem of solving a single quadratic equation. This allows us to write down the solution explicitly (See Sections \ref{GEDB} and \ref{GEDB2} for more details). By analyticity, the solution of the gluing equations can be extended to any generic character in the distinguished component of the $\mathrm{PSL}(2;\CC)$ character variety of the link complements.

We briefly recall the construction of fundamental shadow link complements, each of which admits as a 3-dimensional analogue of the pants decomposition of surfaces. Recall that to construct a closed surface with genus greater or equal to 2, as shown in Figure \ref{construction1}, we can first take the double of a hexagon along the three (green) sides to obtain a pair of pants with three (red) circle boundaries, and then glue the pair of pants together along the (red) boundary suitably to obtain the surface. By repeating a similar construction in one higher dimensional, we can construct a fundamental shadow link complement as shown in Figure \ref{construction3}. The building block of a fundamental shadow link complement is a truncated tetrahedron with six edges removed. To construct a fundamental shadow link complement, we first take the double of the truncated tetrahedron along the hexagonal faces to obtain a \emph{$D$-block} with boundary consisting of four (red) 3-puncture spheres and six (blue) cylinders around the six removed edges. Then we glue copies of $D$-blocks together along the 3-puncture spheres to obtain a fundamental shadow link complement. Note that the (blue) cylinders around the removed edges will be glued together to form the toroidal boundary of the manifold. Each fundamental shadow link complement is hyperbolic and hyperbolic cone structures along the meridians can be constructed concretely by gluing hyperideal tetrahedra (see Section \ref{DhypD} for more details). Well-known examples of fundamental shadow link complements include the Borromean rings complement and octahedral fully augmented link complements (\cite{JP}, \cite[Proposition 6.2]{WY1}). See also \cite{Ku} for more concrete examples of fundamental shadow links in $\SS^3$.

\begin{figure}[h]
\centering
\includegraphics[scale=0.18]{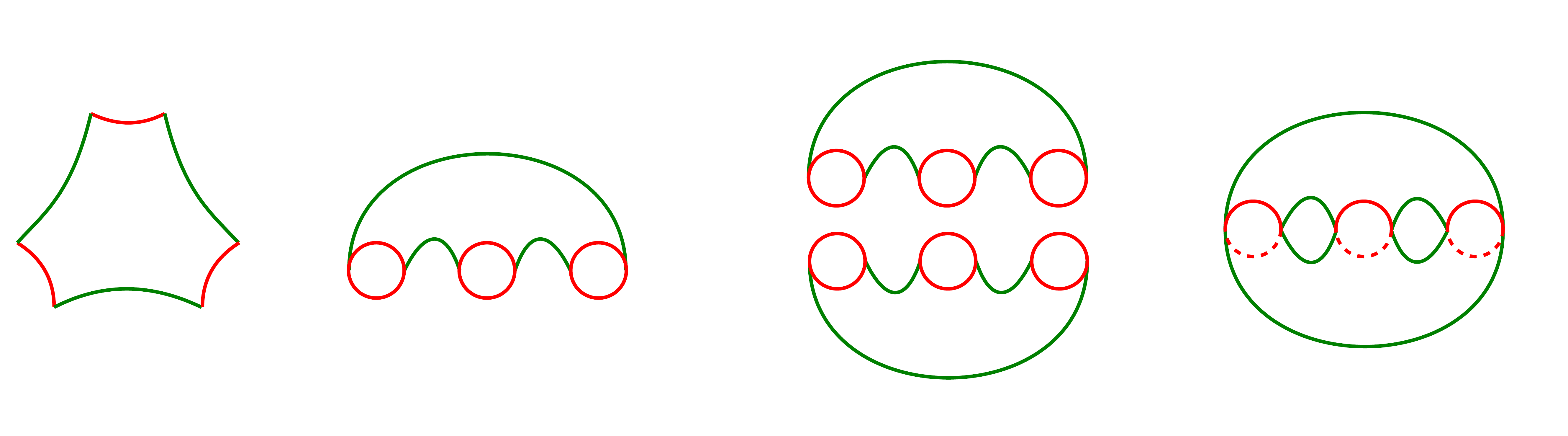}\qquad
\caption{Construction of a closed surface.}\label{construction1}
\end{figure}

\begin{figure}[h]
\centering
\includegraphics[scale=0.18]{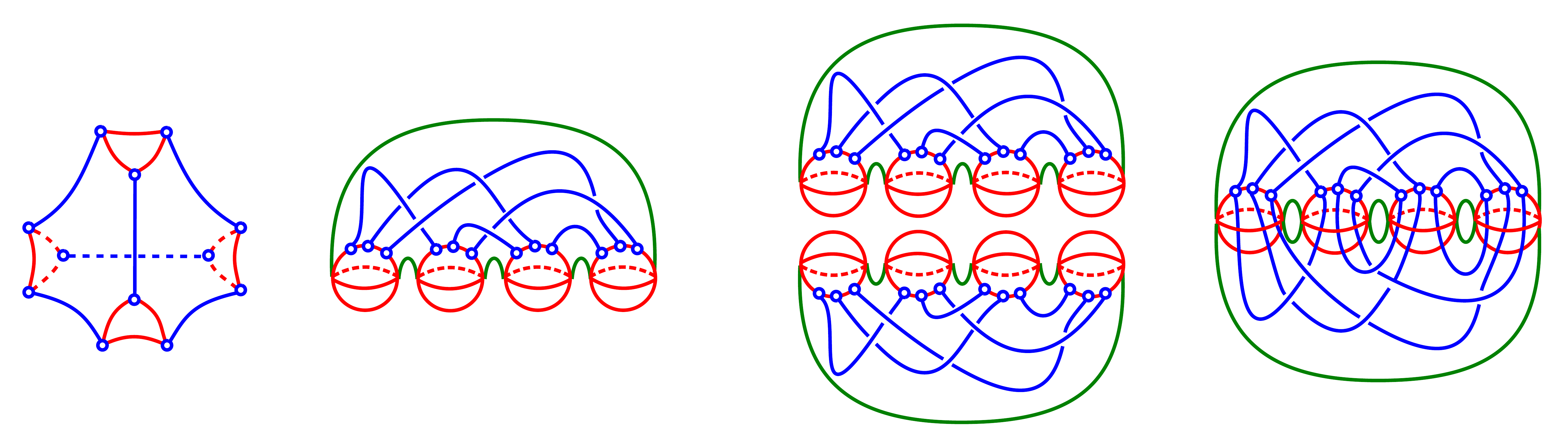}\qquad
\caption{Construction of a fundamental shadow link complement.}\label{construction3}
\end{figure}

\begin{figure}[h]
\centering
\includegraphics[scale=0.145]{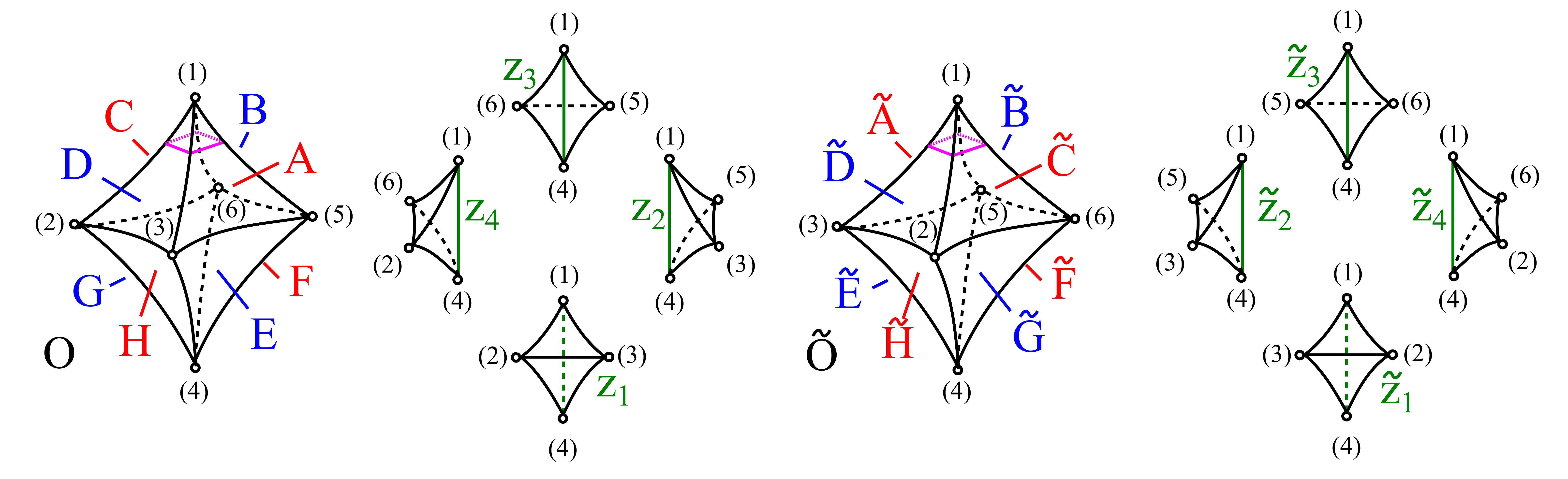}
\caption{Triangulation of each $D$-block into 8 ideal tetrahedra.}\label{idealoctatriintro}
\end{figure}

Each fundamental shadow link complement admits a natural ideal polyhedral decomposition obtained by shrinking the edges of the $D$-blocks into ideal vertices. In this decomposition, each $D$-block $\mathcal{D}$ consists of two ideal octahedra, which can be further decomposed into 8 ideal tetrahedra to give an ideal triangulation $\mathcal{T}_{\mathcal{D}}$ of $\mathcal{D}$ as shown in Figure \ref{idealoctatriintro}. This induces an ideal triangulation $\mathcal{T}=\{\mathcal{T}_{\mathcal{D}}\}$ of the fundamental shadow link complement. Note that the fundamental group of a $D$-block is a rank 3 free group. The $\mathrm{SL}(2;\CC)$-character variety of the free group with three generators is well-studied by Goldman in \cite{G1}. A key step to study the triangulation $\mathcal{T}$ is to understand the ``gluing variety" $\mathcal{Z}(\mathcal{D})$ of $\mathcal{T}_{\mathcal{D}}$ for each $D$-block $\mathcal{D}$ (see Remark \ref{gvDblock}). This leads to our first result about the ideal triangulation $\mathcal{T}$.

\begin{theorem}\label{idealtri}
The ideal triangulation $\mathcal{T}$ of the fundamental shadow link complement satisfies the following properties.
\begin{enumerate}
\item The triangulation is geometric at the complete structure (i.e. the shape parameters have positive imaginary parts) and each $D$-block is a union of 2 regular ideal octahedra which are triangulated into 8 ideal tetrahedra with dihedral angles $(\pi/4, \pi/4, \pi/2)$.
\item Near the complete structure, the 8 shape parameters on each $D$-block are uniquely determined by two edge equations and six holonomy equations and the solution can be explicitly written down in terms of the holonomies around the six ideal vertices. 
\item The solution of the gluing equation of the triangulation of the fundamental shadow link complement is given by a combination of  the solution on each $D$-block.
\end{enumerate}
\end{theorem}

Based on the construction of the layered solid torus studied by Jaco and Rubinstein in \cite{JR}, there are several interesting results about how an ideal triangulation of a manifold behaves under doing Dehn-fillings to some of the boundary components. In some cases, this construction allows us to show that some nice property of ideal triangulations is preserved under doing Dehn-fillings. For example, in \cite{GS}, Guéritaud and Schleimer study how the canonical decompositions of 3-manifolds behave under Dehn-fillings. In \cite{FHH}, Futer, Hamilton and Hoffman show that every hyperbolic 3-manifold has a finite cover that admits infinitely many geometric triangulations, with the additional property that every long Dehn-filling of one cusp in that finite cover also admits infinitely many geometric triangulations. Besides, 
Howie, Mathews and Purcell provide a new method to compute the $A$-polynomials of knots by studying how the Neumann-Zagier datum changes under Dehn-fillings \cite{HMP}. We hope that our result will provide insight to study the above problems from both theoretical and computational points of view.

As an application of Theorem \ref{idealtri}, we obtain a new formula for the volume of a hyperideal tetrahedron which is different from the well-known Murakami-Yano formula \cite{MY}. To present the formula, we define the volume of a $D$-block as follows. First, given $(u_1,\dots,u_6)\in \CC^6$, we let
$$ z^* = \frac{-B - \sqrt{B^2 - 4AC}}{2A} ,$$
where 
\begin{equation*}
\begin{split}
A&= - \frac{u_1}{u_4} - \frac{u_1 u_3}{u_2} - \frac{u_1}{u_2u_3} - \frac{u_1}{u_2^2 u_4} - \frac{u_5}{u_2} - \frac{u_6}{u_2u_4} - \frac{1}{u_2u_4u_6} - \frac{1}{u_2u_5},\\
B&= - u_1u_4 + \frac{u_1}{u_4} + \frac{u_4}{u_1} - \frac{1}{u_1u_4}
 + u_2u_5 + \frac{u_2}{u_5} + \frac{u_5}{u_2} + \frac{1}{u_2u_5}
 - u_3 u_6 - \frac{u_3}{u_6} - \frac{u_6}{u_3} - \frac{1}{u_3u_6}  ,\\
C&= - \frac{u_4}{u_1} - \frac{u_2}{u_1 u_3} - \frac{u_2u_3}{u_1} - \frac{u_2^2 u_4}{u_1} - \frac{u_2}{u_5} - \frac{u_2u_4}{u_6} - u_2u_4u_6 - u_2u_5.
\end{split}
\end{equation*}
Besides, recall that the volume of an ideal tetrahedron with shape parameter $z\in \CC \setminus\{0,1\}$ is given by
$$ D(z) = \mathrm{Im} \Li(z) + \mathrm{Arg}(z) \log |z|,$$
where $\Arg(z) \in (-\pi,\pi)$, $\Li : \CC\setminus (1,\infty) \to \CC$ is the dilogarithm function defined by
$$ \Li(z) = -\int_0^z \frac{\log(1-u)}{u}du$$
and the integral is along any path in $ \CC\setminus (1,\infty)$ going from $0$ to $z$.

We define the volume of a $D$-block $\mathcal{D}$ with logarithmic holonomy $(\mathrm{H}(m_1),\dots,\mathrm{H}(m_6))$ by
$$
\mathrm{Vol}_{\mathcal{D}} (\mathrm{H}(m_1),\dots,\mathrm{H}(m_6))
= \sum_{k=1}^4 \Big(  D(z_k^*) + D(\tilde{z}_k^*) \Big),
$$
where $D(z)$ is the Bloch-Wigner dilogarithm function, $u_l = e^{\frac{\mathrm{H}(m_l)}{2}}$ for $l=1,2,\dots,6$,
\begin{empheq}[left = \empheqlbrace]{equation*}
\begin{split}
z_1^*&=\frac{z^*-u_2^2}{z^*+u_2u_3u_4}, \quad
z_2^*=\frac{z^*u_1u_3u_5 - u_2u_3u_4}{z^* u_1u_3u_5 + u_1u_2u_4u_5},  \\
z_3^*&=\frac{z^*u_1u_6 - u_2u_4u_5u_6}{z^*u_1u_6+u_2}, \quad
z_4^*=\frac{z^* u_1 - u_1}{z^* u_1 + u_2 u_6} \\
\tilde z_1^*&=-\frac{z^*u_2 + u_2^2 u_3 u_4}{z^* u_3 u_4 - u_2^2 u_3 u_4 }, \quad
\tilde z_2^*=-\frac{z^*u_3 + u_2 u_4}{z^* u_1u_5- u_2u_4}, \\
\tilde z_3^*&=-\frac{z^* u_1u_4u_5u_6 + u_2u_4u_5}{z^* u_1 - u_2u_4u_5} ,  \quad
\tilde z_4^* =-\frac{z^* u_1 + u_2u_6}{z^* u_2u_6 - u_2u_6}.
\end{split}
\end{empheq}
\begin{figure}[h]
\centering
\includegraphics[scale=0.2]{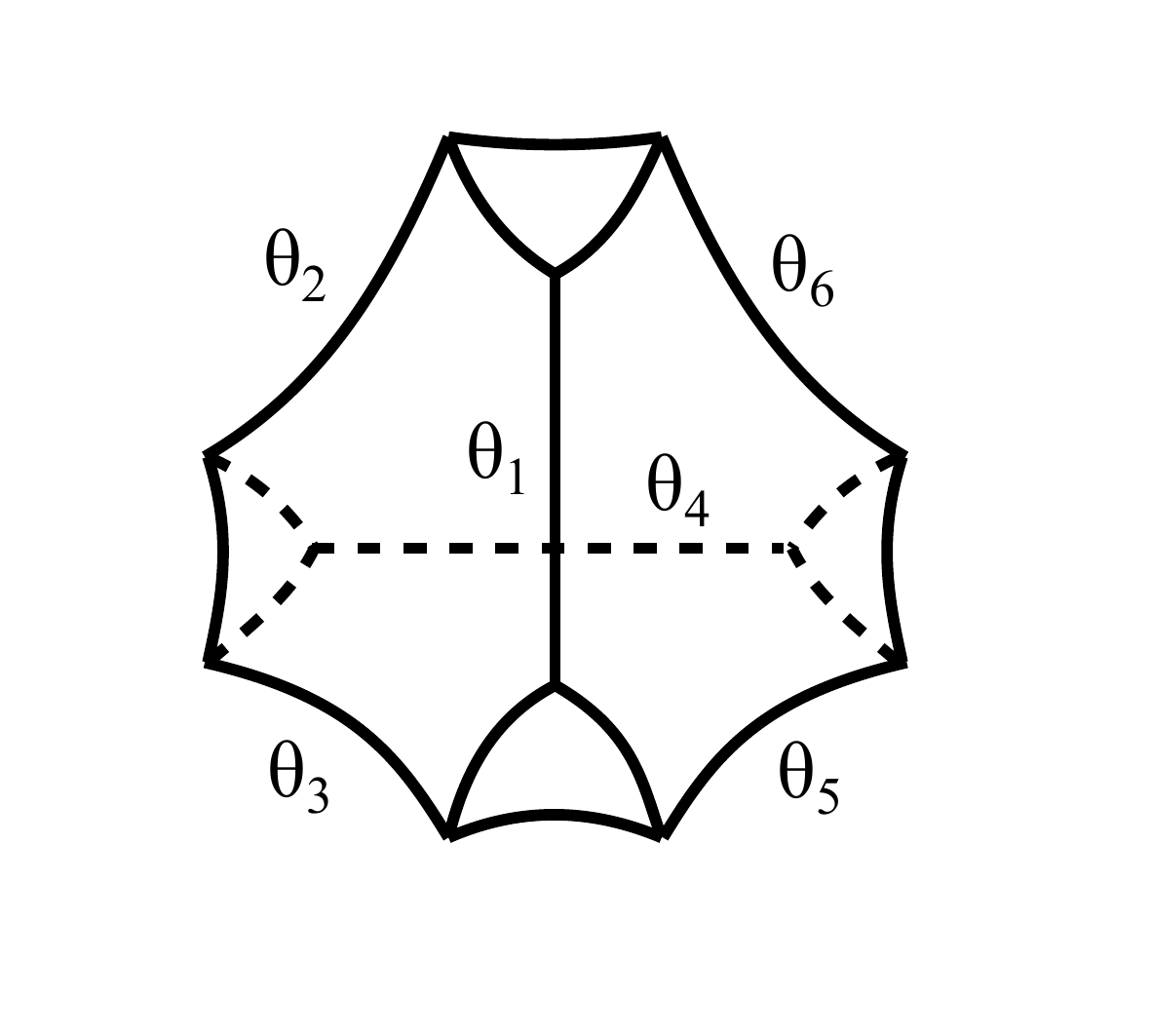}\qquad
\caption{A hyperideal tetrahedron.}\label{hyperideal}
\end{figure}

The volume of a $D$-block is the sum of the volume of the 8 ideal tetrahedra in Figure \ref{idealoctatriintro} with logarithmic holonomy $\mathrm{H}(m_l)$ at the $l$-th vertex for $l=1,2,\dots,6$ (See Section \ref{pfvolhypideal} for more details). 
Given a hyperideal tetrahedron with dihedral angles $(\theta_1,\dots,\theta_6)$ as shown in Figure \ref{hyperideal}. We have the following formula for the volume of the hyperideal tetrahedron.
\begin{theorem} \label{volhypideal} The volume of the hyperideal tetrahedron $\Delta_{(\theta_1,\dots,\theta_6)}$ with dihedral angles $(\theta_1,\dots,\theta_6)$ is given by
\begin{align*}
\mathrm{Vol}(\Delta_{(\theta_1,\dots,\theta_6)}\big)
= &\frac{1}{4}\Big( \mathrm{Vol}_{\mathcal{D}} \Big(2\theta_1\sqrt{-1},\dots,2\theta_6\sqrt{-1}\Big)+ \mathrm{Vol}_{\mathcal{D}} \Big(-2\theta_1\sqrt{-1},\dots,-2\theta_6\sqrt{-1}\Big) \Big).
\end{align*}

\end{theorem}

By using a similar idea, we obtain a formula for the volume of hyperbolic 3-manifolds obtained by doing Dehn-fillings to some of the boundary components of fundamental shadow link complements. Let $M=\#^{c+1}(\SS^2\times \SS^1)\setminus L_{\text{FSL}}$ be the complement of a fundamental shadow link $L_{\text{FSL}}$ with $n$ components $L_1,\dots,L_n$ obtained by gluing $c$ $D$-blocks. For $m$ with $0\leqslant m\leqslant n,$ let $\boldsymbol\mu=(\mu_1,\dots,\mu_m)$ be a system of simple closed curves on $\partial M$ such that $\mu_i\subset T_i$ for $i=1,2,\dots,m$.  Let $M_{\boldsymbol\mu}$ be the  $3$-manifold  obtained from $M$ by doing the Dehn-filling along $\boldsymbol\mu.$ Suppose $M_{\boldsymbol\mu}$ is hyperbolic with a holonomy representation $\rho_{\mu}$ of a possibly incomplete hyperbolic structure. Let $\rho$ be the restriction of $\rho_{\boldsymbol\mu}$ on $M.$ Assume that $\rho$ lies on the distinguished component of the $\mathrm{PSL}(2;\mathbb{C})$-character variety of $M$.
Let $( \mathrm{H}(m_1),\dots, \mathrm{H}(m_n))$ be the logarithmic holonomies of the system of meridians in $[\rho]$. For each $k\in\{1,\dots,c\},$ let $L_{k_1},\dots,L_{k_6}$ be the components of $L_{\text{FSL}}$ intersecting $\Delta_k$ and let $\mathrm{H}(m_{k_1}), \dots, \mathrm{H}(m_{k_6})$ be the logarithmic holonomies of $L_{k_1},\dots,L_{k_6}$.

\begin{theorem}\label{volmfd} The volume of the hyperbolic 3-manifold $M_{\boldsymbol\mu}$ with the holonomy representation $\rho_{\boldsymbol\mu}$ is given by
$$ \mathrm{Vol}(M_{\boldsymbol\mu}) 
= \sum_{k=1}^c  \mathrm{Vol}_{\mathcal{D}} \Big( \mathrm{H}(m_{k_1}), \dots, \mathrm{H}(m_{k_6})\Big).
$$
\end{theorem}

\begin{remark}\label{rmkvolmfd}
For the case where $m=n$, by the result in \cite{HK}, if we remove at most 114 simple closed curves on each boundary torus, then the Dehn-fillings along any remaining system of simple closed curves of $\partial M$ are sufficiently long and $\rho$ lies on the distinguished component of the $\mathrm{PSL}(2;\mathbb{C})$-character variety of $M$. In particular, Theorem \ref{volmfd} gives a formula for the volume of those closed, oriented hyperbolic $3$-manifolds.
\end{remark}

\subsection{1-loop conjecture of fundamental shadow link complements and their Dehn-fillings}

Given a hyperbolic 3-manifold $M$ with toroidal boundary $\partial M = T_1 \coprod \dots \coprod T_k$, let $\rho_0:\pi_1(M)\to\mathrm{PSL}(2;\CC)$ be the discrete faithful representation coming from the complete hyperbolic structure of $M$. Let $\mathrm X_0(M)$ be the distinguished component of the $\mathrm{PSL}(2;\CC)$-character variety of $M$, i.e. the irreducible component of the $\mathrm{PSL}(2;\CC)$-character variety containing $[\rho_0]$. Let $\mathrm X^{\text{irr}}_0(M)\subset\mathrm X_0(M)$ be the subset consisting of all irreducible characters. Let $\boldsymbol\alpha= (\alpha_1,\dots, \alpha_k)$ be a system of simple closed curves such that $\alpha_i \in \pi_1(T_i)$ for $i=1,2,\dots,k$. The \emph{adjoint twisted Reidemesiter torsion} $\mathbb T_{(M,\boldsymbol\alpha)}$, introduced by Porti in \cite{P}, is a non-zero rational function defined on a non-empty Zariski open subset $\mathcal{Z}_{\boldsymbol\alpha} \subset \mathrm X^{\text{irr}}_0(M)$ consisting of all $\boldsymbol\alpha$-regular representations (see Definition \ref{reg}).  Given any $[\rho]\in \mathcal{Z}_{\boldsymbol\alpha}$, $\mathbb T_{(M,\boldsymbol\alpha)}([\rho])$ is defined as the Reidemeister torsion of the cellular complex of the universal cover of $M$ twisted by the composition of the adjoint action with the representation $\rho$ (see Section \ref{TRT} for a review). This invariant is well-behaved under doing Dehn-fillings and is expected to appear as the ``1-loop" term in the asymptotic expansion formula of different quantum invariants, such as the colored Jones polynomials of knots at roots of unity. However, computing the torsion in general is a hard problem. 

In \cite{DG}, Dimofte and Garoufalidis proposed a conjectural formula to compute the adjoint twisted Reidemeister torsion of any hyperbolic 3-manifold with toroidal boundary in terms of the shape parameters of any ideal triangulation of that manifold. To be precise, given an ideal triangulation $\mathcal{T}$ of $M$, we let $\mathcal{P}_{\mathcal{T}}: \mathcal{V}(\mathcal{T}) \to \mathrm{X}(M)$ be the map form the gluing variety $\mathcal{V}(\mathcal{T})$ of $\mathcal{T}$ to the $\mathrm{PSL}(2;\CC)$-character variety $\mathrm{X}(M)$ of $M$, defined by sending a solution of the gluing equation to the character of the associated pseudo-developing map. For any $[\rho]\in \mathcal{Z}_{\boldsymbol\alpha}$, an ideal triangulation $\mathcal{T}$ of $M$ is called $\rho$-regular if $[\rho]$ is in the image of the pseudo-developing map $\mathcal{P}_{\mathcal{T}}$.  For the discrete faithful representation $\rho_0:\pi_1(M)\to\mathrm{PSL}(2;\CC)$, since each element in $\pi_1(\partial M)$ is mapped to a parabolic element which has only one fixed point on $\partial \mathbb{H}^3$, if $\mathcal{T}$ is $\rho_0$-regular, then there exists a unique $\mathbf{z_0}\in \mathcal{V}(\mathcal{T})$ such that $\mathcal{P}_{\mathcal{T}}(\mathbf{z_0}) = ([\rho_0])$. Thus, for any $\rho_0$-regular ideal triangulation, we can use $[\rho_0]$ and $[\mathbf{z_0}]$ interchangeably. For any $\rho_0$-regular ideal triangulation $\mathcal{T}$ of $M$, Dimofte and Garoufalidis define the \emph{1-loop invariant} $\tau(M,\boldsymbol\alpha, \rho_0, \mathcal{T}) $ of the ideally triangulated 3-manifold $M$ by using the shape parameter $\mathbf{z_0}$ and the Neumann-Zagier datum \cite{NZ}. The 1-loop conjecture proposed by Dimofte and Garoufalidis suggests that the 1-loop invariant coincides with the adjoint twisted Reidemeister torsion of the manifold. 

\begin{conjecture}[Conjecture 1.8, \cite{DG}]\label{1loopconjstatementori}
Let $M$ be a hyperbolic 3-manifold $M$ with toroidal boundary $\partial M = T_1 \coprod \dots \coprod T_k$ and let $\boldsymbol\alpha$ be a system of simple closed curves of $\partial M$. Let $[\rho_0]$ be the character of the discrete faithful representation and let $\mathcal{T}$ be a $\rho_0$-regular ideal triangulation. Then
we have
$$
\tau(M,\boldsymbol\alpha, \rho_0, \mathcal{T}) 
= \pm \mathbb T_{(M,\boldsymbol\alpha)}([\rho_0]).
$$
\end{conjecture}
With respect to certain ideal triangulations of the manifolds, Conjecture \ref{1loopconjstatementori} has been verified for the figure eight knot complement \cite{DG}, the sister manifold of the figure eight knot complement \cite{S2}, all hyperbolic once-punctured torus bundle over $\SS^1$ \cite{Yo} and all fibered 3-manifolds with toroidal boundary \cite{GY}. 

As discussed in \cite[Section 4]{DG}, it is natural to generalize Conjecture \ref{1loopconjstatementori} to other characters in $\mathrm X_0(M)$. We apply geometric techniques to study the 1-loop conjecture for characters in the distinguished component of the $\mathrm{PSL}(2;\CC)$-character variety of $M$. Generically, each element in $\pi_1(\partial M)$ is mapped to a loxodromic or elliptic elements with 2 fixed points on $\partial \mathbb{H}^3$, and therefore $\mathcal{P}_{\mathcal{T}}$ is a $2^k$ to $1$ map, where $k$ is the number of boundary component of $M$. This leads us to regard the 1-loop invariant as a complex-valued function defined on the gluing variety $\mathcal{V}_{\mathcal{T}}$ and formulate the 1-loop conjecture for general cases as follows. For every $\rho_0$-regular triangulation with $\mathcal{P}_{\mathcal{T}}(\mathbf{z_0}) = [\rho_0]$, in Proposition \ref{z0smpt} we show that $\mathbf{z_0}$ is a always a smooth point of $\mathcal{V}_{\mathcal{T}}$. We let $\mathcal{V}_0({\mathcal{T}})\subset\mathcal{V}({\mathcal{T}})$ be the irreducible component containing $\mathbf{z_0}$. 

\begin{conjecture}\label{1loopconjstatement}
Let $M$ be a hyperbolic 3-manifold $M$ with toroidal boundary $\partial M = T_1 \coprod \dots \coprod T_k$ and let $\boldsymbol\alpha$ be a system of simple closed curves of $\partial M$. Let $\mathcal{T}$ be a $\rho_0$-regular ideal triangulation of $M$. For any $\mathbf{z} \in \mathcal{V}_0(\mathcal{T})$ with ${\mathcal{P}}_{\mathcal{T}}(\mathbf{z}) = [\rho_{\mathbf{z}}]$, if $[\rho_{\mathbf{z}}]\in \mathcal{Z}_{\boldsymbol\alpha}$, then
$$
\tau(M,\boldsymbol\alpha, \mathbf z, \mathcal{T}) 
= \pm \mathbb T_{(M,\boldsymbol\alpha)}([\rho_{\mathbf{z}}]).
$$
\end{conjecture}
\begin{remark}
Since $[\rho_0] \in \mathcal{Z}_{\boldsymbol\alpha}$ for any $\boldsymbol\alpha$, Conjecture \ref{1loopconjstatement} implies Conjecture \ref{1loopconjstatementori}. 
\end{remark}
\begin{remark}
It is known that $\mathbb T_{(M,\boldsymbol\alpha)}$ can be extended to be a rational function defined on $\mathrm X_0^{\text{irr}}(M)$ by defining $\mathbb T_{(M,\boldsymbol\alpha)}([\rho]) = 0$ for all $[\rho] \in X_0^{\text{irr}}(M) \setminus \mathcal{Z}_{\boldsymbol\alpha} $ \cite[Theorem 4.1]{P}. In this paper, we only focus on the support $\mathcal{Z}_{\boldsymbol\alpha}$ of $\mathbb T_{(M,\boldsymbol\alpha)}$. In particular, all characters considered in this paper are smooth points of $\mathrm X_0^{\text{irr}}(M)$. 
\end{remark}
In the second part of this paper, we first verify the 1-loop conjecture of fundamental shadow link complements with respect to the ideal triangulation $\mathcal{T}$ in Theorem \ref{idealtri}. 

\begin{theorem}\label{mainthm}
Let $M$ be a fundamental shadow link complement and let $\boldsymbol\alpha$ be a system of simple closed curves of $\partial M$. Let $\mathcal{T}$ be the ideal triangulation of $M$ described in Theorem \ref{idealtri}.  Let $\mathbf{z} \in \mathcal{V}_0(\mathcal{T})$ and let ${\mathcal{P}}_{\mathcal{T}}(\mathbf{z}) = [\rho_{\mathbf{z}}]$. If $[\rho_{\mathbf{z}}]\in \mathcal{Z}_{\boldsymbol\alpha}$, then
$$
\tau(M,\boldsymbol\alpha, \mathbf{z}, \mathcal{T}) 
= \pm \mathbb T_{(M,\boldsymbol\alpha)}([\rho_{\mathbf{z}}]).
$$
\end{theorem}

Next, we study the dependence of $\tau(M,\boldsymbol\alpha, \mathbf{z}, \mathcal{T})$ on the $(\mathbf{z}, \mathcal{T})$. Associated to an ideal triangulation $\mathcal{T}$ and a point $ \mathbf{z} \in \mathcal{V}_{\mathcal{T}}$, there is a \emph{pseudo-developing map at infinity} $\mathcal{P}_{\mathcal{T}}^\infty(\mathbf{z})$ that sends the connected components of the universal covering of $M$ to $\partial \mathbb{H}^3$, defined up to conjugation. This labeling tells us how to straigthen out the tetrahedra in the triangulation.  Generically, for each $[\rho]$, there are $2^k$ many labelings subordinated to $[\rho]$, corresponding to the $2^k$ distinct choices of shape parameters. Given two ideal triangulation $\mathcal{T}_1, \mathcal{T}_2$ with $\mathbf{z_1}\in \mathcal{V}_0(\mathcal{T}_1)$ and $\mathbf{z_2}\in \mathcal{V}_0(\mathcal{T}_2)$, we say that two pairs $(\mathcal{T}_1,  \mathbf{z_1})$ and $(\mathcal{T}_2, \mathbf{z_2})$  \emph{boundary equivalent} if ${\mathcal{P}}_{\mathcal{T}_1}(\mathbf{z_1}) = {\mathcal{P}}_{\mathcal{T}_2}(\mathbf{z_2})$ and  $\mathcal{P}^\infty_{\mathcal{T}_1}(\mathbf{z_1}) =  \mathcal{P}^\infty_{\mathcal{T}_2}(\mathbf{z_2})$ up to conjugation. 

In \cite[Corollary 1.2, 1.5 and Proposition A.1]{KSS}, Kalelkar, Schleimer and Segerman prove that there exists a Zariski open subset $U_M$ of $\mathrm X_0(M)$ containing the character $[\rho_0]$ such that for any $[\rho]\in U_M$, 
\begin{enumerate}
\item the set of $\rho$-regular ideal triangulations is non-empty, and
\item any two boundary equivalent pairs are connected by a finite sequence of boundary equivalent pairs through 0-2, 2-0, 2-3 and 3-2 Pachner moves.
\end{enumerate}
It is known that the 1-loop invariant remains unchanged under any 2-3 and 3-2 move connecting two boundary equivalent $\rho$-regular ideal triangulations \cite[Theorem 1.4, 4.1]{DG}. In Proposition \ref{inv02move}, we prove that the 1-loop invariant also remains unchanged under any 0-2 and 2-0 move connecting two boundary equivalent $\rho$-regular ideal triangulations. Altogether, we have the following  invariance of the 1-loop invariant. 

\begin{theorem}\label{1loopreallyinv} Let $M$ be a hyperbolic 3-manifold with toroidal boundary and let $\boldsymbol\alpha$ be a system of simple closed curves of $\partial M$. Let $\mathcal{T}_1$ and $\mathcal{T}_2$ be two $\rho_0$-regular ideal triangulations of $M$. Then for any $[\rho] \in \mathcal{Z}_{\boldsymbol\alpha}$, $\mathbf{z_1}\in \mathcal{V}_0(\mathcal{T}_1)$ and $\mathbf{z_2}\in \mathcal{V}_0(\mathcal{T}_2)$, if $(\mathbf{z_1}, \mathcal{T}_1)$ and $(\mathbf{z_2}, \mathcal{T}_2)$ are boundary equivalent, we have
$$
\tau(M,\boldsymbol\alpha, \mathbf{z_1}, \mathcal{T}_1) =\tau(M,\boldsymbol\alpha, \mathbf{z_2}, \mathcal{T}_2).
$$
In particular, $\tau(M,\boldsymbol\alpha, \rho_0, \mathcal{T})$ is a topological invariant. 
\end{theorem}
\begin{remark}\label{rmkinde}
Without the assumption that $\mathcal{P}^\infty_{\mathcal{T}_1}(\mathbf{z_1}) =  \mathcal{P}^\infty_{\mathcal{T}_2}(\mathbf{z_2})$ up to conjugation, in general it is still unknown whether the 1-loop invariant defined by using $(\mathcal{T}_1,\mathbf{z_1})$ is the same as that defined by using $(\mathcal{T}_2, \mathbf{z_2})$. 
\end{remark}
\begin{corollary}\label{1loopreallyinvcor}
If Conjecture \ref{1loopconjstatement} is true for some $\rho_0$-regular ideal triangulation, then it is true for all $\rho_0$-regular ideal triangulation.
\end{corollary}
Combining Theorem \ref{mainthm} and Corollary \ref{1loopreallyinvcor}, we have the following result about the 1-loop conjectures for all fundamental shadow link complements.
\begin{corollary}\label{1loopconjFSLall}
Conjecture \ref{1loopconjstatementori} and \ref{1loopconjstatement} hold for all fundamental shadow link complements.
\end{corollary}

Next, we study how the 1-loop invariants behave under Dehn-fillings. Theorem \ref{1loopreallyinv} allows us to investigate the 1-loop invariants by choosing suitable triangulations that simplify the computation.
The following result from \cite{KSS} provides ideal triangulations for hyperbolic 3-manifolds that are well-behaved under Dehn-fillings (see Proposition B.1 and Remark B.2 in \cite{KSS} for a more general result).
\begin{proposition}[Proposition B.1, \cite{KSS}]\label{niceidealintro}
Let $M$ be a hyperbolic 3-manifold with $k\geq 2$ toroidal boundary $T_1,\dots,T_k$. Given $l \in \{1,2,\dots, k-1\}$, let $(\alpha_1,\alpha_2,\dots,\alpha_l)$ be a system of simple closed curves on $T_1,\dots, T_l$. Let $M'$ be the manifold obtained by doing Dehn-fillings on the boundaries $T_1,\dots, T_l$ of $M$ that homotopically ``kill" the curves $(\alpha_1,\alpha_2,\dots,\alpha_l)$. Let $\rho' : \pi_1(M') \to \mathrm{PSL}(2;\CC)$ be a representation and let $\rho = \rho' |_{\pi_1(M)}$. Assume that $\rho$ is sufficiently close to the discrete faithful representation of $M$. Then there exists an ideal triangulation $\widehat{\mathcal{T}}$ of $M$ such that the following hold.
\begin{enumerate}
\item For $j = 1, . . . , k-1$, the cusp corresponding to $T_j$  meets exactly two ideal tetrahedra, $\Delta_1^j$ and $\Delta_2^j$. Each of these tetrahedra meets $T_j$ in exactly one ideal vertex.
\item On each $\{T_i\}_{i=1}^{l}$, if we remove the two ideal tetrahedra $\Delta_1^i,\Delta_2^i$ and fold the once-punctured torus along certain diagonal as shown in Figure \ref{niceideal1}, we obtain an ideal triangulation $\widehat{\mathcal{T}}'$ of the manifold $M'$.
\item There exists a choice of generators for $\mathrm{H}_1(T_k;\mathbb{Z})$, represented by curves $m_k$ and $l_k$, such that $m_k$ and $l_k$ meet the cusp triangulation inherited from $\widehat{\mathcal{T}}$ in a sequence of arcs cutting off single vertices of triangles, without backtracking, and such that $m_k$ and $l_k$ are disjoint from the tetrahedra $\Delta_1^i$ and $\Delta_2^i$, for all $i = 1, . . . , k-1$.
\item There exists $\mathbf{\widehat{z}} \in \mathcal{V}_{\widehat{\mathcal{T}}}$ and $\mathbf{\widehat{z}'} \in \mathcal{V}_{\widehat{\mathcal{T}}'}$ such that
$\mathcal{P}_{\widehat{\mathcal{T}}}(\mathbf{\widehat{z}}) = [\rho]$ and $\mathcal{P}_{\widehat{\mathcal{T}}'}(\mathbf{\widehat{z}'}) = [\rho']$. Moreover, $\mathbf{\widehat{z}'}$ is obtained from $\mathbf{\widehat{z}}$ by removing the shape parameters associated to the tetrahedra $\{\Delta_1^i, \Delta_2^i\}_{i=1}^l$. 
\end{enumerate} 
\end{proposition}
\begin{figure}[h]
\centering
\includegraphics[scale=0.2]{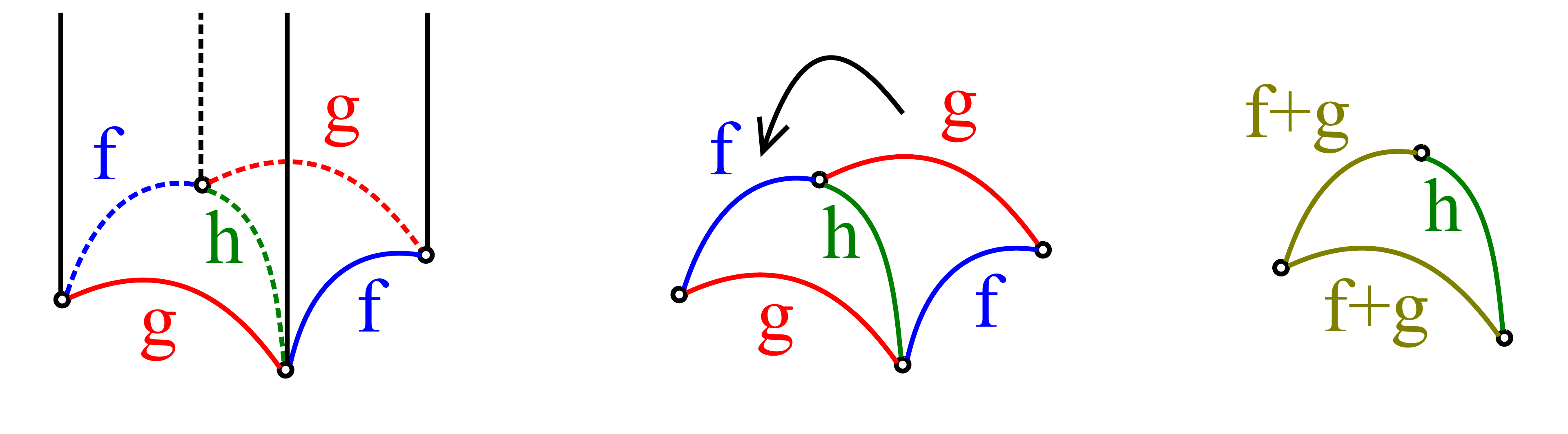}
\caption{In Proposition \ref{niceidealintro}, the cusp $T_j$ for $j=1,\dots,k-1$ meet exactly two ideal tetrahedra with the edges of the base triangles labelled by $f,g$ and $h$ (left). To obtain a triangulation of the Dehn-filled manifold, we first remove the two tetrahedra and then fold the two ideal triangles along the edge labelled by $h$ (middle). After the gluing, the four edges of the parallelogram are identified and labelled by $f+g$ (right). }\label{niceideal1}
\end{figure}
By using the triangulations described in Proposition \ref{niceidealintro}, we obtain the following result.
\begin{theorem}\label{1loopsurg} Let $\boldsymbol \alpha = (\alpha_1,\dots,\alpha_l,\alpha_{l+1},\dots,\alpha_k)$ and $\boldsymbol\alpha' = (\alpha_{l+1},\dots,\alpha_k)$ be systems of simple closed curves on $\partial M$ and $\partial M'$ respectively. 
Under the assumption of Proposition \ref{niceidealintro}, we have
$$
\tau(M',\boldsymbol \alpha',\mathbf{\widehat{z}'},\widehat{\mathcal{T}}')
= 
\pm \tau(M,\boldsymbol \alpha,\mathbf{\widehat{z}},\widehat{\mathcal{T}})\frac{1}{\prod_{i=1}^l 4\sinh^2 \frac{\mathrm{H}(\gamma_i)}{2}},
$$
where $\mathrm{H}(\gamma_i)$ is the holonomy of the core curve $\gamma_i$ of the i-th Dehn-filled solid torus.
\end{theorem}
Note that the behavior of the 1-loop invariants under Dehn-fillings described in Theorem \ref{1loopsurg} matches with the surgery formula satisfied by the adjoint twisted Reidemeister torsion \cite[Theorem 4.1]{P}. As a consequence of Theorem \ref{mainthm} and \ref{1loopsurg}, we prove the 1-loop conjecture for a large class of hyperbolic 3-manifolds with toroidal boundary. 
\begin{corollary}\label{corsufflong}
Conjecture \ref{1loopconjstatementori} and \ref{1loopconjstatement} hold for any hyperbolic 3-manifold $M$ with toroidal boundary obtained by doing sufficiently long Dehn-fillings on the boundary components of some fundamental shadow link complement. 
\end{corollary}
\begin{remark}
For any fundamental shadow link complement, by \cite[Remark B.2]{KSS} and the result in \cite{HK}, if we remove at most 114 simple closed curves on each boundary torus of $M$, then the Dehn-fillings along any remaining system of simple closed curves are sufficiently long. In this case, Corollary \ref{corsufflong} applies.
\end{remark}
\begin{remark}
In contrast to Remark \ref{rmkinde}, Corollary \ref{1loopconjFSLall} and \ref{corsufflong} imply that for all fundamental shadow link complments and most of their Dehn-fillings, the 1-loop invariant depends on the ideal triangulation but not on the pseudo-developing map at infinity.
\end{remark}

As an intermediate step in the proof of Theorem \ref{mainthm}, we prove that the 1-loop invariant satisfies the same change of curve formula as the one satisfied by the adjoint twisted Reidemeister torsion. In particular, once the 1-loop conjecture is true for a system of simple closed curves, the conjecture is true for any system of simple closed curve (See \cite{S1} for another proof of the change of curve formula for manifolds with one torus boundary).
\begin{theorem}\label{COC1loop} Let $\mathcal{T}$ be an ideal triangulation and $\mathbf{z} \in \mathcal{V}_{\mathcal{T}}$. Let $\boldsymbol\alpha, \boldsymbol\alpha'$ are two systems of simple closed curves on $\partial M$. Assume that $\tau(M,\boldsymbol\alpha, \mathbf{z}, \mathcal{T})$ is non-zero. Then
$$
\tau(M,\boldsymbol\alpha', \mathbf{z}, \mathcal{T}) 
= \pm\det\bigg( \frac{\partial \mathrm H(\alpha'_i)}{\partial \mathrm H(\alpha_j)}\bigg)_{ij}\tau(M,\boldsymbol\alpha, \mathbf{z}, \mathcal{T}) .
$$
\end{theorem}
\begin{remark}
Since conjecturally the 1-loop invariant is equal to the adjoint twisted Reidemeister torsion, which is non-zero by definition, the requirement in Theorem \ref{COC1loop} that $\tau(M,\boldsymbol\alpha, \mathbf{z}, \mathcal{T})$ being non-zero is not an additional assumption in the context of 1-loop conjecture.
\end{remark}

\subsection*{Outline of this paper}
In Section \ref{prelim}, we recall the construction of fundamental shadow links, the definition of the adjoint twisted Reidemeister torsion and the definition of the 1-loop invariant. Then in Section \ref{triFSL}, we study the geometry of fundamental shadow link complements. More precisely, we construct an ideal triangulation of a $D$-block in Section \ref{GEDB} and solve the gluing equation explicitly in terms of the holonomy around the ideal vertices. In Section \ref{GEDB2}, by using the ideal triangulation on each $D$-block, we construct an ideal triangulation on every fundamental shadow link complement and solve the gluing equations explicitly. We prove Theorem \ref{volhypideal} and \ref{volmfd} in Section \ref{pfvolhypideal}. The key idea is to glue hyperideal tetrahedra together to get a fundamental shadow link complement with hyperbolic cone structure and then compute the volume of the link complement by summing up the volume of the ideal tetrahedra in the triangulation. We prove Theorems \ref{mainthm} and \ref{COC1loop} in Section \ref{1loopinv}. The idea is to do the computation on each $D$-block (Proposition \ref{1loopDB}) and then show that the torsion of the manifold is a product of the contribution from the $D$-blocks (Section \ref{pfmainthm} and Lemma \ref{detL}). In Section \ref{Sinv02}, we prove Theorem \ref{1loopreallyinv} and Corollary \ref{1loopconjFSLall}. Finally, in Section \ref{surgforS}, we prove Theorem \ref{1loopsurg} and Theorem \ref{corsufflong}. 

\section*{Acknowledgement}
The authors would like to express their gratitude to their supervisor Tian Yang for his guidance. The authors would also like to thank Giulio Belletti, Francis Bonahon, Effie Kalfagianni, Feng Luo, Joan Porti, Jessica Purcell and Seokbeom Yoon for valuable discussions. The authors are grateful to Saul Schleimer and Henry Segerman for carefully explaining their results in \cite{KSS}. The first author was supported by the NSF grants DMS-1812008 and DMS-2203334 (PI: Tian Yang). The second author was partially supported by the Hagler Institute HEEP Fellowship from the Hagler Institute for Advanced Study at Texas A\&M University.

\section{Preliminary}\label{prelim}

\subsection{Fundamental shadow links}\label{fsl}
In this section we recall the construction and basic properties of the fundamental shadow links. 
\subsubsection{Topological construction}
\begin{figure}[h]
\centering
\includegraphics[scale=0.2]{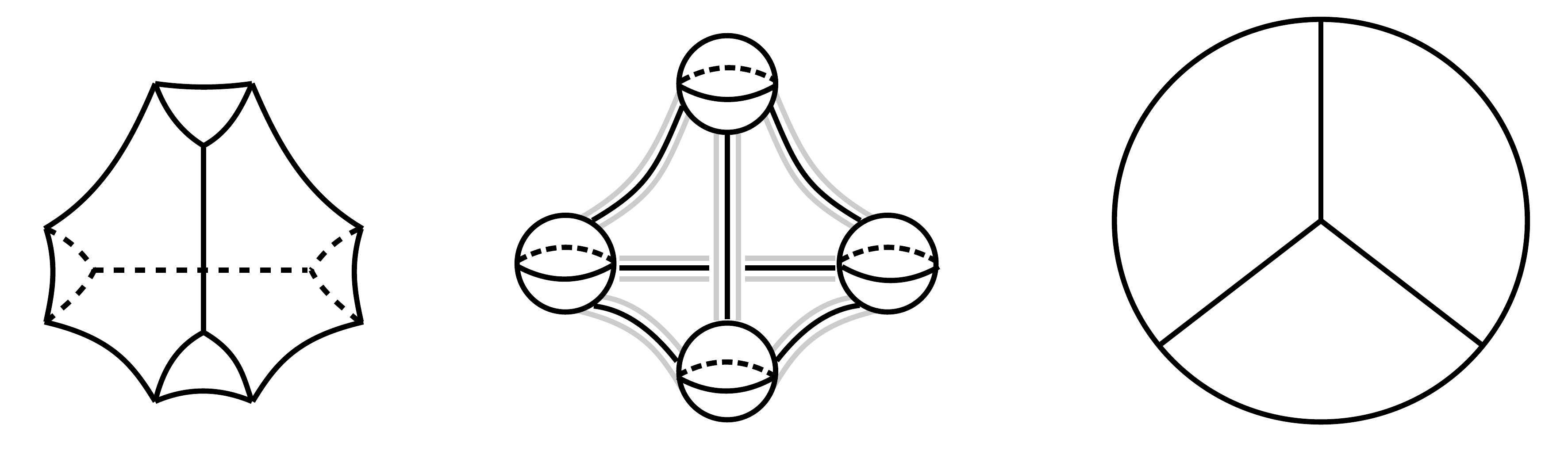}
\caption{A truncated tetrahedron (left), a $D$-block (middle) and a tetrahedron graph (right).}
\label{bb}
\end{figure}

The building blocks for the fundamental shadow links are truncated tetrahedra as shown in Figure \ref{bb}. An \emph{edge} of a truncated tetrahedron is the intersection of two hexagonal faces. If we glue two truncated tetrahedra together along the hexagonal faces via the identity map, we obtain the double of the truncated tetrahedra whose boundary consists of four spheres, each of them has three marked points corresponding to the edges incident to the triangular faces of the truncated tetrahedra. To construct a fundamental shadow link, we pick $c$ copies of the double of the truncated tetrahedra constructed above and glue them together via orientation reversing homeomorphisms along the boundaries that send marked points to marked points. The edges of the truncated tetrahedra will be glued together to form a link $L_{\text{FSL}}$ inside the ambient manifold $M_c=\#^{c+1}(S^2\times S^1).$ We call a link obtained this way a \emph{fundamental shadow link}, and its complement in $M_c$ a \emph{fundamental shadow link complement}. 
The family of the fundamental shadow link complements satisfies the following universal property.

\begin{theorem}[\cite{CT}]\label{CT} Any compact oriented $3$-manifold with toroidal or empty boundary can be obtained from a suitable fundamental shadow link complement by doing an integral Dehn-filling to some of the boundary components.
\end{theorem}

\subsubsection{$D$-blocks and hyperbolic $D$-blocks}\label{DhypD}
In the language of \cite{CT}, a \emph{$D$-block} (Figure \ref{bb}, middle) is the manifold obtained by first gluing two truncated tetrahedra together along the hexagonal faces via the identity map and then removing the six edges of the tetrahedra. Topologically, a $D$-block is homeomorphic to the complement of the tetrahedron graph in $\SS^3$ (Figure \ref{bb}, right). Alternatively, a fundamental shadow link complement can be constructed by gluing $c$ copies of $D$-blocks together via orientation reversing homeomorphisms along the 3-puncture spheres.

To see that all fundamental shadow link complements are hyperbolic, note that topologically, by shrinking all the edges of the tetrahedra into ideal vertices, we can decompose the link complements into ideal octahedra. Geometrically, we can put a complete hyperbolic structure on the link complement by putting a hyperbolic structure on the ideal octahedra such that each of them becomes a regular ideal hyperbolic octahedron. 
More generally, for a fundamental shadow link $L_{\text{FSL}}$ with $n$ components, a hyperbolic cone structure on $M_c$ with singular locus $L_{\text{FSL}}$ and with sufficiently small cone angles can be constructed as follows. Recall that each component of $L_{\text{FSL}}$ is obtained by gluing the edges of the $D$-blocks together along the end points. For each edge $e^i_k$ of a $D$-block $\mathcal{D}_i$, where $k=1,\dots,6$ and $i=1,\dots,c$, we let $\theta^i_k$ be half of the cone angle of the link component that the edge belongs to. Then we can construct a hyperbolic $D$-blocks with cone angles $(2\theta^i_1,\dots,2\theta^i_6)$ by gluing two hyperideal tetrahedra with dihedral angles $(\theta^i_1,\dots,\theta^i_6)$ along the hyperbolic hexagonal faces via the identity map. The existence of such hyperideal tetrahedron is guaranteed when the cone angles are sufficiently small. We call the $D$-block with such a hyperbolic structure a \emph{hyperbolic $D$-block}. Then the desired hyperbolic cone structure on $M_c$ with singularity along $L_{\text{FSL}}$ can be constructed by gluing the hyperbolic $D$-blocks together via orientation reversing isometries between the hyperbolic 3-puncture spheres.

\subsection{Adjoint twisted Reidemeister torsion}\label{TRT}
\subsubsection{Definition}
Let $\mathrm C_*$ be a finite chain complex 
$$0\to \mathrm C_d\xrightarrow{\partial}\mathrm C_{d-1}\xrightarrow{\partial}\cdots\xrightarrow{\partial}\mathrm C_1\xrightarrow{\partial} \mathrm C_0\to 0$$
of $\mathbb C$-vector spaces, and for each $\mathrm C_k$ choose a basis $\mathbf c_k.$ Let $\mathrm H_*$ be the homology of $\mathrm C_*,$ and for each $\mathrm H_k$ choose a basis $\mathbf h_k$ and a lift $\widetilde{\mathbf h}_k\subset \mathrm C_k$ of $\mathbf h_k.$ We also choose a basis $\mathbf b_k$ for each image $\partial (\mathrm C_{k+1})$ and a lift $\widetilde{\mathbf b}_k\subset \mathrm C_{k+1}$ of $\mathbf b_k.$ Then $\mathbf b_k\sqcup \widetilde{\mathbf b}_{k-1}\sqcup \widetilde{\mathbf h}_k$ form a basis of $\mathrm C_k.$ Let $[\mathbf b_k\sqcup \widetilde{\mathbf b}_{k-1}\sqcup \widetilde{\mathbf h}_k;\mathbf c_k]$ be the determinant of the transition matrix from the standard basis $\mathbf c_k$ to the new basis $\mathbf b_k\sqcup \widetilde{\mathbf b}_{k-1}\sqcup \widetilde{\mathbf h}_k.$
 Then the Reidemeister torsion of the chain complex $\mathrm C_*$ with the chosen bases $\mathbf c_*$ and $\mathbf h_*$ is defined by 
\begin{equation*}
\mathrm{Tor}(\mathrm C_*, \{\mathbf c_k\}, \{\mathbf h_k\})=\pm\prod_{k=0}^d[\mathbf b_k\sqcup \widetilde{\mathbf b}_{k-1}\sqcup \widetilde{\mathbf h}_k;\mathrm c_k]^{(-1)^{k+1}}.
\end{equation*}
It is easy to check that $\mathrm{Tor}(\mathrm C_*, \{\mathbf c_k\}, \{\mathbf h_k\})$ depends only on the choice of $\{\mathbf c_k\}$ and $\{\mathbf h_k\},$ and does not depend on the choices of $\{\mathbf b_k\}$ and the lifts $ \{\widetilde{\mathbf b}_k\}$ and $\{\widetilde{\mathbf h}_k\}.$

We recall the twisted Reidemeister torsion of a CW-complex following the conventions in \cite{P2}. Let $K$ be a finite CW-complex and let $\rho:\pi_1(M)\to\mathrm{SL}(N;\mathbb C)$ be a representation of its fundamental group. Consider the twisted chain complex 
$$\mathrm C_*(K;\rho)= \mathbb C^N\otimes_\rho \mathrm C_*(\widetilde K;\mathbb Z)$$
where $\mathrm C_*(\widetilde K;\mathbb Z)$ is the simplicial complex of the universal covering of $K$ and $\otimes_\rho$ means the tensor product over $\mathbb Z$ modulo the relation
$$\mathbf v\otimes( \gamma\cdot\mathbf c)=\Big(\rho(\gamma)^T\cdot\mathbf v\Big)\otimes \mathbf c,$$
where $T$ is the transpose, $\mathbf v\in\mathbb C^N,$ $\gamma\in\pi_1(K)$ and $\mathbf c\in\mathrm C_*(\widetilde K;\mathbb Z).$ The boundary operator on $\mathrm C_*(K;\rho)$ is defined by
$$\partial(\mathbf v\otimes \mathbf c)=\mathbf v\otimes \partial(\mathbf c)$$
for $\mathbf v\in\mathbb C^N$ and $\mathbf c\in\mathrm C_*(\widetilde K;\mathbb Z).$ Let $\{\mathbf e_1,\dots,\mathbf e_N\}$ be the standard basis of $\mathbb C^N,$ and let $\{c_1^k,\dots,c_{d^k}^k\}$ denote the set of $k$-cells of $K.$ Then we call
$$\mathbf c_k=\big\{ \mathbf e_i\otimes c_j^k\ \big|\ i\in\{1,\dots,N\}, j\in\{1,\dots,d^k\}\big\}$$
the standard basis of $\mathrm C_k(K;\rho).$ Let $\mathrm H_*(K;\rho)$ be the homology of the chain complex $\mathrm C_*(K;\rho)$ and let $\mathbf h_k$ be a basis of $\mathrm H_k(K;\rho).$ Then the Reidemeister torsion of $K$ twisted by $\rho$ with basis $\{\mathbf h_k\}$ is 
$$\mathrm{Tor}(K, \{\mathbf h_k\}; \rho)=\mathrm{Tor}(\mathrm C_*(K;\rho),\{\mathbf c_k\}, \{\mathbf h_k\}).$$
By \cite{P}, $\mathrm{Tor}(K, \{\mathbf h_k\}; \rho)$ depends only on the conjugacy class of $\rho.$ By for e.g. \cite{T2}, the Reidemeister torsion is invariant under elementary expansions and elementary collapses of CW-complexes, and by \cite{M}  it is invariant under subdivisions, hence defines an invariant of PL-manifolds and of topological manifolds of dimension less than or equal to $3.$

\subsubsection{Torsion as a function on the character variety}
We list some results by Porti\,\cite{P} for the Reidemeister torsions of hyperbolic $3$-manifolds twisted by the adjoint representation $\mathrm {Ad}_\rho=\mathrm {Ad}\circ\rho$ of the holonomy  $\rho$ of the hyperbolic structure. Here $\mathrm {Ad}$ is the adjoint action of $\mathrm {PSL}(2;\mathbb C)$ on its Lie algebra $\mathbf{sl}(2;\mathbb C)\cong \mathbb C^3.$

For a closed oriented hyperbolic $3$-manifold  $M$ with the holonomy representation $\rho,$ by the Weil local rigidity theorem and the Mostow rigidity theorem,
$\mathrm H_k(M;\mathrm{Ad}_\rho)=0$ for all $k.$ Then the twisted Reidemeister torsion 
$$\mathrm{Tor}(M;\mathrm{Ad}_\rho)\in\mathbb C^*/\{\pm 1\}$$
 is defined without making any additional choice.

For a compact, orientable  $3$-manifold  $M$ with boundary consisting of $n$ disjoint tori $T_1 \dots,  T_n$ whose interior admits a complete hyperbolic structure with  finite volume, let $\mathrm X(M)$ be the $\mathrm{PSL}(2; \CC)$-character variety of $M,$ let $\mathrm X_0(M)\subset\mathrm X(M)$ be the distinguished component containing the character of a chosen lifting of the holonomy representation of the complete hyperbolic structure of $M,$ and let $\mathrm X_0^{\text{irr}}(M)\subset\mathrm X_0(M)$ be the subset consisting of the irreducible characters. 
 
\begin{theorem}[Section 3.3.3, \cite{P}]\label{HM} For a system of simple closed curves $\boldsymbol\alpha=(\alpha_1,\dots,\alpha_n)$ on $\partial M$ with $\alpha_i\subset T_i,$ $i\in\{1,\dots,n\},$  and a character $[\rho]$ in a non-empty Zariski open subset of $\mathrm X_0^{\text{irr}}(M),$  we have:
\begin{enumerate}[(i)]
\item For $k\neq 1,2,$ $\mathrm H_k(M;\mathrm{Ad}\rho)=0.$
\item  For $i\in\{1,\dots,n\},$ up to scalar $\mathrm Ad_\rho(\pi_1(T_i))^T$ has a unique invariant vector $\mathbf I_i\in \mathbb C^3;$ and
$$\mathrm H_1(M;\mathrm{Ad}\rho)\cong \mathbb C^n$$ 
with a basis
$$\mathbf h^1_{(M,\alpha)}=\{\mathbf I_1\otimes [\alpha_1],\dots, \mathbf I_n\otimes [\alpha_n]\}$$
where $([\alpha_1],\dots,[\alpha_n])\in \mathrm H_1(\partial M;\mathbb Z)\cong 
\bigoplus_{i=1}^n\mathrm H_1(T_i;\mathbb Z).$
 
\item Let $([T_1],\dots,[T_n])\in \bigoplus_{i=1}^n\mathrm H_2(T_i;\mathbb Z)$ be the fundamental classes of $T_1,\dots, T_n.$ Then 
 $$\mathrm H_2(M;\mathrm{Ad}\rho)\cong\bigoplus_{i=1}^n\mathrm H_2(T_i;\mathrm{Ad}\rho)\cong \mathbb C^n$$ 
with  a basis 
$$\mathbf h^2_M=\{\mathbf I_1\otimes [T_1],\dots, \mathbf I_n\otimes [T_n]\}.$$
\end{enumerate}
\end{theorem}

\begin{remark}[\cite{P}] Important examples of the generic characters in Theorem \ref{HM} include the holonomy representation of the complete hyperbolic structure on the interior of $M,$
 the restriction of the holonomy of the closed $3$-manifold $M_\mu$ obtained from $M$ by doing the hyperbolic Dehn surgery along the system of simple closed curves $\mu$ on $\partial M,$
 and by \cite{HK} the holonomy of a hyperbolic structure on the interior of $M$ whose completion is a conical manifold with cone angles less than $2\pi.$
\end{remark}

\begin{definition}\label{reg} Let $\boldsymbol\alpha=(\alpha_1,\dots,\alpha_n)$ be a system of simple closed curves on $\partial M$ with $\alpha_i\subset T_i,$ $i\in\{1,\dots,n\}.$  A character $[\rho]\in \mathrm X_0^{\text{irr}}(M)$ is \emph{$\boldsymbol\alpha$-regular} if condition (ii) in Theorem \ref{HM} is satisfied.
\end{definition}
It is known that for any $\boldsymbol\alpha=(\alpha_1,\dots,\alpha_n)$, every character near the character of the discrete faithful representation is $\boldsymbol\alpha$-regular \cite{P, NZ}. In particular, the set $\mathcal{Z}_{\boldsymbol\alpha} \subset X_0^{\text{irr}}(M)$ of all $\boldsymbol\alpha$-regular character is non-empty and Zariski open.

\begin{definition} \label{ATRT}
The \emph{adjoint twisted Reidemeister torsion} of $M$  \emph{with respect to $\boldsymbol\alpha$} is the function 
$$\mathbb T_{(M,\boldsymbol\alpha)}: \mathrm X_0^{\text{irr}}(M)\to\mathbb C/\{\pm 1\}$$ 
defined by
$$\mathbb T_{(M,\boldsymbol\alpha)}([\rho])=\mathrm{Tor}(M, \{\mathbf h^1_{(M,\boldsymbol\alpha)},\mathbf h^2_M\};\mathrm{Ad}_\rho)$$
if $[\rho] \in \mathcal{Z}_{\boldsymbol\alpha}$, and by $0$ if otherwise. 
\end{definition}

\begin{theorem}[Theorem 4.1, \cite{P}]\label{funT}
Let $M$ be a compact, orientable  $3$-manifold with boundary consisting of $n$ disjoint tori $T_1 \dots,  T_n$ whose interior admits a complete hyperbolic structure with  finite volume. Let  $\mathbb C(\mathrm X_0^{\text{irr}}(M))$ be the ring of rational functions over $\mathrm X_0^{\text{irr}}(M).$ Then
\begin{equation*}
\begin{split}
\mathrm H_1(\partial M;\mathbb Z)&\to \mathbb C( X_0^{\text{irr}}(M))\\
\quad\quad\boldsymbol\alpha\quad\ \  &\mapsto \quad\mathbb T_{(M,\boldsymbol\alpha)}
\end{split}
\end{equation*}
 up to sign defines  a  function which is a $\mathbb Z$-multilinear homomorphism with respect to the direct sum $\mathrm H_1(\partial M;\mathbb Z)\cong 
\bigoplus_{i=1}^n\mathrm H_1(T_i;\mathbb Z)$ satisfying the following properties:
\begin{enumerate}[(i)]
\item For any system of simple closed curves $\boldsymbol\alpha$ on $\partial M,$ the support of $\mathbb T_{(M,\boldsymbol\alpha)}$
contains a Zariski-open subset of $\mathrm X_0^{\text{irr}}(M)$ consisting of all the $\boldsymbol\alpha$-regular characters in $\mathrm X^{\text{irr}}_0(M).$ 

\item \emph{(Change of Curves Formula).} Let $\boldsymbol\beta=\{\beta_1,\dots,\beta_n\}$ and $\boldsymbol\gamma=\{\gamma_1,\dots,\gamma_n\}$ be two systems of simple closed curves on $\partial M.$ Let $( \mathrm{H}(\beta_1),\dots,  \mathrm{H}(\beta_n))$ and $( \mathrm{H}(\gamma_1),\dots, \mathrm{H}(\gamma_n))$ respectively be the logarithmic holonomies of the curves in $\boldsymbol\beta$ and $\boldsymbol\gamma.$ Then we have the equality of rational functions
\begin{equation}\label{coc}
\mathbb T_{(M,\boldsymbol\beta)}
=\pm\det\bigg( \frac{\partial ( \mathrm{H}(\beta_1),\dots,  \mathrm{H}(\beta_n))}{\partial ( \mathrm{H}(\gamma_1),\dots, \mathrm{H}(\gamma_n))}\bigg)\mathbb T_{(M,\boldsymbol\gamma)},
\end{equation}
 where $\frac{\partial ( \mathrm{H}(\beta_1),\dots,  \mathrm{H}(\beta_n))}{\partial ( \mathrm{H}(\gamma_1),\dots, \mathrm{H}(\gamma_n))}$ is the Jocobian matrix of the holomorphic function $( \mathrm{H}(\beta_1),\dots,  \mathrm{H}(\beta_n))$ with respect to $( \mathrm{H}(\gamma_1),\dots, \mathrm{H}(\gamma_n)).$

\item \emph{(Surgery Formula).} For $m$ with $0\leqslant m\leqslant n,$ let $\boldsymbol\mu=(\mu_1,\dots,\mu_m)$ be a system of simple closed curves on $\partial M$ such that $\mu_i\subset T_i,$ and let $\boldsymbol\nu=(\nu_{m+1},\dots,\nu_n)$  be a system of simple closed curves on $\partial M$ such that $\nu_j\subset T_j.$ Let $M_{\boldsymbol\mu}$ be a hyperbolic $3$-manifold obtained from $M$ be doing the Dehn-filling along $\boldsymbol\mu.$ Then $\boldsymbol\nu$ can be considered as a system of simple closed curves on $\partial M_{\boldsymbol\mu}.$ Let $[\rho_{\boldsymbol\mu}]\in\mathrm X_0^{\text{irr}}(M_{\boldsymbol\mu})$ and let $[\rho]\in \mathrm X^{\text{irr}}_0(M)$ be the restriction of $[\rho_{\boldsymbol\mu}]$ on $M.$ Let $( \mathrm{H}(\gamma_1),\dots, \mathrm{H}(\gamma_m))$ be the logarithmic holonomies in $\rho$ of the core curves $\gamma_1,\dots,\gamma_m$ of the solid tori filled in. If $\rho_{\boldsymbol\mu}$ is $\boldsymbol\nu$-regular and $\sinh (\mathrm{H}(\gamma_i)/2) \neq 0$ for $i=1,\dots,m$, then $\rho$ is $\boldsymbol\mu\cup\boldsymbol\nu$-regular, and 
\begin{equation}\label{sf}
\mathbb T_{(M_{\boldsymbol\mu},\boldsymbol\nu)}([\rho_{\boldsymbol\mu}])=\pm\mathbb T_{(M,\boldsymbol\mu\cup\boldsymbol\nu)}([\rho])\prod_{i=1}^m\frac{1}{4\sinh^2\frac{  \mathrm{H}(\gamma_i)}{2}}.
\end{equation}
\end{enumerate}
\end{theorem}

\begin{remark}\label{rmksurg} For the surgery formula, if the core curve $\gamma_i$ is free homotopic to a geodesic in $M_{\boldsymbol\mu}$, then $\Re\mathrm{H}(\gamma_i)\neq 0$ and hence the condition that $\sinh (\mathrm{H}(\gamma_i)/2) \neq 0$ is satisfied. In particular, by the result in \cite{HK}, if we remove at most 114 simple closed curves on each boundary torus, then for the Dehn-fillings along any remaining system of simple closed curves, we have $\sinh (\mathrm{H}(\gamma_i)/2) \neq 0$.
\end{remark}
\begin{remark}
In this paper, we only focus on the adjoint twisted Reidemeister torsion defined on the distinguished component of the $\mathrm{PSL}(2;\CC)$-character variety of $M$. Nevertheless, the invariant can be defined on any component of the character variety with dimension equal to the number of boundary component of $M$ (see for example \cite[Section 2.2]{WY3}). 
\end{remark}

\subsubsection{Adjoint twisted Reidemeister torsion of fundamental shadow link complements}\label{TFSL}
In \cite{WY3}, T. Yang and the second author compute the adjoint twisted Reidemeister torsion of fundamental shadow link complement in terms of the determinants of the Gram matrix functions of the $D$-blocks. 
Recall that an edge of a truncated tetrahedron is the intersection of two hexagonal faces. Let $(e_1,\dots,e_6)$ be the six edges of a truncated tetrahedron in such a way that $(e_k, e_{k+3})$ is a pair of opposite edges for $k=1,2,3$ as shown in Figure \ref{pi1tetra}. By construction, a $D$-block $\mathcal{D}$ is obtained by gluing two truncated tetrahedra together along the hexagonal faces and then removing the edges of the tetrahedra.  Especially, as shown in Figure \ref{pi1tetra}, since $\mathcal{D}$ is homeomorphic to the complement of the tetrahedron graph, the fundamental group $\pi_1(\mathcal{D})$ of a $D$-block is a rank $3$ free group generated by $[m_1], [m_2]$ and $[m_3]$, where $m_1,m_2,m_3 \in \pi_1(\mathcal{D})$ are the meridians around the edges $e_1,e_2,e_3$. In the Wirtinger presentation, the meridians around the edges $e_4,e_5$ and $e_6$ are given by $[m_3^{-1}m_2], [m_1^{-1}m_3]$ and $[m_2^{-1}m_1]$ respectively, where $m_k^{-1}m_j$ represents the curve that goes along $m_k^{-1}$ first and then goes along $m_j$.

\begin{figure}[h]
\centering
\includegraphics[scale=0.25]{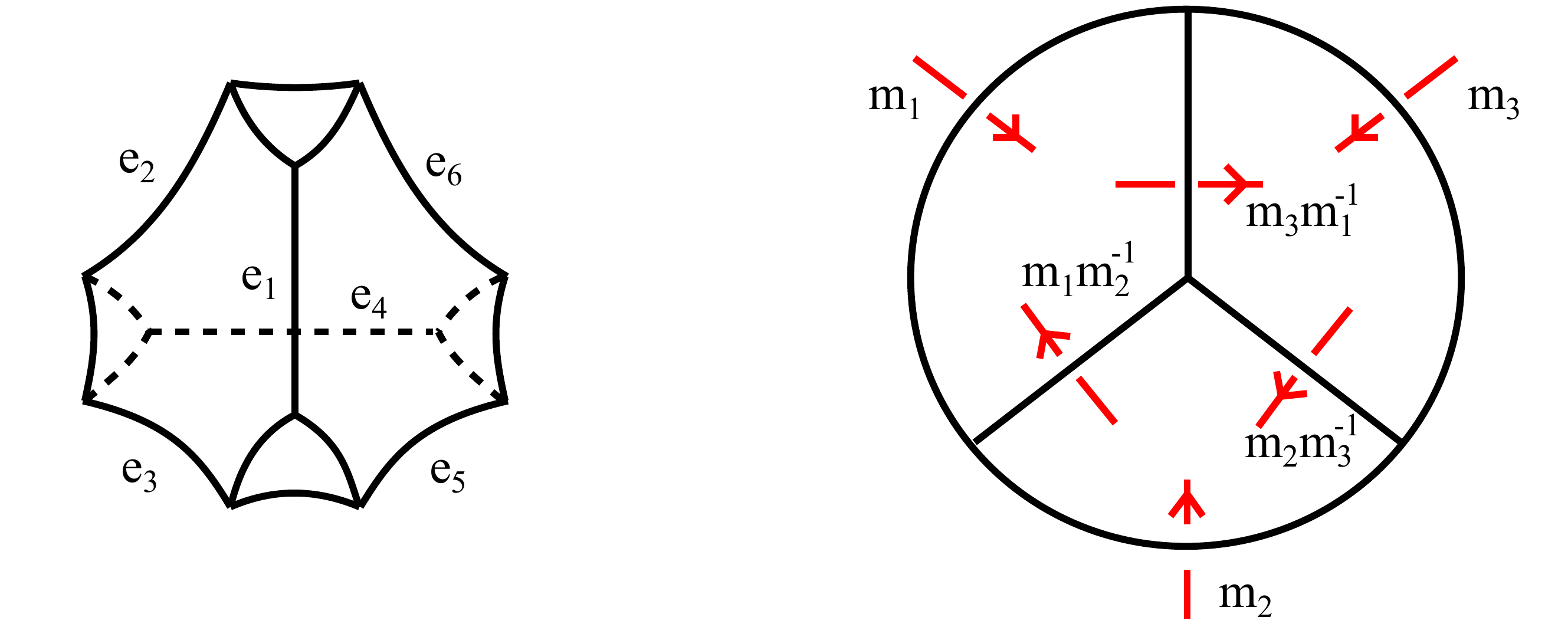}
\caption{The Wirtinger presentation of $\pi_1(\mathcal{D})$. In this figure, $m_i^{-1}m_j$ represents the element in $\pi_1(\mathcal{D})$ that follows $m_i^{-1}$ and then $m_j$.}
\label{pi1tetra}
\end{figure}

The $\mathrm{SL}(2;\CC)$-character variety of the rank 3 free group $\pi_1(\mathcal{D})$ is studied by Goldman in \cite{G1}. It is shown in \cite[Section 5.1.1]{G1} that the $\mathrm{SL}(2;\CC)$-character variety $\chi(\mathcal{D})$ of $\pi_1(\mathcal{D})$ is homeomorphic to a hypersurface in $\CC^7$ parametrized by the seven traces 
$\mathrm{Tr}\tilde\rho([m_1])$, $\mathrm{Tr}\tilde\rho([m_2])$, $\mathrm{Tr}\tilde\rho([m_3])$, $\mathrm{Tr}\tilde\rho([m_1m_2])$, $\mathrm{Tr}\tilde\rho([m_1m_3])$, $\mathrm{Tr}\tilde\rho([m_2m_3])$ and $\mathrm{Tr}\tilde\rho([m_1m_2m_3])$, where $[\tilde\rho] \in \chi(\mathcal{D})$. Moreover, it is shown in \cite[Proposition 5.1.1]{G1} that $\chi(\mathcal{D})$ is a double branched cover of $\CC^6$ parametrized by the first six traces. Furthermore, a representation $\tilde\rho: \pi_1(D) \to \mathrm{SL}(2;\CC)$ is not in the branch locus if and only if 
$$f_{\mathcal{D}}\big(\mathrm{Tr}\tilde\rho([m_1]),\mathrm{Tr}\tilde\rho([m_2]),\mathrm{Tr}\tilde\rho([m_3]),\mathrm{Tr}\tilde\rho([m_1m_2]),\mathrm{Tr}\tilde\rho([m_1m_3]),\mathrm{Tr}\tilde\rho([m_2m_3]) \big)\neq 0,$$
where $f_{\mathcal{D}}$ is the polynomial
\begin{equation*}
\begin{split}
f_{\mathcal{D}}(t_1,t_2,t_3&,t_{12},t_{13},t_{23})=\big(t_{12}t_3+t_{13}t_2+t_{23}t_1-t_1t_2t_3\big)^2\\
&-4\big(t_1^2+t_2^2+t_3^2+t_{12}^2+t_{13}^2+t_{23}^2-t_1t_2t_{12}-t_1t_3t_{13}-t_2t_3t_{23}+t_{12}t_{13}t_{23}-4\big).
\end{split}
\end{equation*}

In our setting, we let $m_4=m_3^{-1}m_2$, $m_5=m_1^{-1}m_3$ and $m_6=m_2^{-1}m_1$ such that $m_1,m_2,\dots,m_6$ are the meridians around the six edges $e_1,e_2,\dots,e_6$ in $\mathcal{D}$ respectively. Given a representation $\tilde\rho: \pi_1(D) \to \mathrm{SL}(2;\CC)$, 
the \emph{logarithmic holonomies} of $(m_1,m_2,\dots,m_6)$ in $\widetilde \rho$ are up to sign the complex numbers 
$$(\mathrm{H}(m_1),\mathrm{H}(m_2),\dots,\mathrm{H}(m_6)) \in \CC^6$$ such that 
$$
\mathrm {Tr}\widetilde\rho([m_l])=-2\cosh\frac{\mathrm{H}(m_l)}{2}
$$
for $l=1,\dots,6$. 
In this way, if $\mathcal{D}$ is with the hyperbolic structure obtained by doubling the regular ideal octahedron, $\rho_0:\pi_1(\mathcal{D})\to\mathrm{PSL}(2;\mathbb C)$ is the holonomy representation of this hyperbolic structure on $\mathcal{D}$ and $\widetilde \rho_0:\pi_1(\mathcal{D})\to\mathrm{SL}(2;\mathbb C)$ is the lifting of $\rho_0$ with
$$
\mathrm {Tr}\widetilde\rho([m_l])=-2
$$
for $l=1,\dots,6$, then the logarithmic holonomies of $(m_1,\dots,m_6)$  in $\widetilde\rho_0$ are $(0,\dots,0).$ We notice that the complete hyperbolic structure on a fundamental shadow link complement is obtained by gluing such $D$-blocks together by isometries along the faces. Therefore, this hyperbolic structure can be considered as the ``complete hyperbolic structure" on $\mathcal{D}.$

Given a representation $\widetilde \rho:\pi_1(\mathcal{D})\to\mathrm{SL}(2;\mathbb C)$, the \emph{determinant of the associated Gram matrix} is the determinant of the matrix 
$$\mathbb G=\mathbb G\Bigg(\frac{\mathrm{H}(m_1)}{2},\dots,\frac{\mathrm{H}(m_6)}{2} \Bigg)=\left[
\begin{array}{cccc}
1 & -\cosh \frac{\mathrm{H}(m_1)}{2} & -\cosh  \frac{\mathrm{H}(m_6)}{2} &-\cosh  \frac{\mathrm{H}(m_5)}{2}\\
-\cosh  \frac{\mathrm{H}(m_1)}{2}& 1 &-\cosh  \frac{\mathrm{H}(m_2)}{2} & -\cosh  \frac{\mathrm{H}(m_3)}{2}\\
-\cosh  \frac{\mathrm{H}(m_6)}{2} & -\cosh  \frac{\mathrm{H}(m_2)}{2} & 1 & -\cosh  \frac{\mathrm{H}(m_4)}{2} \\
-\cosh  \frac{\mathrm{H}(m_5)}{2} & -\cosh  \frac{\mathrm{H}(m_3)}{2} & -\cosh  \frac{\mathrm{H}(m_4)}{2}  & 1 \\
 \end{array}\right].$$
 By the trace identity of the matrices in $\mathrm{SL}(2;\mathbb C)$ that
$$
\mathrm{Tr}(A)\mathrm{Tr}(B) = \mathrm{Tr}(AB) + \mathrm{Tr}(AB^{-1})
$$
for $A,B \in \mathrm{SL}(2;\mathbb C)$, we have 
\begin{align*}
\mathrm {Tr}\widetilde\rho([m_2m_3]) &= \mathrm {Tr}\widetilde\rho([m_2])\mathrm {Tr}\widetilde\rho([m_3]) - \mathrm {Tr}\widetilde\rho([m_4]),\\
\mathrm {Tr}\widetilde\rho([m_1m_3]) &=\mathrm {Tr}\widetilde\rho([m_1])\mathrm {Tr}\widetilde\rho([m_3]) - \mathrm {Tr}\widetilde\rho([m_5]),\\
\mathrm {Tr}\widetilde\rho([m_1m_2]) &=\mathrm {Tr}\widetilde\rho([m_1])\mathrm {Tr}\widetilde\rho([m_2]) - \mathrm {Tr}\widetilde\rho([m_6]).
\end{align*}
Then by a direct computation, as functions in $\mathrm {Tr}\widetilde\rho([m_1]), \dots,\mathrm {Tr}\widetilde\rho([m_6])$, we have
$$
f_{\mathcal{D}}\big(\mathrm{Tr}\tilde\rho([m_1]),\mathrm{Tr}\tilde\rho([m_2]),\mathrm{Tr}\tilde\rho([m_3]),\mathrm{Tr}\tilde\rho([m_1m_2]),\mathrm{Tr}\tilde\rho([m_1m_3]),\mathrm{Tr}\tilde\rho([m_2m_3]) \big)\\
=16\det \mathbb G,$$
  and $\widetilde\rho$ is not in the branch locus of the double branched cover of the $\mathrm{SL}(2;\mathbb C)$-character variety of $D$ over $\mathbb C^6$  if and only if $\det\mathbb G\neq 0.$

Since $\pi_1(\mathcal{D})$ is a free group, every $\mathrm{PSL}(2;\mathbb C)$-representation of it lifts to  $\mathrm{SL}(2;\mathbb C)$-representation Hence the  $\mathrm{SL}(2;\mathbb C)$-character variety of $\mathcal{D}$ is a branched cover of the $\mathrm{PSL}(2;\mathbb C)$-character variety of  $\mathcal{D},$ and the latter is an irreducible algebraic variety.  For a representation $\rho:\pi_1(\mathcal{D})\to\mathrm{PSL}(2;\mathbb C),$ we defined the logarithmic holonomies $(\mathrm{H}(m_1),\dots,\mathrm{H}(m_6))$ and the determinant of the associated Gram matrix $\mathbb G$ of $\rho$  as those of a lifting $\widetilde \rho:\pi_1(\mathcal{D})\to\mathrm{SL}(2;\mathbb C)$ of $\rho.$ Notice that the logarithmic holonomies depend on the choice of the liftings of $\rho,$ and a different lifting will change $\mathbb G$ by multiplying some rows and the corresponding columns by $-1$ at the same time, which does not change its determinant. Therefore, the \emph{determinant of the associated Gram matrix} $\det\mathbb G$ is independent of the choice of the liftings, and  is a well defined quantity of $\rho.$

By using the determinant of the associated Gram matrix of the building blocks, the adjoint twisted Reidemeister torsion of fundamental shadow link complements is given as follows.

\begin{theorem}[Theorem 1.1, \cite{WY3}] \label{main1}  Let $M=\#^{c+1}(\SS^2\times \SS^1)\setminus L_{\text{FSL}}$ be the complement of a fundamental shadow link $L_{\text{FSL}}$ with $n$ components $L_1,\dots,L_n,$ which is the orientable double of the union of truncated tetrahedra $\Delta_1,\dots, \Delta_c$ along pairs of the triangles of truncation (see Section \ref{fsl}).

\begin{enumerate}[(1)] 
\item  Let $\boldsymbol m=(m_1,\dots, m_n)$ be the system of the meridians of a tubular neighborhood of the components of $L_{\text{FSL}}.$ For an $\boldsymbol m$-regular $\mathrm{PSL}(2; \mathbb C)$-character $[\rho]$ of $M$ (see Definition \ref{reg}), let $( \mathrm{H}(m_1),
\dots, \mathrm{H}(m_n))$ be the logarithmic holonomies of $\boldsymbol m$ in $\rho.$ For each $k\in\{1,\dots,c\},$ let $L_{k_1},\dots,L_{k_6}$ be the components of $L_{\text{FSL}}$ intersecting $\Delta_k,$ and let $\mathbb G_k=\mathbb G\Big(\frac{ \mathrm{H}(m_{k_1})}{2},\dots,\frac{ \mathrm{H}(m_{k_6})}{2}\Big)$ be the value of the Gram matrix function at $\Big(\frac{ \mathrm{H}(m_{k_1})}{2},\dots,\frac{ \mathrm{H}(m_{k_6})}{2}\Big).$ Then the adjoint twisted Reidemeister torsion of $M$ with respect to $\boldsymbol m$ (see Definition \ref{ATRT}) at $[\rho]$ is 
$$\mathbb T_{(M,\boldsymbol m)}([\rho])=\pm2^{3c}\prod_{k=1}^c \sqrt{\det\mathbb G_k}.$$

\item In addition to the conditions of (1), let $\boldsymbol \mu=(\mu_1,\dots,\mu_n)$ be a system of simple closed curves on $\partial M,$ and let $( \mathrm{H}(\mu_1),\dots,  \mathrm{H}(\mu_n))$ be their logarithmic holonomies which are holomorphic functions of $( \mathrm{H}(m_1),\dots, \mathrm{H}(m_n))$ near $[\rho].$ If $[\rho]$ is $\boldsymbol\mu$-regular, then the adjoint twisted Reidemeister torsion of $M$ with respect to $\boldsymbol\mu$ at $[\rho]$ is
 $$ \mathbb T_{(M,\boldsymbol\mu)}([\rho])=\pm2^{3c}\det\bigg(\frac{\partial ( \mathrm{H}(\mu_1),\dots,  \mathrm{H}(\mu_n))}{\partial ( \mathrm{H}(m_1),\dots, \mathrm{H}(m_n))}\bigg)\prod_{k=1}^c \sqrt{\det\mathbb G_k},$$
 where $\frac{\partial ( \mathrm{H}(\mu_1),\dots,  \mathrm{H}(\mu_n))}{\partial ( \mathrm{H}(m_1),\dots, \mathrm{H}(m_n))}$ is the Jacobian matrix of the holomorphic function $( \mathrm{H}(\mu_1),\dots,  \mathrm{H}(\mu_n))$ with respect to $( \mathrm{H}(m_1),\dots, \mathrm{H}(m_n))$ evaluated at $[\rho].$ 
\end{enumerate} 
\end{theorem}

Let $M$ be a fundamental shadow link complement as in Theorem \ref{main1}. 
For $m$ with $0\leqslant m\leqslant n,$ let $\boldsymbol\mu=(\mu_1,\dots,\mu_m)$ be a system of simple closed curves on $\partial M$ such that $\mu_i\subset T_i,$ and let $\boldsymbol\nu=(\nu_{m+1},\dots,\nu_n)$  be a system of simple closed curves on $\partial M$ such that $\nu_j\subset T_j.$  Let $M_{\boldsymbol\mu}$ be the  $3$-manifold  obtained from $M$ by doing the Dehn-filling along $\boldsymbol\mu.$ Then $\boldsymbol\nu$ can be considered as a system of simple closed curves on $\partial M_{\boldsymbol\mu}.$ If $m=n,$ then $\boldsymbol \nu=\emptyset$ and $M_{\boldsymbol\mu}$ is a closed  $3$-manifold.

 \begin{theorem}[Theorem 1.4, \cite{WY3}] \label{Rtorthm}  Suppose $M_{\boldsymbol\mu}$ is hyperbolic. 
Let $[\rho_{\boldsymbol\mu}]$ be a $\boldsymbol\nu$-regular character of $M_{\boldsymbol\mu}$  and let $\rho$ be the restriction of $\rho_{\boldsymbol\mu}$ on $M.$ 
Let $( \mathrm{H}(m_1),\dots, \mathrm{H}(m_n))$ be the logarithmic holonomies of the system of meridians $\boldsymbol m$  in $[\rho]$ and for each $k\in\{1,\dots,c\},$ let $L_{k_1},\dots,L_{k_6}$ be the components of $L_{\text{FSL}}$ intersecting $\Delta_k$ and let $\mathbb G_k=\mathbb G\Big(\frac{ \mathrm{H}(m_{k_1})}{2},\dots,\frac{ \mathrm{H}(m_{k_6})}{2}\Big)$ be the value of the Gram matrix function at $\Big(\frac{ \mathrm{H}(m_{k_1})}{2},\dots,\frac{ \mathrm{H}(m_{k_6})}{2}\Big).$ Let $( \mathrm{H}(\mu_1),\dots, \mathrm{H}(\mu_m))$ and $( \mathrm{H}(\nu_{m+1}),\dots, \mathrm{H}(\nu_n))$ respectively be the logarithmic holonomies of $\boldsymbol\mu$ and $\boldsymbol\nu$ considered as functions of $( \mathrm{H}(m_1),\dots, \mathrm{H}(m_n))$ near $[\rho].$  Let $(\gamma_1,\dots,\gamma_m)$ be a system of  simple closed curves on $\partial M$ that are isotopic to the core curves of the solid tori filled in and let $( \mathrm{H}(\gamma_1),\dots, \mathrm{H}(\gamma_m))$ be their logarithmic holonomies in $[\rho].$ 
If $[\rho]$ is in the distinguished component of the $\mathrm{PSL}(2;\mathbb C)$-character variety of $M,$ then the adjoint twisted Reidemeister torsion of $M_{\boldsymbol\mu}$ with respect to $\boldsymbol \nu$ at $[\rho_{\boldsymbol\mu}]$ is 
\begin{align*} &\mathbb T_{(M_{\boldsymbol\mu},\boldsymbol\nu)}([\rho_{\boldsymbol\mu}])\\
=& \pm2^{3c-2m}\det\bigg(\frac{\partial ( \mathrm{H}(\mu_1),\dots, \mathrm{H}(\mu_m), \mathrm{H}(\nu_{m+1}),\dots, \mathrm{H}(\nu_n))}{\partial ( \mathrm{H}(m_1),\dots, \mathrm{H}(m_n))}\bigg)\prod_{k=1}^c \sqrt{\det\mathbb G_k}\prod_{i=1}^m\frac{1}{\sinh^2\frac{ \mathrm{H}(\gamma_i)}{2}},
\end{align*}
 where $\frac{\partial ( \mathrm{H}(\mu_1),\dots, \mathrm{H}(\mu_m), \mathrm{H}(\nu_{m+1}),\dots, \mathrm{H}(\nu_n))}{\partial ( \mathrm{H}(m_1),\dots, \mathrm{H}(m_n))}$ is the Jacobian matrix of the holomorphic function \linebreak $( \mathrm{H}(\mu_1),\dots, \mathrm{H}(\mu_m), \mathrm{H}(\nu_{m+1}),\dots, \mathrm{H}(\nu_n))$ with respect to $( \mathrm{H}(m_1),\dots, \mathrm{H}(m_n))$ evaluated at $[\rho].$
 
 In particular, if $M_{\boldsymbol\mu}$ is closed, $\rho_{\boldsymbol\mu}$ is the holonomy representation of the hyperbolic structure and $[\rho]$ is in the distinguished component of the $\mathrm{PSL}(2;\mathbb C)$-character variety of $M,$ then the adjoint twisted Reidemeister torsion of $M_{\boldsymbol\mu}$ is 
  $$ \mathrm {Tor}(M_{\boldsymbol\mu};\mathrm{Ad}_{\rho_{\boldsymbol\mu}})=\pm2^{3c-2n}\det\bigg(\frac{\partial ( \mathrm{H}(\mu_1),\dots, \mathrm{H}(\mu_n))}{\partial ( \mathrm{H}(m_1),\dots, \mathrm{H}(m_n))}\bigg)\prod_{k=1}^c \sqrt{\det\mathbb G_k}\prod_{i=1}^n\frac{1}{\sinh^2\frac{ \mathrm{H}(\gamma_i)}{2}}.$$
 
\end{theorem}

\subsection{Neumann-Zagier datum and gluing equations}\label{NZD}
Let $M$ be an oriented hyperbolic $3$-manifold with $\partial M = T_1 \coprod \dots \coprod T_k$. On each $T_i$, we choose a simple closed curve $\alpha_i \in \pi_1(T_i)$ and let $\alpha = (\alpha_1,\dots,\alpha_k) \in \pi_1(\partial M)$ be the system of simple closed curves. Let $\mathcal{T} = \{\Delta_1,\dots,\Delta_n\}$ be an ideal triangulation of $M$ and let $E = \{e_1,\dots, e_n\}$ be the set of edges. For each $\Delta_i$, we choose a quad type (i.e. a pair of opposite edges) and assign a shape parameter $z_i\in \CC\setminus \{0,1\}$ to the edges. For $z_i \in \CC\setminus \{0,1\}$, we define $z_i ' = \frac{1}{1-z_i}$, $z_i'' = 1-\frac{1}{z_i}$. Recall that for each ideal tetrahedron, opposite edges share the same shape parameters (See Figure \ref{edgehol}, left). By \cite{NZ}, there exists $n-k$ linearly independent edge equations, in the sense that if these $n-k$ edge equations are satisfied, the remaining $k$ edge equations will automatically be satisfied. Without loss of generality we assume that $\{e_1,\dots,e_{n-k}\}$ is a set of linearly independent edges. For each edge $e_i$, we let $E_{i,j}$ be the numbers of edges with shape parameter $z_j$ that is incident to $e_i$. We define $E_{i,j}'$ and $E_{i,j}''$ be respectively the corresponding counting with respect to $z_j'$ and $z_j''$ (see Figure \ref{edgehol}, middle). The gluing variety $\mathcal{V}_{\mathcal{T}}$ is the affine variety in $(z_1,z_1',z_1'', \dots, z_n, z_n', z_n'') \in \CC^{3n}$ defined by the zero sets of the polynomials
\begin{align*}
p_i = z_i(1-z_i'')- 1,\quad p_i' = z_i'(1-z_i) -1,\quad p_i'' = z_i''(1-z_i') -1 
\end{align*}
for $i=1,\dots, n$ and the polynomials
\begin{align*}
\prod_{i=1}^n z_i^{E_{ij}} (z_i')^{E_{ij}'} (z_i'')^{E_{ij}''}  -1
\end{align*}
for $j=1,\dots, n-k$. By using the equations $p_i=p_i'=p_i'' =0$ for $i=1,\dots,n$, for simplicity we will use $(z_1,\dots, z_n) \in (\CC\setminus \{0,1\})^n$ to represent a point in $\mathcal{V}_{\mathcal{T}}$. Besides, for each $\alpha_i \in \pi_1(T_i)$, we let $C_{i,j}$ be the numbers of edges with shape $z_i$ on the left hand side of $\alpha_i$ minus the numbers of edges with shape $z_i$ on the right hand side of $\alpha_i$. We define $C_{i,j}'$ and $C_{i,j}''$ be respectively the corresponding counting with respect to $z_j'$ and $z_j''$ (see Figure \ref{edgehol}, right). Given $(z_1,\dots, z_n) \in \mathcal{V}_{\mathcal{T}}$, the holonomy of $\alpha_i$ is given by
$$
\lambda_i = \prod_{i=1}^n z_i^{C_{ij}} (z_i')^{C_{ij}'} (z_i'')^{C_{ij}''}
$$
for $i=1,\dots, k$. There is a well-defined map
$$ \mathcal{P}_{\mathcal{T}} : \mathcal{V}_{\mathcal{T}} \to \mathrm{X}(M)$$
that sends $\mathbf{z}=(z_1,\dots, z_n)  \in \mathcal{V}_{\mathcal{T}}$ to the character $[\rho_{\mathbf{z}}]$ of the pseudo-developing map $\rho_{\mathbf{z}}$ with 
$$
\rho_{\mathbf{z}}(\alpha_i) = \pm 
\begin{pmatrix}
\sqrt{\lambda_i} & * \\
0 & \sqrt{\lambda_i}^{-1}
\end{pmatrix}$$
for $i=1,\dots,k$ up to conjugation. 

\begin{figure}[h]
  \centering
          \includegraphics[scale=0.13]{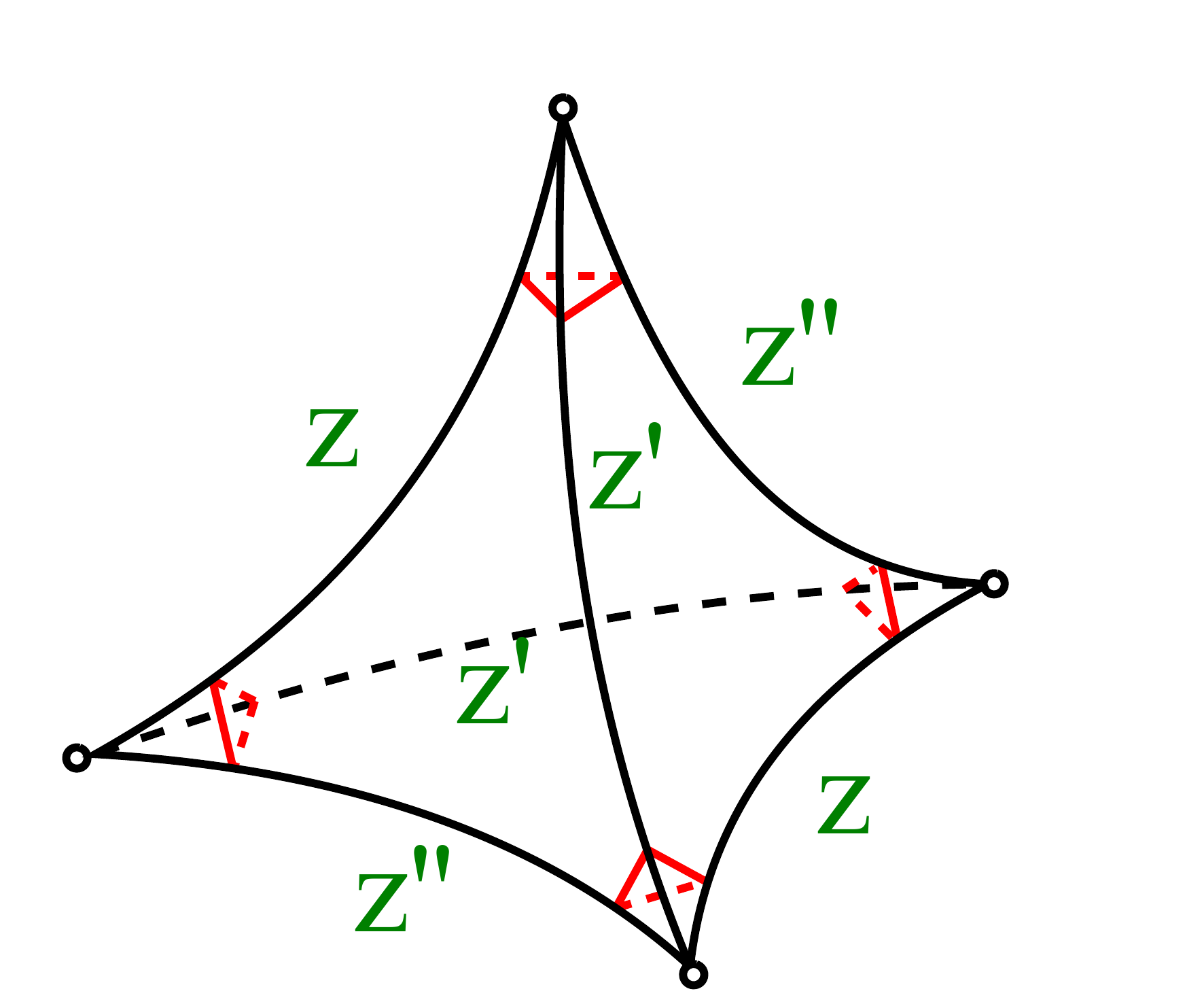} \qquad
\includegraphics[scale=0.13]{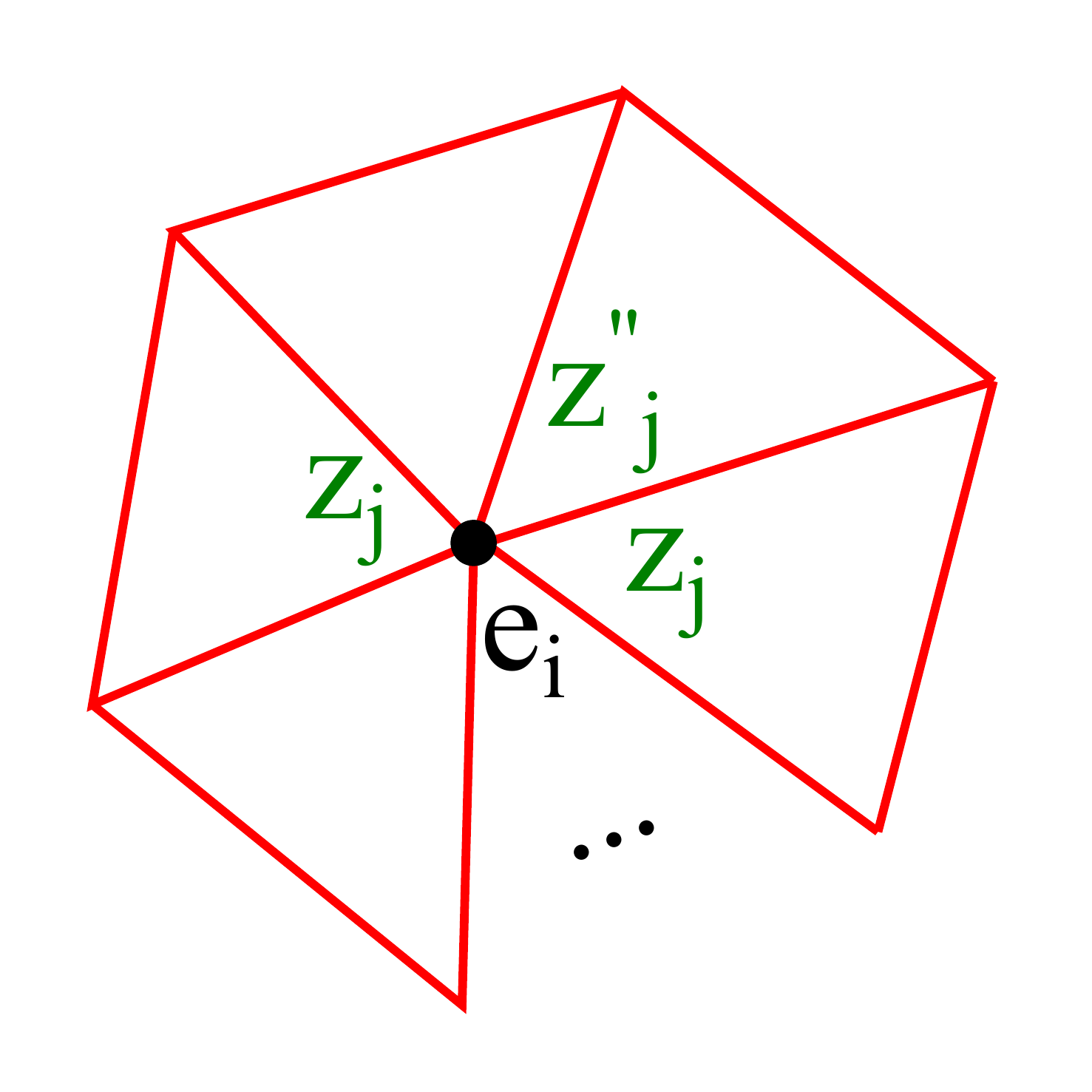}      \qquad\quad  \includegraphics[scale=0.13]{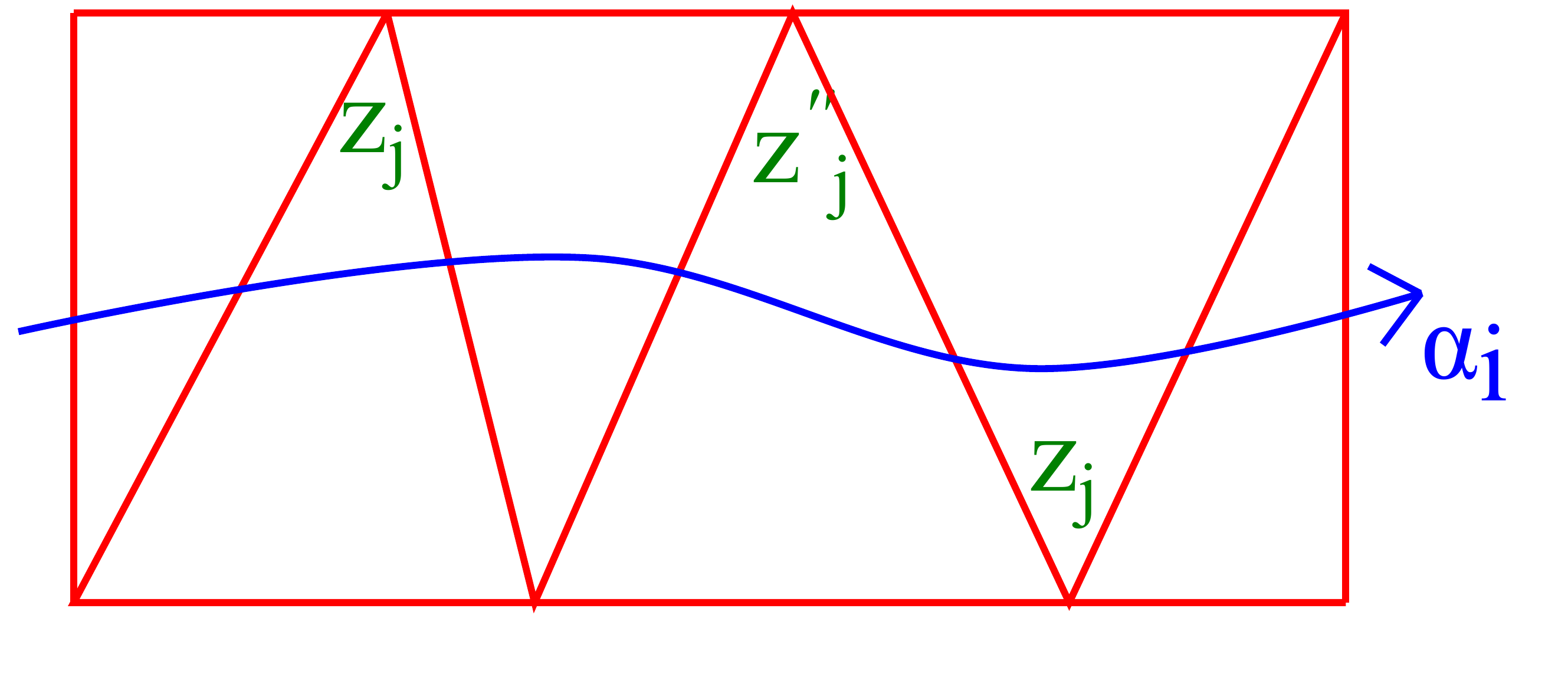}
\caption{On the left hand side, we have an ideal tetrahedron with shape parameters $z_i, z_i'$ and $z_i''$ assigned to the edges. In the middle, the black dot with a label $e_i$ corresponds to the $i$-th edge, which is surrounded by truncated triangles around the ideal vertices of the tetrahedra in the triangulation. In this example, $E_{ij} = 2$. On the right hand side, the rectangle is a fundamental domain of the boundary torus $T_i$ and $\alpha_i$ is a lifting of the simple closed curve $\alpha_i \in \pi_1(T_i)$ to the fundamental domain. In this example, we have $C_{ij} = 1 - 1 = 0$.}\label{edgehol}
\end{figure}
        
We define three $n \times n$ matrices $G,G',G'' \in M_{n\times n}(\ZZ)$ by
\begin{align*}
G=
\begin{pmatrix}
E_{1,1} & E_{1,2} & \dots & E_{1,n} \\
\vdots & &  & \vdots \\
E_{n-k,1} & E_{n-k,2} & \dots & E_{n-k,n} \\
C_{1,1} & C_{1,2} & \dots & C_{1,n} \\
\vdots & &  & \vdots   \\
C_{k,1} & C_{k,2} & \dots & C_{k,n}   
\end{pmatrix},\qquad
G'=
\begin{pmatrix}
E_{1,1}' & E_{1,2}' & \dots & E_{1,n}' \\
\vdots & &  & \vdots \\
E_{n-k,1}' & E_{n-k,2}' & \dots & E_{n-k,n}' \\
C_{1,1}' & C_{1,2}' & \dots & C_{1,n}' \\
\vdots & &  & \vdots   \\
C_{k,1}' & C_{k,2}' & \dots & C_{k,n}'   
\end{pmatrix}
\end{align*}
and
\begin{align*}
G''=
\begin{pmatrix}
E_{1,1}'' & E_{1,2}'' & \dots & E_{1,n}'' \\
\vdots & &  & \vdots \\
E_{n-k,1}'' & E_{n-k,2}'' & \dots & E_{n-k,n}'' \\
C_{1,1}'' & C_{1,2}'' & \dots & C_{1,n}'' \\
\vdots & &  & \vdots   \\
C_{k,1}'' & C_{k,2}'' & \dots & C_{k,n}''  
\end{pmatrix}.
\end{align*}
Let $\log z, \log z', \log z'' \in M_{n\times 1}(\CC)$ be the column vectors
\begin{align*}
\log z &= 
\begin{pmatrix}
\log z_1\\
\vdots \\
\log z_n
\end{pmatrix} ,\quad
\log z' = 
\begin{pmatrix}
\log z_1' \\
\vdots\\
\log z_n'
\end{pmatrix} \text{ and }
\log z'' = 
\begin{pmatrix}
\log z_1'' \\
\vdots\\
\log z_n''
\end{pmatrix}.
\end{align*}
With respect to the ideal triangulation $\mathcal{T}$, we define the \emph{gluing map} $F:(\CC\setminus\{0,1\})^n \to \CC^n$ by 
\begin{align}\label{defF1loop}
F(z_1,\dots,z_n) = G \cdot \log z + G' \cdot \log z' + G'' \cdot \log z''.
\end{align}
Given $\mathbf{H(\boldsymbol\alpha)}=(\mathrm{H}(\alpha_1), \dots, \mathrm{H}(\alpha_k))\in \CC^k$ and $(l_1,\dots,l_n) \in \ZZ^n$, consider the gluing equations 
\begin{align}\label{gluingeql}
\mathcal{G}_{(l_1,\dots,l_n)}(\mathbf{H(\boldsymbol\alpha)}, \mathbf{z}) = 
\mathbf{0},
\end{align} 
where $\mathbf{z}=(z_1,\dots,z_n)$, $\mathbf{0} \in \CC^n$ is the zero vector and $\mathcal{G}_{(l_1,\dots,l_n)} : \CC^k \times (\CC\setminus\{0,1\})^n \to \CC^n$ is defined by
\begin{align}\label{defG1loop}
\mathcal{G}_{(l_1,\dots,l_n)}(\mathbf{H(\boldsymbol\alpha)}, \mathbf{z}) 
= F(z_1,\dots,z_n) 
-
\begin{pmatrix}
2\pi \sqrt{-1}\\
\vdots\\
2\pi \sqrt{-1}\\
\mathrm{H}(\alpha_1)\\
\vdots\\
\mathrm{H}(\alpha_k)
\end{pmatrix}
-
2\pi\sqrt{-1}
\begin{pmatrix}
l_1 \\
\vdots\\
l_{n-k} \\
l_{n-k+1} \\
\vdots\\
l_{n}
\end{pmatrix}.
\end{align}
Note that each point $(z_1,z_1',z_1'', \dots, z_n, z_n', z_n'') \in \mathcal{V}_{\mathcal{T}}$ gives a solution of the gluing equation (\ref{gluingeq}) with $\lambda_i = e^{\mathrm H(\alpha_i)}$ for $i=1,\dots, k$ and for some $(l_1,\dots, l_n)\in \ZZ^n$. 
To simplify the notation, we let $\mathcal{G} = \mathcal{G}_{(0,\dots,0)}$. In Section \ref{triFSL}, we will focus on the gluing equation 
\begin{align}\label{gluingeq}
\mathcal{G}(\mathbf{H(\boldsymbol\alpha)}, \mathbf{z}) = 
\mathbf{0}.
\end{align} 
Notice that for any $(l_1,\dots, l_n)\in \ZZ^n$, we have 
\begin{align}\label{DG=DF}
D_{\mathbf{z}}\mathcal{G}_{(l_1,\dots,l_n)}(\mathbf{H(\boldsymbol\alpha)}, \mathbf{z}) = DF_{\mathbf{z}},
\end{align}
where $D_{\mathbf{z}}\mathcal{G}_{(l_1,\dots,l_n)}$ and $DF_{\mathbf{z}}$ are the Jacobian matrixs of $\mathcal{G}_{(l_1,\dots,l_n)}$ and $F$ with respect to the variable $\mathbf{z}$.

\subsection{Combinatorial flattening}
Recall that a system of simple closed curves $\boldsymbol\alpha = (\alpha_1,\dots,\alpha_k)\in \pi_1(\partial M)$ consists of $k$ non-trivial simple closed curves $\alpha_i \in \pi_1(T_i)$. Given a vector $\mathbf v\in \CC^n$, we denote the transpose of $\mathbf v$ by $\mathbf v^T$.

\begin{definition}\label{CF}
For a system of simple closed curves $\boldsymbol\alpha = (\alpha_1,\dots,\alpha_k)\in \pi_1(\partial M)$, a combinatorial flattening with respect to $\boldsymbol \alpha$ consists of three vectors 
\begin{align*}
\mathbf{f} = (f_1,\dots,f_n),\quad
\mathbf{f'} = (f_1',\dots,f_n'),\quad
\mathbf{f}'' = (f_1'',\dots,f_n'') \in \mathbb{Z}^n
\end{align*}
such that 
\begin{itemize}
\item for $i=1,\dots, n$, $f_i + f_i' + f_i'' = 1$ and
\item the $i$-th entry of the vector 
$$G\cdot \mathbf{f}^T + G' \cdot {\mathbf{f}'}^T + G'' \cdot {\mathbf{f}''}^T $$ is equal to $2$ for $i=1,\dots, n-k$ and is equal to $0$ for $i=n-k+1,\dots, n$.
\end{itemize}
\end{definition}

We have the following stronger version of combinatorial flattening, which require the last condition in Definition \ref{CF} to be satisfied for all simple closed curves.
\begin{definition}\label{SCF}
A strong combinatorial flattening consists of three vectors 
\begin{align*}
\mathbf{f} = (f_1,\dots,f_n),\quad
\mathbf{f'} = (f_1',\dots,f_n'),\quad
\mathbf{f}'' = (f_1'',\dots,f_n'') \in \mathbb{Z}^n
\end{align*}such that 
\begin{itemize}
\item for $i=1,\dots, n$, $f_i + f_i' + f_i'' = 1$ and
\item for {\textbf {\textit any}} system of simple closed curves $\boldsymbol \alpha = (\alpha_1,\dots,\alpha_k)$, the $i$-th entry of the vector 
$$G\cdot \mathbf{f}^T + G' \cdot {\mathbf{f}'}^T + G'' \cdot {\mathbf{f}''}^T $$  is equal to $2$ for $i=1,\dots, n-k$ and is equal to $0$ for $i=n-k+1,\dots, n$.
\end{itemize}
\end{definition}
\begin{remark}
By \cite[Lemma 6.1]{N}, strong combinatorial flattening exists for any ideal triangulation.
\end{remark}

For computation purposes, we introduce the following definition of generalized strong combinatorial flattening, which allows $f_i, f_i'$ and $f_i''$ in Definition \ref{SCF} to be half-integers.
\begin{definition}\label{GSCF}
A generalized strong combinatorial flattening consists of three vectors 
\begin{align*}
\mathbf{f} = (f_1,\dots,f_n),\quad
\mathbf{f'} = (f_1',\dots,f_n'),\quad
\mathbf{f}'' = (f_1'',\dots,f_n'') \in \Big(\mathbb{Z} \cup \Big(\mathbb{Z}+\frac{1}{2}\Big)\Big)^n
\end{align*}
such that 
\begin{itemize}
\item for $i=1,\dots, n$, $f_i + f_i' + f_i'' = 1$ and
\item for {\textbf {\textit any}} system of simple closed curves $\alpha = (\alpha_1,\dots,\alpha_k)$, the $i$-th entry of the vector 
$$G\cdot \mathbf{f}^T + G' \cdot {\mathbf{f}'}^T + G'' \cdot {\mathbf{f}''}^T $$  is equal to $2$ for $i=1,\dots, n-k$ and is equal to $0$ for $i=n-k+1,\dots, n$.
\end{itemize}
\end{definition}
\begin{remark}
To verify the second condition in Definitions \ref{SCF} and \ref{GSCF}, since the condition is linear in $\alpha_k$, it suffices to verify the condition on a basis of $\pi_1(\partial M)$.
\end{remark}

\subsection{1-loop invariant and torsion as rational functions on the gluing variety}
\begin{definition}\label{defnrhoregular}
Let $M$ be a hyperbolic 3-manifold with toroidal boundary and let $\rho : \pi_1(M) \to \mathrm{PSL}(2;\CC)$ be a representation. An ideal triangulation $\mathcal{T}$ of $M$ is $\rho$-regular if there exists $\mathbf{z} \in \mathcal{V}(\mathcal{T})$ such that $\mathcal{P}_{\mathcal{T}}(\mathbf z) = [\rho]$.
\end{definition}
\begin{definition}[\cite{DG}] Let $\rho : \pi_1(M) \to \mathrm{PSL}(2;\CC)$ be an irreducible representation. 
Suppose $\mathcal{T}$ is a $\rho$-regular ideal triangulation with $\mathcal{P}_{\mathcal{T}}(\mathbf{z})=[\rho]$. Then the 1-loop invariant of $(M,\boldsymbol\alpha, \mathbf{z}, \mathcal{T})$ is defined by
$$\tau(M,\boldsymbol\alpha, \mathbf{z}, \mathcal{T}) = 
\pm \frac{1}{2} \mathrm{det}\Big((G - G') \Delta_{\mathbf{z}''} + (G'' - G') \Delta_{\mathbf{z}}^{-1}\Big) \prod_{i=1}^n \Big(z_i^{f_i''} z_i''^{-f_i}\Big)
,$$
where $(\mathbf{f},\mathbf{f}',\mathbf{f}'')$ is a strong combinatorial flattening, 
$$\Delta_{\mathbf{z}} 
= \begin{pmatrix}
z_1 & 0 & 0 & \dots & 0 \\
0 & z_2 & 0 & \dots & 0 \\
\vdots & \vdots & \vdots & \vdots & \vdots \\
0 & 0 & 0 & 0 & z_n
\end{pmatrix}
\quad \text{and} \quad
\Delta_{\mathbf{z}''} 
= \begin{pmatrix}
z_1'' & 0 & 0 & \dots & 0 \\
0 & z_2'' & 0 & \dots & 0 \\
\vdots & \vdots & \vdots & \vdots & \vdots \\
0 & 0 & 0 & 0 & z_n''
\end{pmatrix}.$$
\end{definition}
\begin{remark} For any hyperbolic 3-manifold, at the discrete faithful representation, the 1-loop invariant can be defined by using the usual combinatorial flattening defined in Definition \ref{CF} and it is independent of the choice of the flattening \cite[Theorem 1.4]{DG}.
\end{remark}

For $i=1,\dots, n$, let $\xi_i = \frac{d \log z_i}{d z_i} = \frac{1}{z_i}$, $\xi_i' = \frac{d \log z_i'}{d z_i} = \frac{1}{1-z_i}$ and $\xi_i''=\frac{d \log z_i''}{d z_i} = \frac{1}{z_i(z_i-1)}$. Let 
\begin{align*}
\boldsymbol{\xi} &= 
\begin{pmatrix}
\xi_1\\
\vdots \\
\xi_n
\end{pmatrix} ,\quad
\boldsymbol{\xi}' = 
\begin{pmatrix}
\xi_1'\\
\vdots \\
\xi_n'
\end{pmatrix}\quad \text{and}\quad
\boldsymbol{\xi}'' = 
\begin{pmatrix}
\xi_1''\\
\vdots \\
\xi_n''
\end{pmatrix}.
\end{align*}
The following symmetric formula of the 1-loop invariant was first studied by Siejakowski in \cite{S1,S2}. In particular, the formula provides a geometric interpretation of the adjoint twisted Reidemeister torsion in terms of the determinant of the Jacobian of the gluing equations of an $\rho$-regular ideal triangulation. For reader's convenience, we include a proof of the result. 
\begin{proposition}\label{1loopgluingjacobian}[Section 5.1.3, \cite{S1}]
Suppose $\mathcal{T}$ is $\rho$-regular with $\mathcal{P}_{\mathcal{T}} (\mathbf{z}) = [\rho]$. Then
$$\tau(M,\boldsymbol\alpha, \mathbf{z}, \mathcal{T}) = 
\pm \frac{1}{2}\frac{\det \Big(G \Delta_{\boldsymbol\xi} + G' \Delta_{\boldsymbol\xi'} + G'' \Delta_{\boldsymbol\xi''} \Big)}{\prod_{i=1}^n \xi_i^{f_i} \xi_i'^{f_i'} \xi_i''^{f_i''}}
= 
\pm \frac{1}{2}\frac{\mathrm{det}(D_{\mathbf z}F(\mathbf z))}{\prod_{i=1}^n \xi_i^{f_i} \xi_i'^{f_i'} \xi_i''^{f_i''}},
$$
where $F$ is defined in (\ref{defF1loop}) and $D_{\mathbf z}F$ is the Jacobian matrix of $F$.
\end{proposition}
\begin{proof}
Recall that $z'' = 1-\frac{1}{z}$. Note that
\begin{align*}
(G - G') \Delta_{\mathbf{z}''} + (G'' - G') \Delta_{\mathbf{z}}^{-1}
&= G \Delta_{\mathbf{\frac{z-1}{z}}} - G' + G'' \Delta_{\mathbf{\frac{1}{z}}} \\
&= \Big(G \Delta_{\mathbf{\frac{1}{z}}} + G' \Delta_{\mathbf{\frac{1}{1-z}}} + G'' \Delta_{\mathbf{\frac{1}{z(z-1)}}} \Big) \Delta_{\mathbf{z-1}}\\
&= \Big(G \Delta_{\boldsymbol{\xi}} + G' \Delta_{\boldsymbol{\xi}'} + G'' \Delta_{\boldsymbol{\xi}''} \Big) \Delta_{\mathbf{z-1}}.
\end{align*}
Thus, we have
\begin{align}\label{for1}
\det\Big((G - G') \Delta_{\mathbf{z}''} + (G'' - G') \Delta_{\mathbf{z}}^{-1}\Big)
&= \det \Big(G \Delta_{\boldsymbol\xi} + G' \Delta_{\boldsymbol\xi'} + G'' \Delta_{\boldsymbol\xi''} \Big)
\prod_{i=1}^n (z_i - 1).
\end{align}
Besides,
\begin{align}\label{for2}
\frac{1}{\prod_{i=1}^n \xi_i^{f_i} \xi_i'^{f_i'} \xi_i''^{f_i''}} \notag
&= \prod_{i=1}^n z_i^{f_i} (1-z_i)^{f_i'} (z_i (z_i-1))^{f_i''} \notag\\
&=  \pm \prod_{i=1}^n z_i^{f_i+f_i''} (1-z_i)^{f_i' + f_i''} \notag\\
&=  \pm \prod_{i=1}^n z_i^{f_i+f_i''} (1-z_i)^{1-f_i} \notag\\
&= \pm \prod_{i=1}^n (1-z_i) \prod_{i=1}^n z_i^{f_i''} z_i''^{-f_i} ,
\end{align}
where in the second last equality we use the property of the combinatorial flattening that $f_i + f_i' + f_i'' = 1$.
Altogether, 
\begin{align*}
&\pm \frac{1}{2}\frac{\det \Big(G \Delta_{\boldsymbol\xi} + G' \Delta_{\boldsymbol\xi'} + G'' \Delta_{\boldsymbol\xi''} \Big)}{\prod_{i=1}^n \xi_i^{f_i} \xi_i'^{f_i'} \xi_i''^{f_i''}} \\
=& \pm \frac{1}{2}\frac{\det\Big((G - G') \Delta_{\mathbf{z}''} + (G'' - G') \Delta_{\mathbf{z}}^{-1}\Big)}{\prod_{i=1}^n (z_i - 1)} \Big(\prod_{i=1}^n (1-z_i) \prod_{i=1}^n z_i^{f_i''} z_i''^{-f_i} \Big) \\
=& \pm \frac{1}{2} \mathrm{det}\Big((G - G') \Delta_{\mathbf{z}''} + (G'' - G') \Delta_{\mathbf{z}}^{-1}\Big) \prod_{i=1}^n \Big(z_i^{f_i''} z_i''^{-f_i}\Big).
\end{align*}
This proves the first equality. The second equality follows from (\ref{defF1loop}) and the fact that $\xi_i = \frac{d \log z_i}{d z_i} $, $\xi_i' = \frac{d \log z_i'}{d z_i} $ and $\xi_i''=\frac{d \log z_i''}{d z_i}$.
\end{proof}
\begin{remark}\label{edgeid}
Recall that for any shape parameter $z\in \CC\setminus\{0,1\}$, we have $\log z + \log z' + \log z'' = \pi \sqrt{-1}$. In particular, we have $\xi+\xi'+\xi'' = \frac{d}{dz}(\log z + \log z' + \log z'') = 0$. Thus, when we compute the 1-loop invariant by taking the derivative of the gluing equations, we can apply the equation $\log z + \log z' + \log z'' = \pi \sqrt{-1}$ and ignore the term $\pi\sqrt{-1}$ without affecting the final result.
\end{remark}

To compare the 1-loop invariant with the adjoint twisted Reidemeister torsion discussed in Section \ref{TRT}, it is more convenient to consider the torsion as a function on the gluing variety. More precisely, given a point $\mathbf{z}$ in the gluing variety $\mathcal{V}_{\mathcal{T}}$, we let $\rho_{\mathbf{z}}:\pi_1(M) \to \mathrm{PSL}(2;\CC)$ be the associated pseudo-developing map. 
\begin{proposition}\label{toranaglu} 
The function sending $\mathbf{z} \in \mathcal{V}_{\mathcal{T}}$ to $\mathbb T_{(M,\boldsymbol\alpha)}(\rho_{\mathbf{z}})$ is locally a holomorphic function on the smooth points of $\mathcal{V}$.
\end{proposition}
\begin{proof}
For any $\gamma \in \pi_1(M)$, the entries of $\rho_{\mathbf{z}}(\gamma)$ can be written as a rational function in terms of the shape parameters $\mathbf{z}$ \cite[Section 4.0.8]{C}. Furthermore, the adjoint twisted Reidemeister torsion of is a holomorphic function in terms of the entries of the image of $\rho_{\mathbf{z}}$. Since the shape parameters are coordinate functions of the gluing variety, we have the desired result.
\end{proof}

\begin{proposition}\label{biratiso}
Let $\mathcal{T}_1$ and $\mathcal{T}_2$ be two $\rho_0$-regular ideal triangulations with $\mathcal{P}_{\mathcal{T}_1}(\mathbf{z_0^1}) = \mathcal{P}_{\mathcal{T}_2}(\mathbf{z_0^2}) = [\rho_0]$ for some $\mathbf{z_0^1} \in \mathcal{V}_{\mathcal{T}_1}$ and $\mathbf{z_0^2} \in \mathcal{V}_{\mathcal{T}_2}$. Then there exist birational maps between $\mathcal{V}_0(\mathcal{T}_1)$ and $\mathcal{V}_0(\mathcal{T}_2)$ that send $\mathbf{z_0^1}$ to $\mathbf{z_0^2}$ and vice versa.  In particular, these induce local biholomorphisms between a Euclidean neighborhood of $\mathbf{z_0^1} \in \mathcal{V}_0(\mathcal{T}_1)$ and a Euclidean neighborhood of $\mathbf{z_0^2} \in \mathcal{V}_0(\mathcal{T}_2)$.
\end{proposition}
\begin{proof}
By \cite[Corollary 1.2]{KSS}, $\mathcal{T}_1$ and $\mathcal{T}_2$ are connected by a finite sequence of $\rho_0$-regular ideal triangulations through 0-2, 2-0, 2-3 and 3-2 moves. There is a rational map relating the shape parameters before and after each of those moves. In particular, there exists a rational map $f$ that sends $\mathbf{z_0^1}$ to $\mathbf{z_0^2}$ and send every $\mathbf{z^1} \in \mathcal{V}_0(\mathcal{T}_1)$ sufficiently close to $\mathbf{z_0^1}$ to some $\mathbf{z^2} \in \mathcal{V}_0(\mathcal{T}_2)$. The domain of this map can be extended to a Zariski open subset of $ \mathcal{V}_0(\mathcal{T}_1)$ where the denominators of all the rational functions involved are non-zero and all the shape parameters are non-degenerate. By switching the role of $\mathcal{T}_1$ and $\mathcal{T}_2$ and using the above sequence of moves in reverse, one obtains the inverse of $f$ on a Zariski open subset of $\mathcal{V}_0(\mathcal{T}_2)$ containing $\mathbf{z_0^2}$. This completes the proof.
\end{proof}
The following result is known under the assumption that all shape parameters have non-negative imaginary parts \cite[Corollary 15.2.17]{BM}. We prove that the result is also true in general.
\begin{proposition}\label{z0smpt}
Let $\mathcal{T}$ be a $\rho_0$-regular ideal triangulation with $\mathcal{P}_{\mathcal{T}}(\mathbf{z_0}) = [\rho_0]$. Then $\mathbf{z}_0$ is a smooth point of $\mathcal{V}(\mathcal{T})$.
\end{proposition}
\begin{proof}
It is known that there exists some ideal triangulation $\mathcal{T}_{EP}$ coming from the Epstein-Penner decomposition of the manifold such that the proposition is true for $\mathcal{T}_{EP}$ \cite[Corollary 15.2.17]{BM}. By Proposition \ref{biratiso}, near $\mathbf{z}_0 \in \mathcal{T}$, there exists a local biholomorphism from $\mathcal{V}(\mathcal{T})$ and $\mathcal{V}(\mathcal{T}_{EP})$. This gives the desired result.
\end{proof}
By Proposition \ref{toranaglu} and \ref{z0smpt}, we can consider the torsion as a holomorphic function defined on the smooth point of the irreducible component $\mathcal{V}_0(\mathcal{T}) \subset \mathcal{V}(\mathcal{T})$ containing $\mathbf{z}_0$.

\section{Triangulation of fundamental shadow link complements}\label{triFSL}
Given a fundamental shadow link $L_{\text{FSL}} \subset M_c$, where $c\in \NN$, to obtain an ideal triangulation of $M_c \setminus L_{\text{FSL}}$, we shrink the six edges of each $D$-block into six ideal vertices to obtain the union of two ideal octahedra $\mathcal{O}$ and $\tilde{\mathcal{O}}$, where $\mathcal{O}$ and $\tilde{\mathcal{O}}$ are glued together along four pair of faces as shown in Figure \ref{idealocta}. Then we glue the faces of the $D$-blocks together according to the construction of the fundamental shadow link. This gives a decomposition of $M_c \setminus L_{\text{FSL}}$ into $2c$ ideal octahedra. 
Note that around each ideal vertex of the $D$-block, the truncated rectangles are glued together to form a cylinder (Figure \ref{idealocta}, right). The boundary of the tubular neighborhood of each component of $L_{\text{FSL}}$ is obtained by gluing the cylinders together. 
To obtain an ideal triangulation, for $i=1,\dots, c$, we cut the $i$-th $D$-block, which consists of two ideal octahedra $\mathcal{O}_i$ and $\mathcal{O}_i'$, into eight ideal tetrahedra and assign shape parameters $\{z_{i,1},\dots, z_{i,4}, \tilde z_{i,1},\dots, \tilde z_{i,4}\}$ as shown in Figure \ref{idealoctatri}. This ideal triangulation of $M_c\setminus L_{\text{FSL}}$ induces triangulations of the boundary cylinders as shown in Figures \ref{trib14}, \ref{trib25} and \ref{trib36}. Note that at the complete hyperbolic structure, all shape parameters are equal to $\sqrt{-1}$ and all the ideal octahedra are regular (i.e. every dihedral angles are $\pi/2$). In particular, this triangulation is geometric.

\begin{figure}[h]
\centering
\includegraphics[scale=0.135]{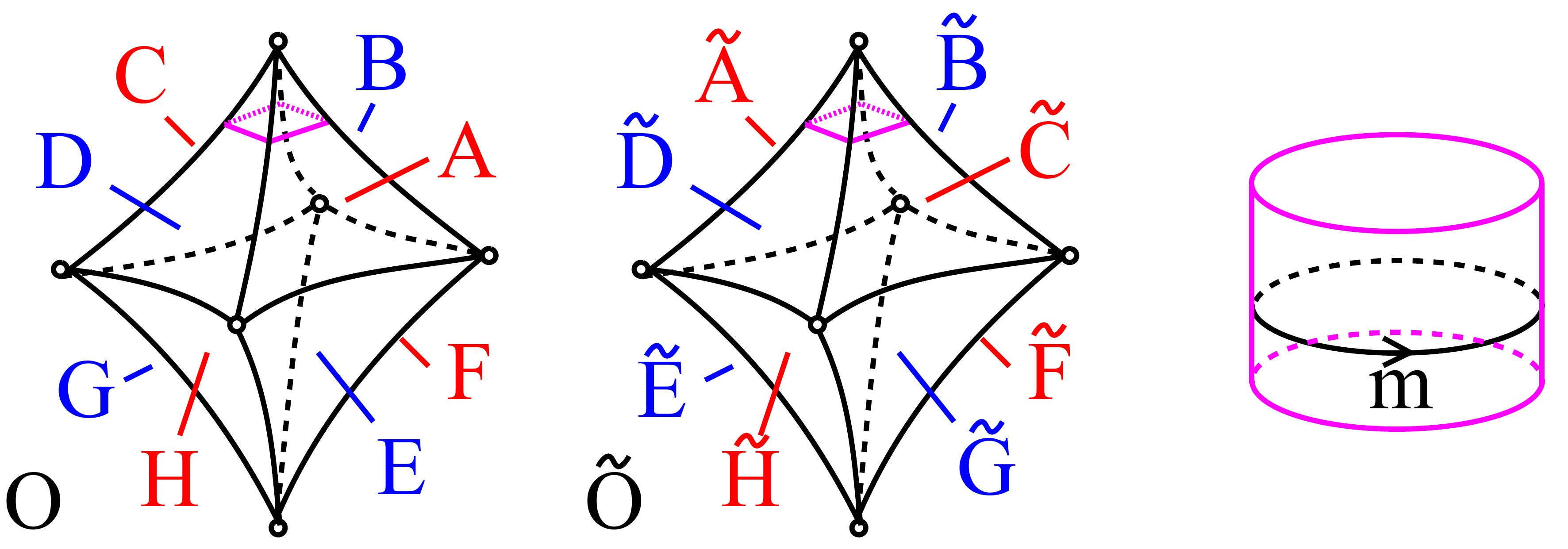}
\caption{The ideal octahedron $O$ is obtained by shrinking the 6 edges of a truncated tetrahedron into 6 ideal vertices. $\tilde{O}$ is another copy of $O$ and it is glued to $O$ along the blue faces by identifying $B$ with $\tilde{B}$, $D$ with $\tilde{D}$, $E$ with $\tilde{E}$ and $G$ with $\tilde{G}$. Especially, around each ideal vertex, such as the top one, a pair of opposite edges of the truncated rectangle (in purple) in $O$ is glued to another pair of opposite edges of the truncated rectangle in $\tilde{O}$ to form a cylinder (right).}\label{idealocta}
\end{figure}

\begin{figure}[h]
\centering
\includegraphics[scale=0.145]{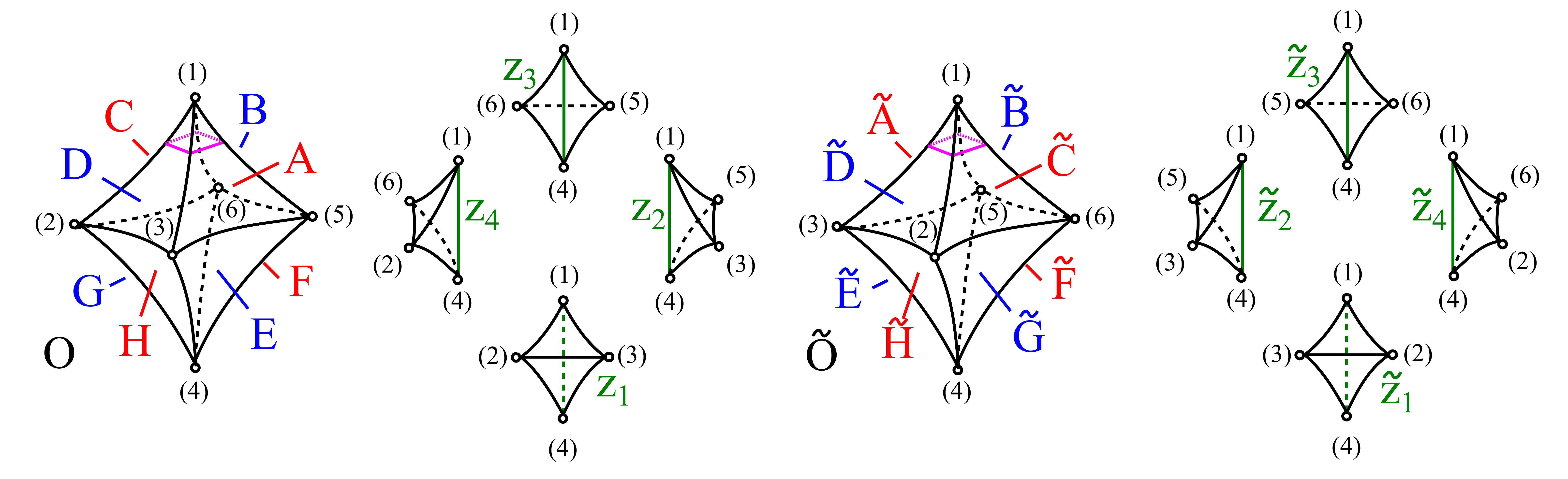}
\caption{Triangulation of each $D$-block into 8 ideal tetrahedra. The shape parameters of the green edges are shown in the Figure.}\label{idealoctatri}
\end{figure}

\begin{figure}[h]
\centering
\includegraphics[scale=0.12]{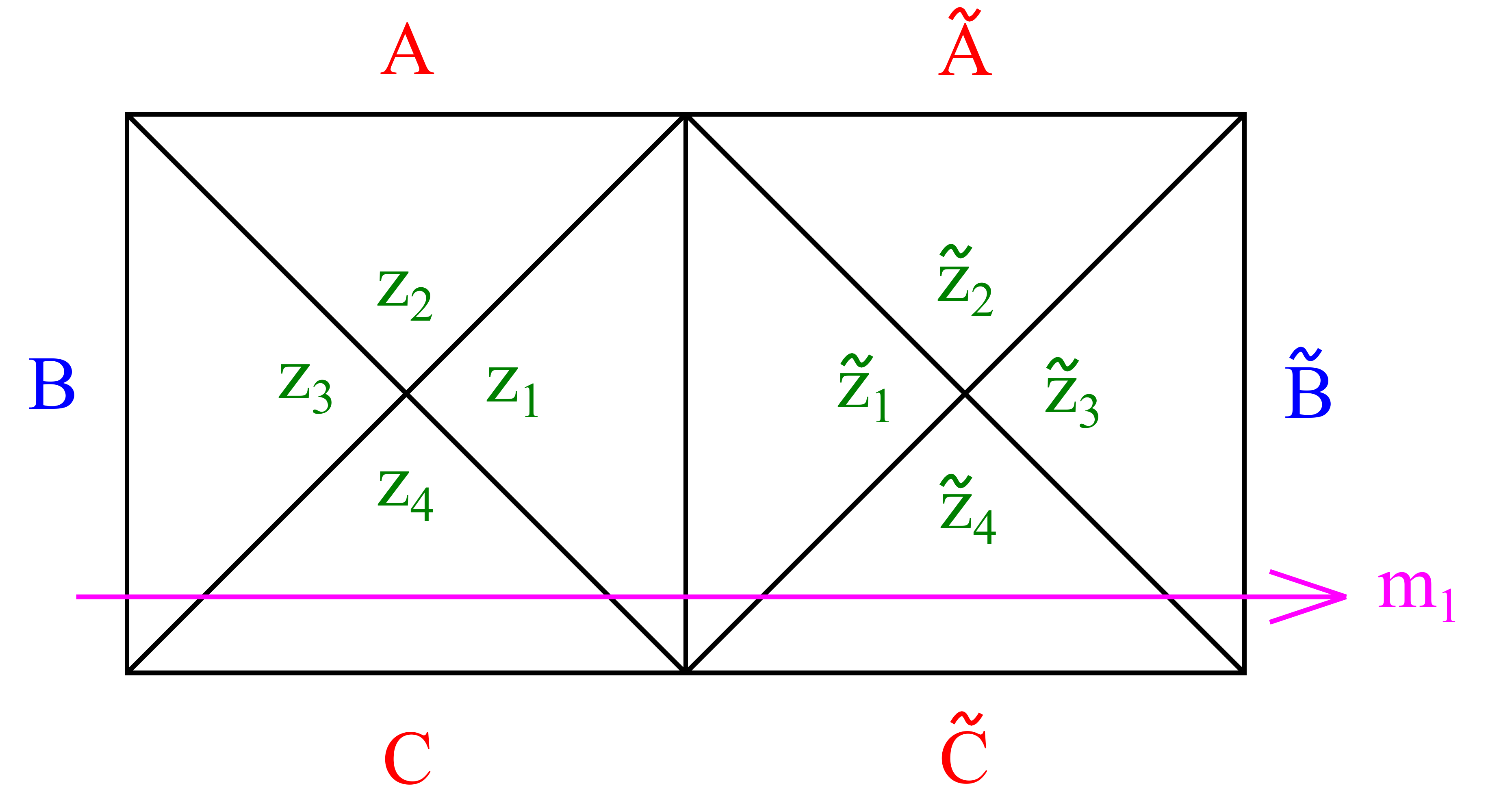}\qquad
\includegraphics[scale=0.12]{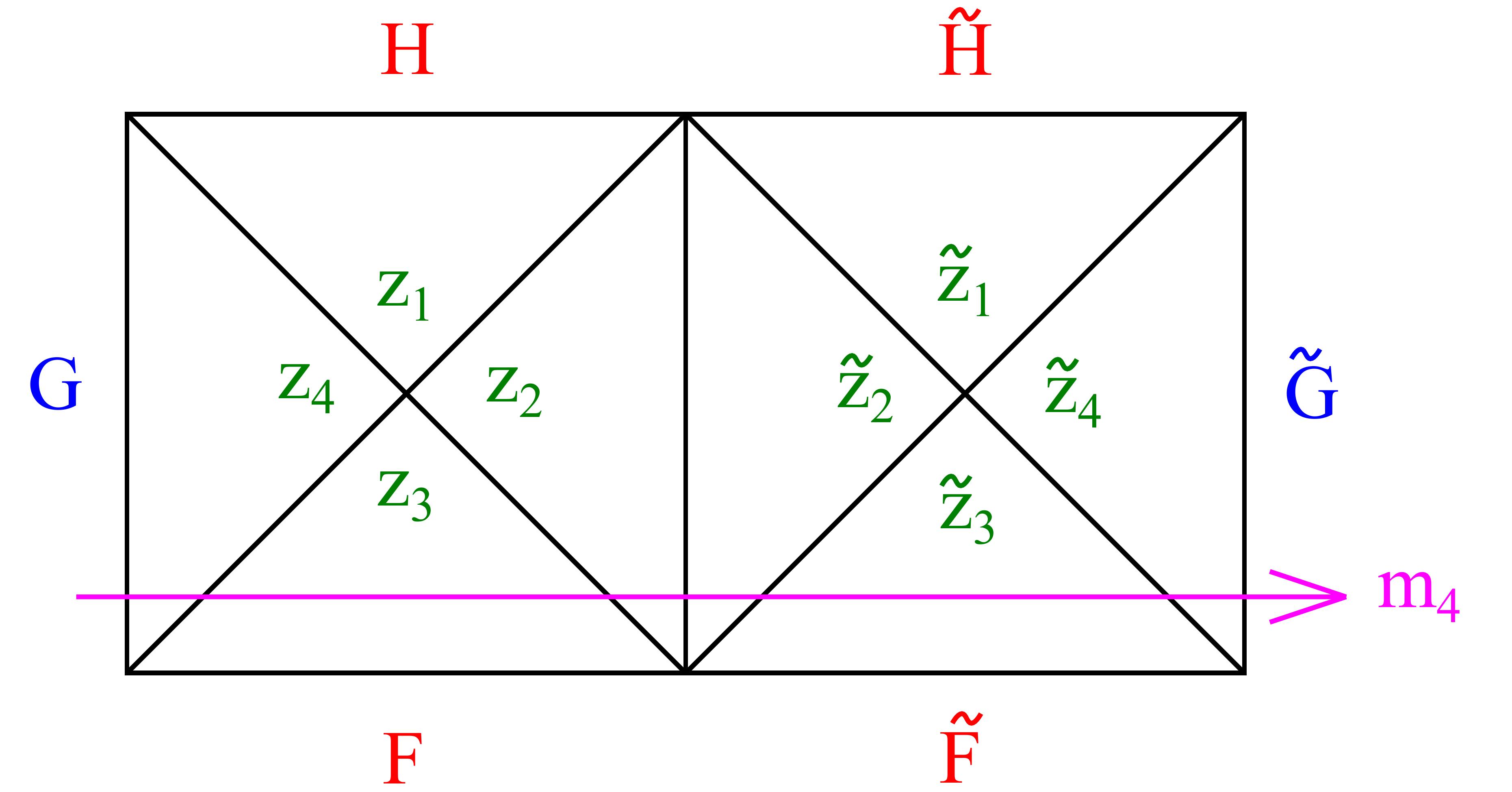}
\caption{Triangulations of vertices (1) and (4)}\label{trib14}
\end{figure}
\begin{figure}[h]
\centering
\includegraphics[scale=0.12]{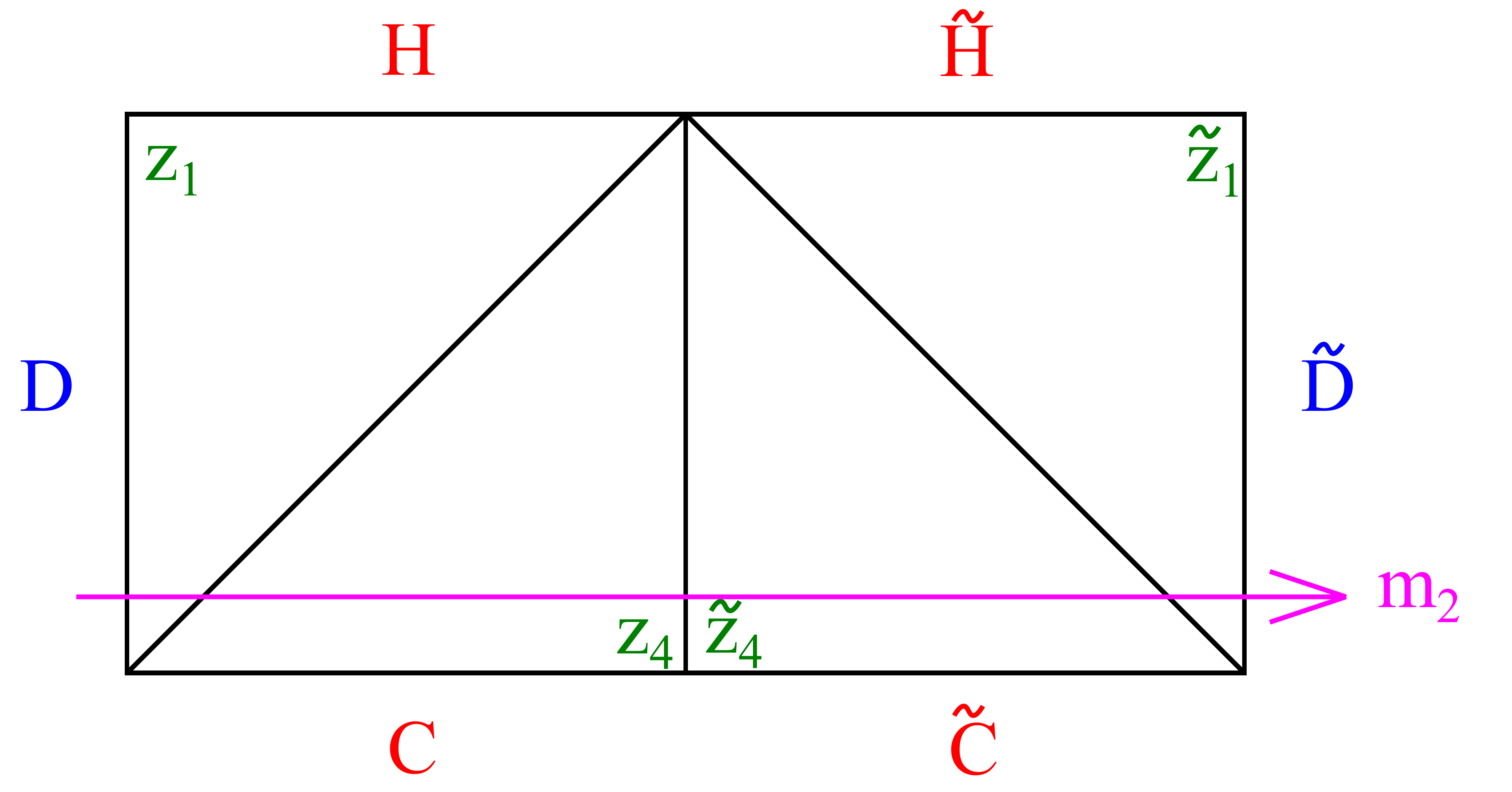}\qquad
\includegraphics[scale=0.12]{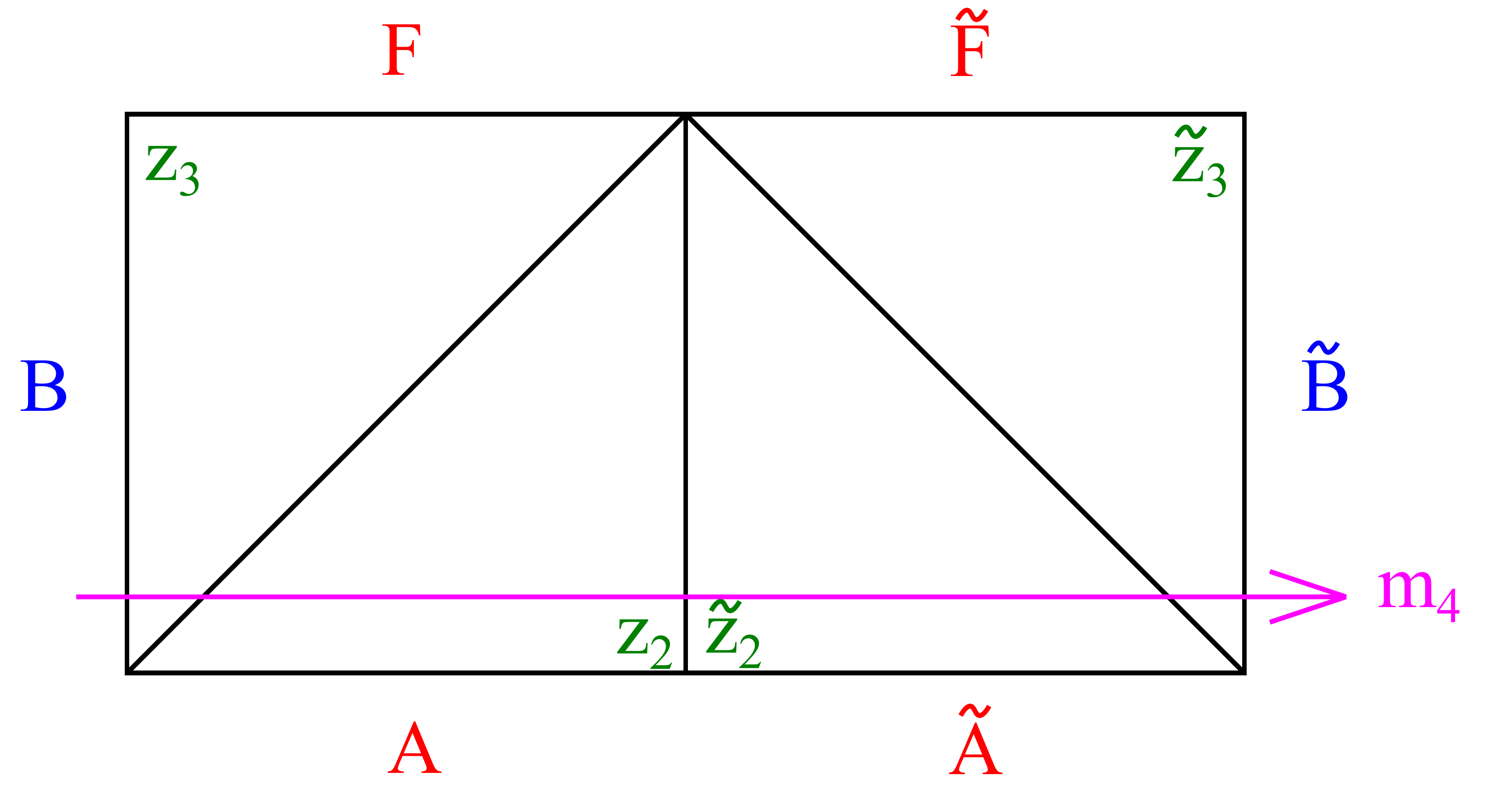}
\caption{Triangulations of vertices (2) and (5)}\label{trib25}
\end{figure}
\begin{figure}[h]
\centering
\includegraphics[scale=0.12]{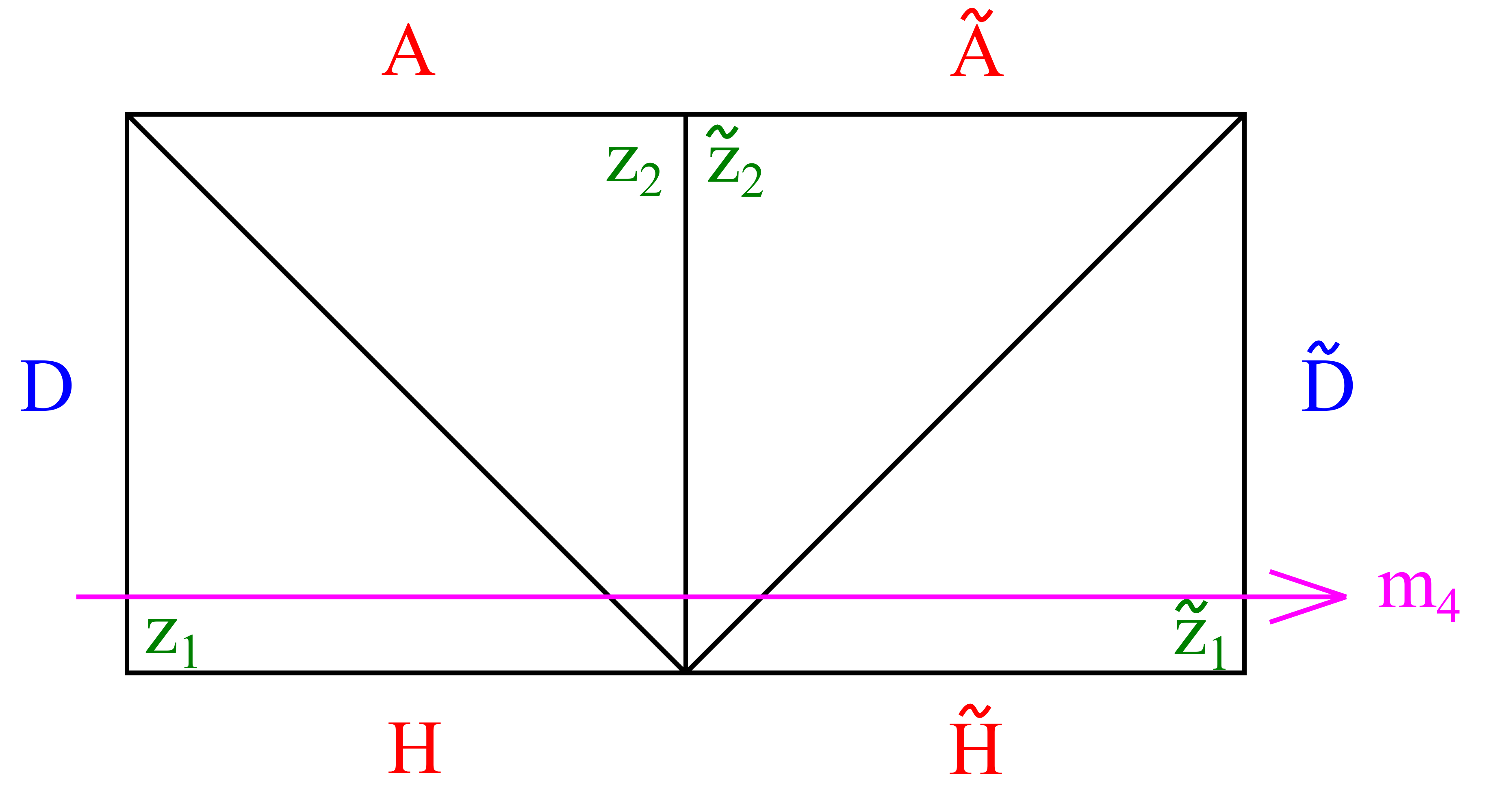}\qquad
\includegraphics[scale=0.12]{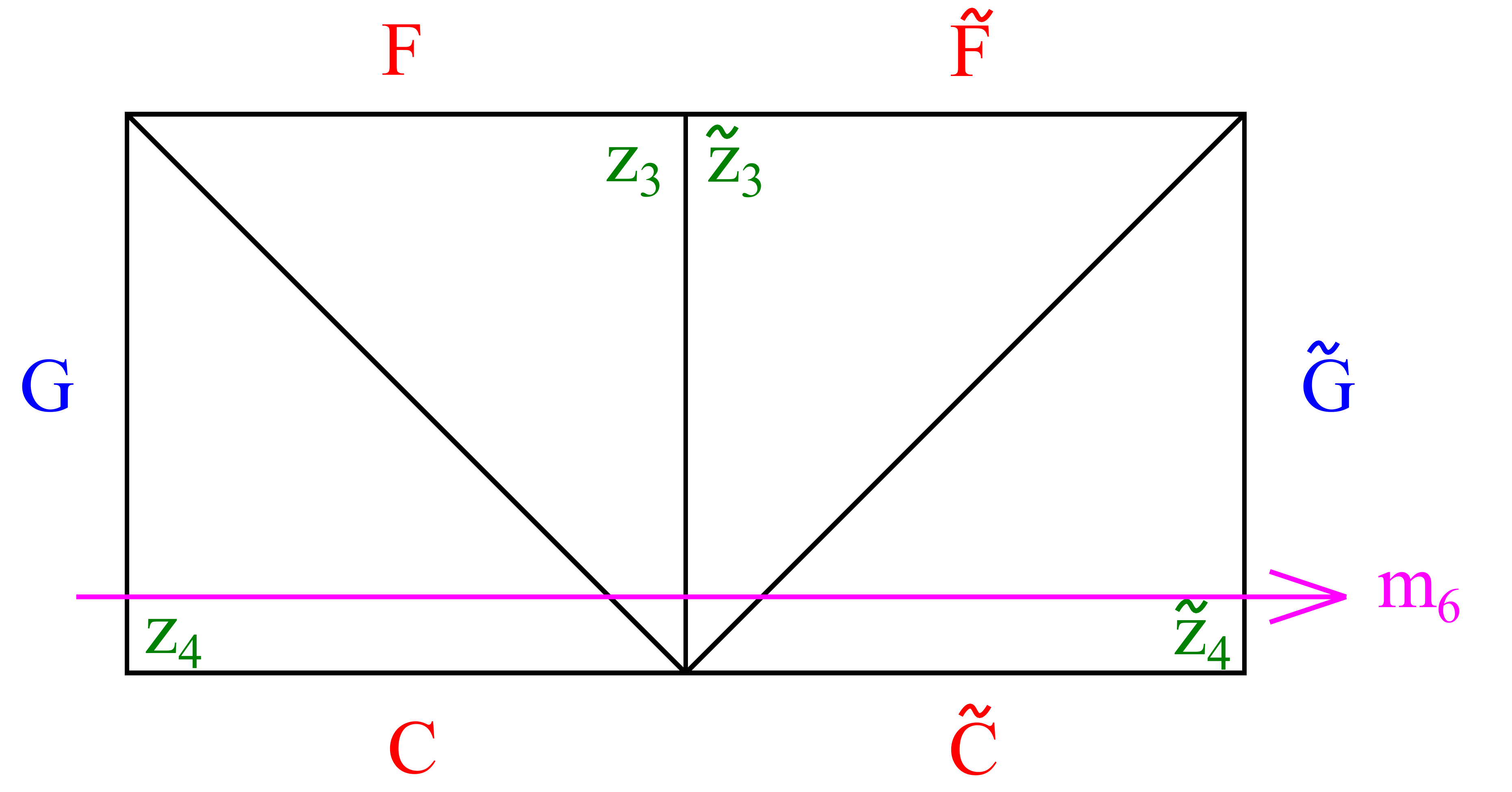}
\caption{Triangulations of vertices (3) and (6)}\label{trib36}
\end{figure}

\subsection{Gluing equations on a $D$-block}\label{GEDB}
We first study the gluing equations on each individual $D$-block. The main goal of this subsection is to prove Proposition \ref{explictsolnDB}, which provides an explicit solution of the gluing equations in terms of the holonomies of the six ideal vertices of the $D$-block. 

Given a $D$-block $\mathcal{D}$ consisting of ideal octahedra $\mathcal{O}$ and $\tilde{\mathcal{O}}$, we let $m_{1},\dots,m_{6}$ be the meridian of the cylinder at the $i$-th ideal vertex as shown in Figures \ref{trib14}, \ref{trib25} and \ref{trib36}. For $\mathbf{z} = (z_1,\dots, z_4, \tilde{z_1}\dots, \tilde{z_4})\in \CC\setminus\{0,1\}^8$, let $F_{e_1}, F_{e_2}, F_{m_1}, \dots, F_{m_6} : \CC\setminus\{0,1\}^8 \to \CC$ be the functions defined by
\begin{align*}
F_{e_1}(\mathbf{z} ) &= \log z_1 + \log z_2 + \log z_3 + \log z_4, \\
F_{e_2}(\mathbf{z} )  &= \log \tilde z_1 + \log \tilde z_2 + \log \tilde z_3 + \log \tilde z_4,\\
F_{m_1}(\mathbf{z} )  &= \log z_4 + \log \tilde z_4 - \log z_3'' - \log z_1' - \log \tilde z_1'' - \log \tilde z_3', \\
F_{m_2}(\mathbf{z} )  &= \log z_4' + \log \tilde z_4'' - \log z_1' - \log \tilde z_1'',\\
F_{m_3}(\mathbf{z} )  &= \log z_1'' + \log \tilde z_1' - \log z_2'' - \log \tilde z_2',\\
F_{m_4}(\mathbf{z} )  &= \log z_3 + \log \tilde z_3 - \log z_4'' - \log z_2' - \log \tilde z_2'' - \log \tilde z_4',\\
F_{m_5}(\mathbf{z} )  &= \log z_2' + \log \tilde z_2'' - \log z_3' - \log \tilde z_3'',\\
F_{m_6}(\mathbf{z} )  &= \log z_4'' + \log \tilde z_4' - \log z_3'' - \log \tilde z_3'.
\end{align*}
For each $\mathbf{H(m)}=\big(\mathrm{H}(m_1),\dots, \mathrm{H}(m_6)\big)\in \CC^6$, 
let $\mathcal{G}_{\mathcal{D}}: \CC^6 \times (\CC\setminus\{0,1\})^8 \to \CC^8$ be the function defined by
\begin{align}\label{defFDB}
&\mathcal{G}_{\mathcal{D}}( \mathbf{H(m)}, \mathbf{z} ) \notag\\
=&\Big(F_{e_1}(\mathbf{z} ) - 2\pi \sqrt{-1}, F_{e_2}(\mathbf{z} )  - 2\pi \sqrt{-1}, F_{m_1}(\mathbf{z} ) - \mathrm{H}(m_1), \dots, F_{m_6}(\mathbf{z} ) - \mathrm{H}(m_6)\Big).
\end{align}
Note that by direct computation, we have $\mathcal{G}_{\mathcal{D}}( (0,\dots,0), (\sqrt{-1},\dots, \sqrt{-1})) = (0,\dots,0)$. Geometrically, it means that when $\mathbf{H(m)} = (0,\dots,0)$, the octahedron becomes a regular ideal tetrahedron which is decomposed into four ideal tetrahedra with dihedral angles $\pi/2, \pi/4, \pi/4$. The following proposition shows that a solution to the equation $\mathcal{G}_{\mathcal{D}} = 0$ exists locally around $\mathbf{H(m)} = (0,\dots,0)$.

\begin{proposition}\label{solnexist} There exists a neighborhood $U \subset \CC^6$ containing the origin such that for any \linebreak $\mathbf{H(m)} \in U$, there exists a unique $\mathbf{z^*(H(m))}$ with $\mathbf{z^*}(0,\dots, 0) = \big(\sqrt{-1},\dots, \sqrt{-1}\big)$ such that \linebreak$\mathcal{G}_{\mathcal{D}}(\mathbf{H(m)}, \mathbf{z}^*(\mathbf{H(m)}) ) = (0,\dots,0)$.
\end{proposition}
\begin{proof}
Note that the Jacobian matrix of $\mathcal{G}_{\mathcal{D}}$ with respect to $\mathbf{z}$ is given by
\begin{align*}
D_{\mathbf{z}}\mathcal{G}_{\mathcal{D}}
=
\begin{pmatrix}
\frac{1}{z_1} & \frac{1}{z_2} & \frac{1}{z_3} & \frac{1}{z_4} & 0 & 0 & 0 & 0 \\
0 & 0 & 0 & 0 & \frac{1}{\tilde{z_1}} & \frac{1}{\tilde{z_2}} & \frac{1}{\tilde{z_3}} & \frac{1}{\tilde{z_4}} \\
\frac{-1}{1-z_1} & 0 & \frac{-1}{z_3(z_3-1)} & \frac{1}{z_4} & \frac{-1}{\tilde{z_1}(\tilde{z_1} - 1)} & 0 & \frac{-1}{1-\tilde{z_3}} & \frac{1}{\tilde{z_4}} \\
\frac{-1}{1-z_1} & 0 & 0 & \frac{1}{1-z_4} & \frac{-1}{\tilde{z_1}(\tilde{z_1} - 1)} & 0 & 0 & \frac{1}{\tilde{z_4}(\tilde{z_4}-1)} \\
\frac{1}{z_1(z_1-1)} & \frac{-1}{z_2(z_2-1)} & 0 & 0 & \frac{1}{1-\tilde{z_1}} & \frac{-1}{1-\tilde{z_2}} & 0 & 0 \\
0 & \frac{-1}{1-z_2} & \frac{1}{z_3} & \frac{-1}{z_4(z_4-1)} & 0 & \frac{-1}{\tilde{z_2}(\tilde{z_2}-1)} & \frac{1}{\tilde{z_3}} & \frac{-1}{1-\tilde{z_4}}\\
0 & \frac{1}{1-z_2} & \frac{-1}{1-z_3} & 0 & 0 & \frac{1}{\tilde{z_2}(\tilde{z_2}-1)} & \frac{-1}{\tilde{z_3}(\tilde{z_3}-1)} & 0 \\
0 & 0 & \frac{-1}{z_3(z_3-1)} & \frac{1}{z_4(z_4-1)} & 0 & 0 & \frac{-1}{1-\tilde{z_3}} & \frac{1}{1-\tilde{z_4}} 
\end{pmatrix}
\end{align*}
with 
\begin{align}\label{detcomplete}
\mathrm{det} \big(D_{\mathbf{z}}\mathcal{G}_{\mathcal{D}} \big( (0,\dots,0), \big(\sqrt{-1},\dots, \sqrt{-1}\big)\big) \big) = -32 \sqrt{-1} \neq 0.
\end{align} 
The result follows from the implicit function theorem.
\end{proof}

Next, we are going to solve the equation $\mathcal{G}_{\mathcal{D}}( \mathbf{H(m)}, \mathbf{z} ) = (0,\dots, 0)$ explicitly. The key idea is to apply a change of variables that convert the 8 non-linear gluing equations into 7 linear equations together with 1 non-linear equation. This simplifies the computation and allows us to reduce the problem of solving the gluing equations into the problem of solving a single quadratic equation that has a close relationship with the Gram matrix of the $D$-block. We will further discuss the change of variables later in Section \ref{compEB}.

More precisely, let $u_l = e^{\frac{\mathrm{H}(m_l)}{2}}$ for $l=1,\dots,6$ and  
$$ z^* = \frac{-B - \sqrt{B^2 - 4AC}}{2A},$$
where 
\begin{equation}\label{defABC}
\begin{split}
A&= - \frac{u_1}{u_4} - \frac{u_1 u_3}{u_2} - \frac{u_1}{u_2u_3} - \frac{u_1}{u_2^2 u_4} - \frac{u_5}{u_2} - \frac{u_6}{u_2u_4} - \frac{1}{u_2u_4u_6} - \frac{1}{u_2u_5},\\
B&= - u_1u_4 + \frac{u_1}{u_4} + \frac{u_4}{u_1} - \frac{1}{u_1u_4}
 + u_2u_5 + \frac{u_2}{u_5} + \frac{u_5}{u_2} + \frac{1}{u_2u_5}
 - u_3 u_6 - \frac{u_3}{u_6} - \frac{u_6}{u_3} - \frac{1}{u_3u_6}  ,\\
C&= - \frac{u_4}{u_1} - \frac{u_2}{u_1 u_3} - \frac{u_2u_3}{u_1} - \frac{u_2^2 u_4}{u_1} - \frac{u_2}{u_5} - \frac{u_2u_4}{u_6} - u_2u_4u_6 - u_2u_5.
\end{split}
\end{equation}
By a direct computation, one can verify that
\begin{align}\label{disgram}
B^2 - 4AC = 16 \det \mathbb{G},
\end{align}
where $\det \mathbb{G}$ is the determinant of the associated Gram matrix defined in Section \ref{TFSL}.

\begin{proposition}\label{explictsolnDB}
The solution $\mathbf{z^*(H(m))}$ to the equation $\mathcal{G}_{\mathcal{D}}( \mathbf{H(m)}, \mathbf{z} ) = (0,\dots, 0)$ in Proposition \ref{solnexist} is given by $\mathbf{z^*(H(m))} = (z_1^*,\dots, z_4^*, \tilde{z}_1^*,\dots, \tilde{z}_4^*)$, where
\begin{empheq}[left = \empheqlbrace]{equation}\label{solnu}
\begin{split}
z_1^*&=\frac{z^*-u_2^2}{z^*+u_2u_3u_4}, \quad
z_2^*=\frac{z^*u_1u_3u_5 - u_2u_3u_4}{z^* u_1u_3u_5 + u_1u_2u_4u_5},  \\
z_3^*&=\frac{z^*u_1u_6 - u_2u_4u_5u_6}{z^*u_1u_6+u_2}, \quad
z_4^*=\frac{z^* u_1 - u_1}{z^* u_1 + u_2 u_6} \\
\tilde z_1^*&=-\frac{z^*u_2 + u_2^2 u_3 u_4}{z^* u_3 u_4 - u_2^2 u_3 u_4 }, \quad
\tilde z_2^*=-\frac{z^*u_3 + u_2 u_4}{z^* u_1u_5- u_2u_4}, \\
\tilde z_3^*&=-\frac{z^* u_1u_4u_5u_6 + u_2u_4u_5}{z^* u_1 - u_2u_4u_5} ,  \quad
\tilde z_4^* =-\frac{z^* u_1 + u_2u_6}{z^* u_2u_6 - u_2u_6}.
\end{split}
\end{empheq}
\end{proposition}
\begin{proof}
{\bf Case 1: }
First, we consider the case where $\mathrm{H}(m_l) = 2 \theta_l \sqrt{-1}$ for $\theta_l \in \RR$. Put 
\begin{empheq}[left = \empheqlbrace]{equation} \label{subs}
\begin{split}
&z_1= r_1 e^{\sqrt{-1} \phi_1}, \quad z_2= r_2 e^{\sqrt{-1} \phi_2}, \quad z_3= r_3 e^{\sqrt{-1} \phi_3}, \quad z_4= r_4 e^{\sqrt{-1} \phi_4}\\
&\tilde z_1= \frac{1}{r_1} e^{\sqrt{-1} \phi_1}, \quad \tilde z_2= \frac{1}{r_2} e^{\sqrt{-1} \phi_2}, \quad \tilde z_3= \frac{1}{r_3} e^{\sqrt{-1} \phi_3}, \quad \tilde z_4= \frac{1}{r_4} e^{\sqrt{-1} \phi_4},
\end{split}
\end{empheq}
where $r_1,\dots, r_4 \in \RR_{>0}$ and $\phi_1,\dots, \phi_4 \in \RR$. In particular, for $i=1,\dots, 4$, we have
\begin{align}\label{cong1}
\tilde z_i'' = \frac{\tilde z_i - 1}{\tilde z_i} 
= \frac{\frac{1}{r_i}e^{\sqrt{-1}\phi_i} - 1}{\frac{1}{r_i}e^{\sqrt{-1}\phi_i}}
= 1 - r_ie^{-\sqrt{-1}\phi_i} = \frac{1}{\overline {z_i'}}
\end{align}
and
\begin{align}\label{cong2}
 z_i'' = \frac{ z_i - 1}{ z_i} 
= \frac{r_ie^{\sqrt{-1}\phi_i} - 1}{r_ie^{\sqrt{-1}\phi_i}}
= 1 - \frac{1}{r_i}e^{-\sqrt{-1}\phi_i} = \frac{1}{\overline {\tilde z_i'}},
\end{align}
where $\overline{z}$ is the complex conjugate of $z$. Especially, from (\ref{cong1}) and (\ref{cong2}) we have 
$$|\tilde z_i''| = |z_i'|^{-1},\ \ |z_i''| = |\tilde z_i'|^{-1},\ \  \Arg(\tilde z_i'') = \Arg(z_i'), \ \ \Arg(\tilde z_i') = \Arg(z_i'').$$
Note that under the substitution (\ref{subs}), the equations $F_{e_1}= 2\pi \sqrt{-1}$ and $F_{e_2} = 2\pi \sqrt{-1} $ become
\begin{align}\label{f2eq}
r_1r_2r_3r_4 = 1 \quad \text{and} \quad \phi_1+\phi_2+\phi_3+\phi_4 = 2\pi.
\end{align}
By (\ref{cong1}), the equation $F_{m_1} = 2\theta_1 \sqrt{-1}$ becomes a linear equation
\begin{align}\label{f3eq1}
\phi_4 - \Arg z_1' - \Arg \tilde z_3' = \theta_1.
\end{align}
Since $\Arg(\tilde z_3') = \Arg(z_3'') = \pi - \phi_3 - \Arg(z_3')$, we have
\begin{align}\label{f3eq2}
\phi_4 - \Arg z_1' + \phi_3 + \Arg z_3' =  \pi + \theta_1.
\end{align}
By using a similar argument, one can verify that under the substitution (\ref{subs}), the equation 
$$\mathcal{G}_{\mathcal{D}}\Big( \big(2\theta_1\sqrt{-1},\dots, 2\theta_6\sqrt{-1}\big), \mathbf{z} \Big) = (0,\dots, 0)$$ becomes 7 linear equations 
\begin{empheq}[left = \empheqlbrace]{equation} \label{linver}
\begin{split}
\phi_1+\phi_2+\phi_3+\phi_4 &= 2\pi\\
\phi_4 - \Arg z_1' + \phi_3 + \Arg z_3' &= \pi + \theta_1 \\
\Arg z_4' - \Arg z_1' &= \theta_2 \\
\phi_2 - \phi_1 + \Arg z_2' - \Arg z_1' &= \theta_3 \\
\phi_3 - \Arg z_2' + \phi_4 + \Arg z_4' &= \pi + \theta_4 \\
\Arg z_2' - \Arg z_3' &= \theta_5 \\
\phi_3 - \phi_4 + \Arg z_3' - \Arg z_4' &= \theta_6 .
\end{split}
\end{empheq}
and 1 non-linear equation 
\begin{align}\label{nonlin}
r_1r_2r_3r_4 &= 1.
\end{align}
One can verify that by using $\Arg z_4'$ as the independent variable,
\begin{empheq}[left = \empheqlbrace]{equation} \label{linsoln}
\begin{split}
\phi_1 &= \frac{\pi}{2} + \frac{\theta_2 - \theta_3 - \theta_4}{2}, \\
\phi_2 &= \frac{\pi}{2} + \frac{-\theta_1 + \theta_3 - \theta_5}{2}, \\
\phi_3 &= \frac{\pi}{2} + \frac{\theta_4 + \theta_5 + \theta_6}{2}, \\
\phi_4 &= \frac{\pi}{2} + \frac{\theta_1 - \theta_2 - \theta_6}{2},\\
\Arg z_1' &= \Arg z_4' - \theta_2, \\
\Arg z_2' &= \Arg z_4' + \frac{\theta_1 - \theta_2 - \theta_4 + \theta_5}{2},  \\
\Arg z_3' &= \Arg z_4' + \frac{\theta_1 - \theta_2 - \theta_4 - \theta_5}{2} 
\end{split}
\end{empheq}
solve the system of linear equations (\ref{linver}). Under these conditions, by applying the Euclidean sine law on the triangle with angles $(\Arg z_i, \Arg z_i', \Arg z_i'') = (\phi_i, \Arg z_i', \pi - \phi_i - \Arg z_i')$ for $i=1,2,3,4$, $r_1,r_2,r_3$ and $r_4$ can be written as
\begin{empheq}[left = \empheqlbrace]{equation}\label{writer}
\begin{split}
r_1 &= \frac{\sin(\Arg z_4' - \theta_2)}{\cos(\Arg z_4' - \frac{\theta_2 + \theta_3 + \theta_4}{2})}, \\
r_2 &= \frac{\sin(\Arg z_4' + \frac{\theta_1 - \theta_2 - \theta_4 + \theta_5}{2})}{\cos(\Arg z_4' - \frac{\theta_2 - \theta_3 + \theta_4}{2})}, \\
r_3 &= \frac{\sin(\Arg z_4' + \frac{\theta_1 - \theta_2 - \theta_4 - \theta_5}{2})}{\cos(\Arg z_4' + \frac{\theta_1 - \theta_2 + \theta_6}{2})}, \\
r_4 &= \frac{\sin(\Arg z_4')}{\cos(\Arg z_4' + \frac{\theta_1 - \theta_2 - \theta_6}{2})} .
\end{split}
\end{empheq}
Let $z = e^{2\sqrt{-1} \Arg z_4'}$ and $u_l = e^{\sqrt{-1}\theta_l}$ for $l=1,\dots,6$. By using (\ref{writer}), the non-linear equation (\ref{nonlin}) can be reduced into the following quadratic equation
\begin{align*}
A z^2 + Bz + C = 0,
\end{align*}
where
\begin{equation*}
\begin{split}
A&= - \frac{u_1}{u_4} - \frac{u_1 u_3}{u_2} - \frac{u_1}{u_2u_3} - \frac{u_1}{u_2^2 u_4} - \frac{u_5}{u_2} - \frac{u_6}{u_2u_4} - \frac{1}{u_2u_4u_6} - \frac{1}{u_2u_5},\\
B&= - u_1u_4 + \frac{u_1}{u_4} + \frac{u_4}{u_1} - \frac{1}{u_1u_4}
 + u_2u_5 + \frac{u_2}{u_5} + \frac{u_5}{u_2} + \frac{1}{u_2u_5}
 - u_3 u_6 - \frac{u_3}{u_6} - \frac{u_6}{u_3} - \frac{1}{u_3u_6}  ,\\
C&= - \frac{u_4}{u_1} - \frac{u_2}{u_1 u_3} - \frac{u_2u_3}{u_1} - \frac{u_2^2 u_4}{u_1} - \frac{u_2}{u_5} - \frac{u_2u_4}{u_6} - u_2u_4u_6 - u_2u_5.
\end{split}
\end{equation*}
In particular, we let 
\begin{align}\label{z^*def}
z^* = \frac{-B - \sqrt{B^2 - 4AC}}{2A} .
\end{align}
Then (\ref{subs}), (\ref{linsoln}), (\ref{writer}) and  (\ref{z^*def}) together imply that (\ref{solnu}) solves the equation. Moreover, by direct computation, when $\mathrm{H}(m_1)=\dots=\mathrm{H}(m_6) = 0$, from (\ref{solnu}) we have 
$$(z_1,\dots, z_4, \tilde z_1, \dots, \tilde z_4) = \big(\sqrt{-1},\dots, \sqrt{-1}\big).$$ 
By Proposition \ref{solnexist}, since $\mathbf{z^*(H(m))}$ is the unique solution to the equation $\mathcal{G}_{\mathcal{D}}(\mathbf{H(m)}, \mathbf{z}(\mathbf{H}(m)) ) = 0$ with $\mathbf{z}^*(0,\dots, 0) = \big(\sqrt{-1},\dots, \sqrt{-1}\big)$, we have $\mathbf{z^*(H(m))}= (z_1^*,\dots, z_4^*, \tilde z_1^*, \dots, \tilde z_4^*)$. \\

\noindent{\bf Case 2:} For the general case, note that both $\mathbf{z^*(H(m))}$ and the solution $(z_1^*,\dots, z_4^*, \tilde z_1^*, \dots, \tilde z_4^*)$ in (\ref{solnu}) are holomorphic functions in $\mathrm{H}(m_1),\dots,\mathrm{H}(m_6)$. Moreover, by case 1, these holomorphic functions agree on the imaginary axis in an open neighborhood of the origin. The result follows from Lemma \ref{MCV} below.
\end{proof}

The following result is from \cite{WY1} and we include the proof here for reader's convenience.
\begin{lemma}\label{MCV}\cite[Lemma 4.2]{WY1} Suppose $D$ is a domain of $\mathbb C^n$ and $F_1$ and $F_2$ are two holomorphic functions on $D.$  If $F_1$ and $F_2$ coincide on $D\cap (\sqrt{-1}\mathbb R)^n,$ then $F_1$ and $F_2$ coincide on $D.$
\end{lemma}

\begin{proof}  We use induction on $n.$ If $n=1,$ then the result follows from the Identity Theorem of a single variable analytic function. Now suppose the result is true for $n\leqslant k.$ For each fixed $(z_2,\dots,z_k)\in (\sqrt{-1}\mathbb R)^{k-1},$ by the assumption of the lemma, we have 
$F_1(z_1,z_2,\dots,z_k)=F_2(z_1,z_2,\dots,z_k)$ for any purely imaginary $z_1.$ Then by the single variable case $F_1(z_1,z_2,\dots,z_k)=F_2(z_1,z_2,\dots,z_k)$ for any complex $z_1.$ This equality can also be understood as for any fixed complex $z_1,$ $F_1(z_1,z_2,\dots,z_k)=F_2(z_1,z_2,\dots,z_k)$  for all purely imaginary $(z_2,\dots,z_k).$ Then by the induction hypothesis, we have $F_1(z_1,z_2,\dots,z_k)=F_2(z_1,z_2,\dots,z_k)$ for all $(z_2,\dots,z_k).$
\end{proof}

\begin{remark}\label{gvDblock}
If we interpret the variety 
$$ \mathcal{Z}(\mathcal{D}) = \{ \mathbf{z}=(z_1,\dots,z_4,\tilde z_1, \dots, \tilde z_4) \in \CC^8 \mid z_1z_2z_3z_4 = \tilde z_1 \tilde z_2\tilde z_3\tilde z_4 = 1 \}$$
as the ``gluing variety" of the $D$-block $\mathcal{D}$, then Proposition \ref{explictsolnDB} shows that the map from $ \mathcal{Z}(\mathcal{D})$ to the $\mathrm{PSL}(2;\CC)$-character variety of $D$-block is surjective onto a neighborhood of the ``complete hyperbolic structure" of $\mathcal{D}$. 
\end{remark}

\subsection{Gluing equations of fundamental shadow link complements}\label{GEDB2}
Given a fundamental shadow link $L_{\text{FSL}} \subset M_c$ with $k$ components constructed by gluing $c$ copies of $D$-blocks $\{ \mathcal{D}_1, \dots , \mathcal{D}_c\}$, we let $\mathcal{D}_{i} = \mathcal{O}_i \cup \tilde{\mathcal{O}_i}$ be the decomposition of the $i$-th $D$-block into two ideal octahedra. We further triangulate the $D$-blocks into totally $8c$ ideal tetrahedra and assign shape parameters $\mathbf{z_i} = (z_{i_1},\dots, z_{i,4}, \tilde{z_{i,1}},\dots, \tilde{z_{i,4}})$ to the $i$-th $D$-block according to the decomposition discussed in Section \ref{GEDB}. Let $\mathbf{H(m)}=(\mathrm{H}(m_1), \dots, \mathrm{H}(m_k)) \in \CC^k$ be an assignment of $k$ complex numbers on the components of $L_{\text{FSL}}$. For each $D$-block $\mathcal{D}_i$, let $\mathcal{G}_{\mathcal{D}_i} : (\CC\setminus\{0,1\})^8 \to \CC^8$ be the map $F$ defined in (\ref{defFDB}) with respect to the holonomy inherited from the holonomy of the components of $L_{\text{FSL}}$. Define $\mathcal{G}_0: \CC^k \times (\CC\setminus\{0,1\})^{8c} \to \CC^{8c}$ by 
\begin{align}\label{defF_0}
\mathcal{G}_0(\mathbf{H(m)}, \mathbf{z_1},\dots, \mathbf{z_c}) = \big(\mathcal{G}_{\mathcal{D}_1}(\mathbf{H_{\mathcal{D}_1}(m)}, \mathbf{z_1}), \dots, \mathcal{G}_{\mathcal{D}_c}(\mathbf{H_{\mathcal{D}_c}(m)}, \mathbf{z_c})\big),
\end{align}
where for $i=1,\dots,c$, $\mathcal{G}_{\mathcal{D}_i}$ is the function $\mathcal{G}$ defined in (\ref{defFDB}) with respect to the $i$-th $D$-block and $\mathbf{H_{\mathcal{D}_i}}(\mathbf{m})=(\mathrm{H}_{\mathcal{D}_i}(m_1),\dots,\mathrm{H}_{\mathcal{D}_i}(m_6))$ are the holonomies of the ideal vertices of $\mathcal{D}_i$ inherited from the holonomies of the components of the fundamental shadow link. Besides, with respect to the triangulation, we have the gluing equation $\mathcal{G}(\mathbf{H(m)}, \mathbf{z}) = 
\mathbf{0}$ defined in (\ref{gluingeq}).

\begin{proposition}\label{solnexistFSL} There exists a neighborhood $U \subset \CC^{k}$ containing the origin such that for any \linebreak $\mathbf{H(m)} \in U$, there exists a unique $\mathbf{z^*(H(m))}$ with $\mathbf{z}^*(0,\dots, 0) = \big(\sqrt{-1},\dots, \sqrt{-1}\big)$ such that $\mathbf{z^*(H(m))}$ solve the gluing equation 
$\mathcal{G}(\mathbf{H(m)}, \mathbf{z}) 
= \mathbf{0}$
and the equation
$
\mathcal{G}_0(\mathbf{H(m)}, \mathbf{z_1},\dots, \mathbf{z_c}) = 0.
$
Furthermore, on each $D$-block $\mathcal D_i$, the solution $\mathbf{z_i}^*$ coincides with the solution in Proposition \ref{solnexist} with respect to the holonomy inherited from the holonomy of the components of $L_{\text{FSL}}$. 
\end{proposition}
\begin{proof}
We first consider the gluing equation $\mathcal{G}(\mathbf{H(m)}, \mathbf{z}) 
= \mathbf{0}$. Note that when $\mathbf{H(m)} = (0,\dots,0)$, the point $\mathbf{z}^*(\mathbf{H(m)} ) = \big(\sqrt{-1},\dots, \sqrt{-1}\big)$ solve the gluing equations. Especially, the triangulation is geometric in the sense that the imaginary parts of all shape parameters are positive. The existence and uniqueness of the solution $\mathbf{z^*(H(m))}$ around $\mathbf{H(m)}=(0,\dots,0)$ then follows from \cite[Corollary 15.2.17]{BM}. Thus, we have $\mathcal{G}(\mathbf{H(m)}, \mathbf{z}^*(\mathbf{H(m)})) 
= \mathbf{0}$. Moreover, the restriction of the solution $\mathbf{z^*(H(m))}$ on each $D$-block satisfies the condition in Proposition \ref{solnexist} with respect to the holonomy inherited from the holonomy of the components of $L_{\text{FSL}}$. By the uniqueness part of Proposition \ref{solnexist}, we have the third claim of the proposition. In particular, by Proposition \ref{solnexist}, we have $
\mathcal{G}_0(\mathbf{H(m)}, \mathbf{z_1},\dots, \mathbf{z_c}) = 0.
$
\end{proof}

\subsection{Proof of the volume formulas}\label{pfvolhypideal}
Motivated by Proposition \ref{explictsolnDB}, given $\big(\mathrm{H}(m_1),\dots, \mathrm{H}(m_6)\big)\in \CC^6$, we  define the volume of a $D$-block $\mathcal{D}$ with logarithmic holonomy $(\mathrm{H}(m_1),\dots,\mathrm{H}(m_6))$ by
\begin{align}\label{defvolDB}
\mathrm{Vol}_{\mathcal{D}} (\mathrm{H}(m_1),\dots,\mathrm{H}(m_6))
= \sum_{k=1}^4 \Big(  D(z_k^*) + D(\tilde{z}_k^*) \Big),
\end{align}
where $z_k^*$ and $\tilde{z}_k^*$ are the solution given in Proposition \ref{explictsolnDB}.
Let $(\theta_1,\dots,\theta_6) \in (0,\pi)^6$ be the dihedral angles of a hyperideal tetrahedron as shown in Figure \ref{hyperideal}. 
\begin{figure}[h]
\centering
\includegraphics[scale=0.2]{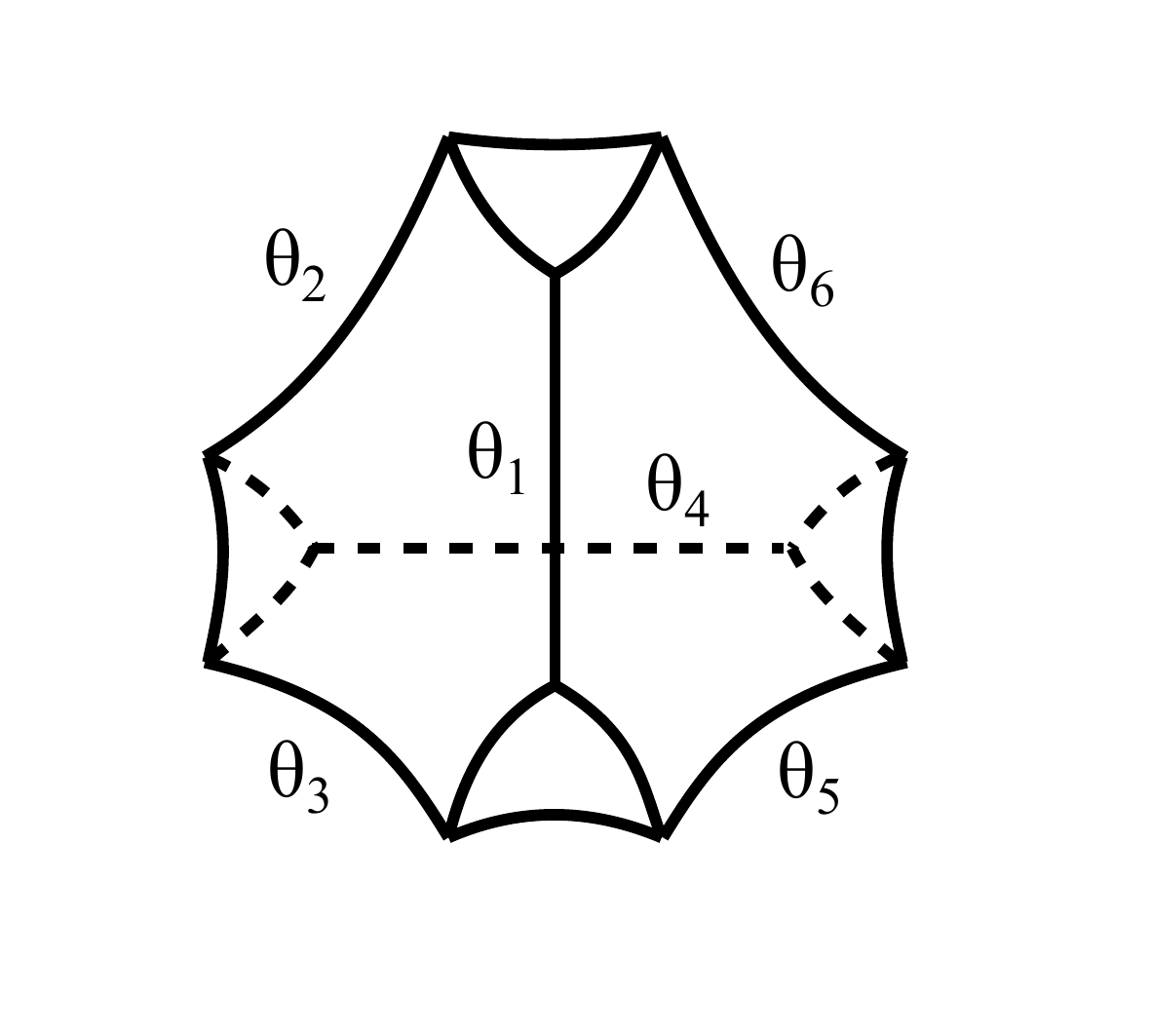}\qquad
\caption{A hyperideal tetrahedron.}
\end{figure}

\begin{proof}[Proof of Theorem \ref{volhypideal}]
Given a hyperideal tetrahedron $\Delta_{(\theta_1,\dots,\theta_6)}$ with dihedral angles $(\theta_1,\dots, \theta_6)$, we first take double of $\Delta_{(\theta_1,\dots,\theta_6)}$ along the triangles of truncation and then take double of the resulting manifold along the remaining boundary. Topologically, this construction gives a fundamental shadow link complement consisting of two $D$-blocks and 6 boundary components. Geometrically, the hyperbolic structure on the hyperideal tetrahedron induces a hyperbolic cone structure on $M_2= \#^3 (\SS^2 \times \SS^1)$ with singular locus the fundamental shadow link and cone angles $(2\theta_1,\dots, 2\theta_6)$. Let $L_{\text{FSL}}$ be the fundamental shadow link and let $M_2$ be the hyperbolic 3-manifold equipped with this cone structure along the singular locus $L_{\text{FSL}}$. On one side, the hyperbolic volume of $M_2$ is equal to 4 times the volume of $\Delta_{(\theta_1,\dots,\theta_6)}$. On the other side, we can compute the volume of $M_2$ by using the ideal triangulation discussed in Section \ref{GEDB2} and summing up the volume of the ideal tetrahedra in the triangulation. More precisely, if we have $\mathrm{H}(m_k)= 2\theta_k\sqrt{-1} $ for $k=1,\dots,6$ on one $D$-block, then $\mathrm{H}(m_k)= -2\theta_k\sqrt{-1} $ for $k=1,\dots,6$ on another $D$-block. Thus, 
\begin{align*}
\Vol(M_2) &= 4 \Vol(\Delta_{(\theta_1,\dots,\theta_6)}) \\
&= \mathrm{Vol}_{\mathcal{D}} \Big(2\theta_1\sqrt{-1},\dots,2\theta_6\sqrt{-1}\Big)+ \mathrm{Vol}_{\mathcal{D}} \Big(-2\theta_1\sqrt{-1},\dots,-2\theta_6\sqrt{-1}\Big) ,
\end{align*}
and the result follows.
\end{proof}

\begin{proof}[Proof of Theorem \ref{volmfd}]
The volume of $M_{\boldsymbol{\mu}}$ is given by the sum of the ideal tetrahedra of the triangulation $\mathcal{T}$ of the fundamental shadow link complement, which is exactly the sum of the volume of the $D$-blocks defined in (\ref{defvolDB}). 
\end{proof}

\section{1-loop conjecture for fundamental shadow link complements}\label{1loopinv}
\subsection{Generalized strong combinatorial flattening}\label{defGSCF}
Given a fundamental shadow link obtained by gluing $c$ copies of $D$-blocks, for each $D$-block $\mathcal{D}_i$ with shape parameters $(z_{i,1}, z_{i,2}, z_{i,3}, z_{i,4}, \tilde z_{i,1}, \tilde z_{i,2}, \tilde z_{i,3}, \tilde z_{i,4})$, we let $\mathbf{f_{i}} = (\frac{1}{2},\dots, \frac{1}{2})$, $\mathbf{f_{i}'} = (0,\dots,0)$ and $\mathbf{f_{i}''} = (\frac{1}{2},\dots, \frac{1}{2})$. Let $\boldsymbol{\mathcal{F}} = (\mathbf{f,f',f''}) $ be defined by
\begin{align}\label{deff}
\mathbf{f} = (\mathbf{f_1,\dots,f_c}), \quad \mathbf{f' = (f_1',\dots,f_c')} \quad\text{ and }\quad \mathbf{f'' = (f_1'',\dots,f_c'')}.
\end{align} 
\begin{lemma}\label{fGSCF}
The triple $\boldsymbol{\mathcal{F}} =\mathbf{(f,f',f'')}$ is a generalized strong combinatorial flattening of the ideal triangulation $\mathcal{T}$.
\end{lemma}
\begin{proof}
It is clear that the first condition in Definition \ref{GSCF} is satisfied. For the second condition, note that for each quadrilateral at an ideal vertex of a $D$-block, the combinatorial flattening contributes 2 and 1 at any interior vertex and at any boundary vertex respectively (Figure \ref{combf1}). Besides, the combinatorial flattening contributes 0 along each meridian (Figure \ref{combf1}). Moreover, along each vertical line segment, the combinatorial flattening also contributes 0. This implies that the combinatorial flattening contributes 0 along each vertical peripheral curve (Figure \ref{combf1}). 
\end{proof}

\begin{figure}[h]
\centering
\includegraphics[scale=0.12]{boundary1.pdf}\qquad
\includegraphics[scale=0.12]{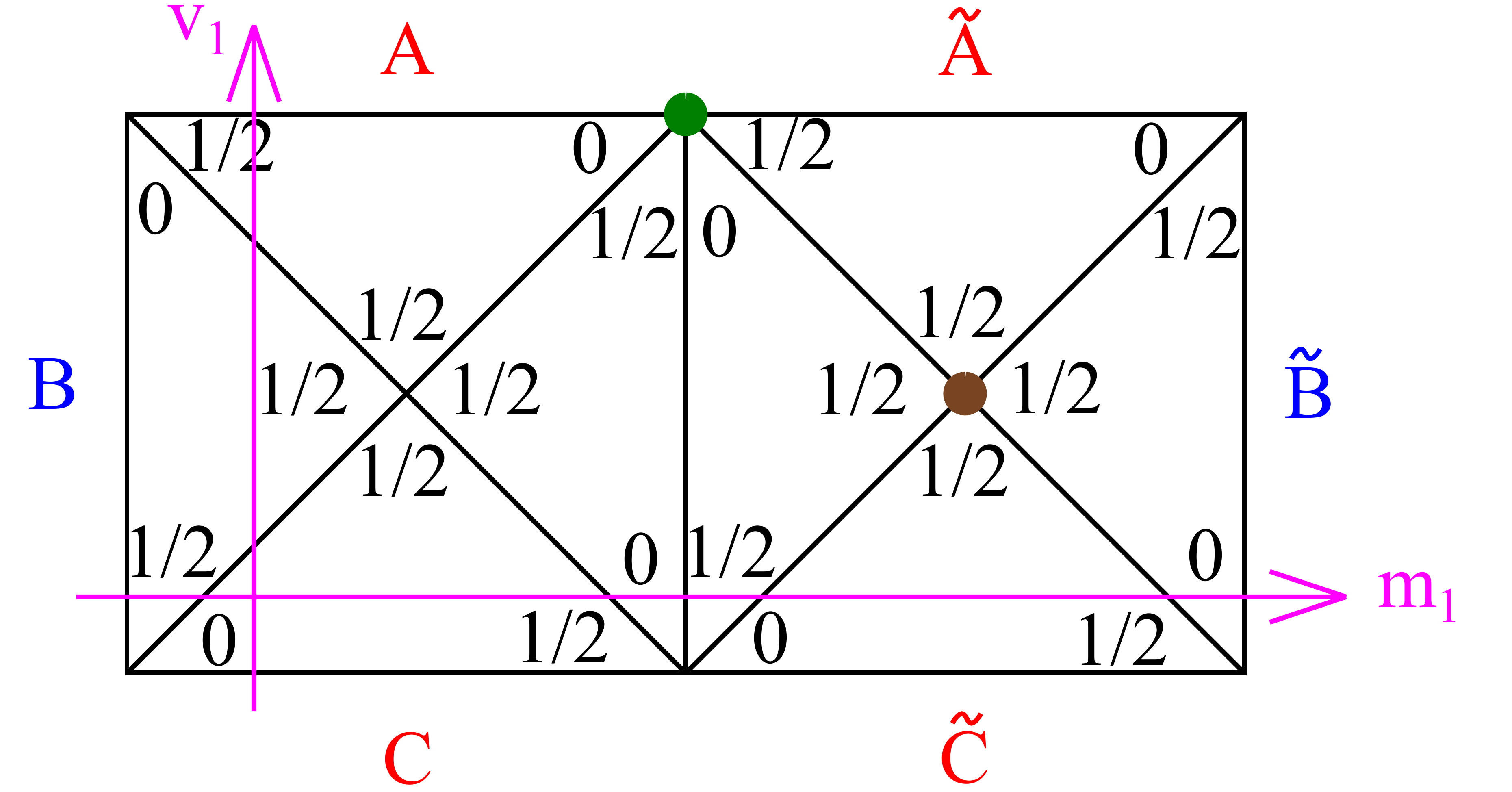}
\caption{The figures on the left and on the right show the triangulation of the truncated rectangle around vertex (1) and the contribution of the combinatorial flattening respectively. Note that at each interior vertex, such as the brown dot, the contribution is given by $1/2 \times 4 = 2$. At each boundary vertex, such as the green dot, the contribution is given by $1/2 \times 2 = 1$. Since every boundary vertex will be glued to another boundary vertex of some rectangle, around each vertex the contribution of the combinatorial flattening is $2$. Besides, along the curve $m_1$, the contribution from left to right is given by $-1/2 + 1/2 - 0 - 1/2 + 1/2 - 0 = 0$. Similarly, along a vertical line segment, such as $v_1$, the contribution is given by $0-1/2+1/2 = 0$.}\label{combf1}
\end{figure}

\begin{lemma}\label{diffg}
Let $\tilde{\boldsymbol{\mathcal{F}}}=\mathbf{(\tilde{f}, \tilde{f'}, \tilde{f''})}\in (\ZZ^{8c})^3$ be a strong combinatorial flattening. Let $\tau(M,\boldsymbol\alpha, \mathbf{z}, \mathcal{T})$ and $\tau(M,\boldsymbol\alpha, \mathbf{z}, \mathcal{T}, \boldsymbol{\mathcal{F}} )$ be the 1-loop invariants defined with respect to the flattenings $\tilde{\boldsymbol{\mathcal{F}} }$ and $\boldsymbol{\mathcal{F}} $ respectively. Then
$$
\tau(M,\boldsymbol\alpha, \mathbf{z}, \mathcal{T})
= \pm (\sqrt{-1})^g \tau(M,\boldsymbol\alpha, \mathbf{z}, \mathcal{T},\mathcal{F})
$$
for some constant $g\in \{0,1\}$.
\end{lemma}
\begin{proof}
From \cite[Equation (3-9)]{DG}, when we change the quad type, the 1-loop invariant defined with respect to the generalized strong combinatorial flattening $\boldsymbol{\mathcal{F}} =(\mathbf f,\mathbf f',\mathbf f'')$ changes by $\pm (\sqrt{-1})^{g_1}$ for some $g_1 \in \{0,1\}$. By changing the quad type if necessary, from \cite[Lemmas A.2 and A.3]{DG}, without loss of generality we can assume that $(G''-G')$ is invertible. By the argument in \cite[Section 3.5]{DG}, we have
$$
\frac{\prod_{i=1}^n \Big(z_i^{\tilde{f_i}''} z_i''^{-\tilde{f_i}}\Big)}{\prod_{i=1}^n \Big(z_i^{f_i''} z_i''^{-f_i}\Big)}
= e^{(\mathbf f'' \cdot \tilde{\mathbf f} - \mathbf f \cdot \tilde{\mathbf f}'')\pi\sqrt{-1}}
= \pm (\sqrt{-1})^{g_2}
$$ 
for some $g_2 \in\{0,1\}$, where $\mathbf v\cdot \mathbf w$ denotes the dot product of the two vectors $\mathbf v, \mathbf w \in \ZZ^{8c}$. The result follows by taking $g = g_1 + g_2 \pmod{2}$.
\end{proof}

For computation purposes, we introduce the following usual combinatorial flattening in the sense of Definition \ref{CF}, which will be used in the proof of Theorem \ref{mainthm} in Section \ref{pfmainthm}. 
For each $D$-block with shape parameters $(z_{i,1}, z_{i,2}, z_{i,3}, z_{i,4}, \tilde z_{i,1}, \tilde z_{i,2}, \tilde z_{i,3}, \tilde z_{i,4})$, let 
\begin{align*}
\mathbf{\hat{f}_{i}} = (1,1,1,-1,0,0,0,2),\quad  \mathbf{\hat{f}_{i}'} = (0,0,0,-1,1,1,1,-2)\quad \text{ and } \quad\mathbf{\hat{f}_{i}''} = (0,0,0,3,0,0,0,1).
\end{align*}
Let $\hat{\boldsymbol{\mathcal{F}} } = (\mathbf{\hat{f}, \hat{f}', \hat{f}''})$ be defined by 
\begin{align}\label{defhatf}
\mathbf{\hat{f} = (\hat{f}_1,\dots,\hat{f}_c), \quad \hat{f}' = (\hat{f}_1',\dots,\hat{f}_c') \quad\text{ and }\quad\hat{f}'' = (\hat{f}_1'',\dots,\hat{f}_c'')}.
\end{align} 
\begin{lemma}\label{hatfCF}
The triple $\hat{\boldsymbol{\mathcal{F}} } = \mathbf{(\hat{f}, \hat{f}', \hat{f}'')}$ is a combinatorial flattening of $\mathcal{T}$ for the system of meridians.
\end{lemma}
\begin{proof}
The proof is the same as the proof of Lemma \ref{fGSCF}. 
It is clear that the first condition in Definition \ref{GSCF} is satisfied. For the second condition, note that for each quadrilateral at an ideal vertex of a $D$-block, the combinatorial flattening contributes 2 and 1 at the interior vertex and at the boundary vertex respectively (Figure \ref{combf2}). Besides, the combinatorial flattening contributes 0 along each meridian (Figure \ref{combf2}). This completes the proof.
\end{proof}
\begin{figure}[h]
\centering
\includegraphics[scale=0.12]{boundary1.pdf}\qquad
\includegraphics[scale=0.12]{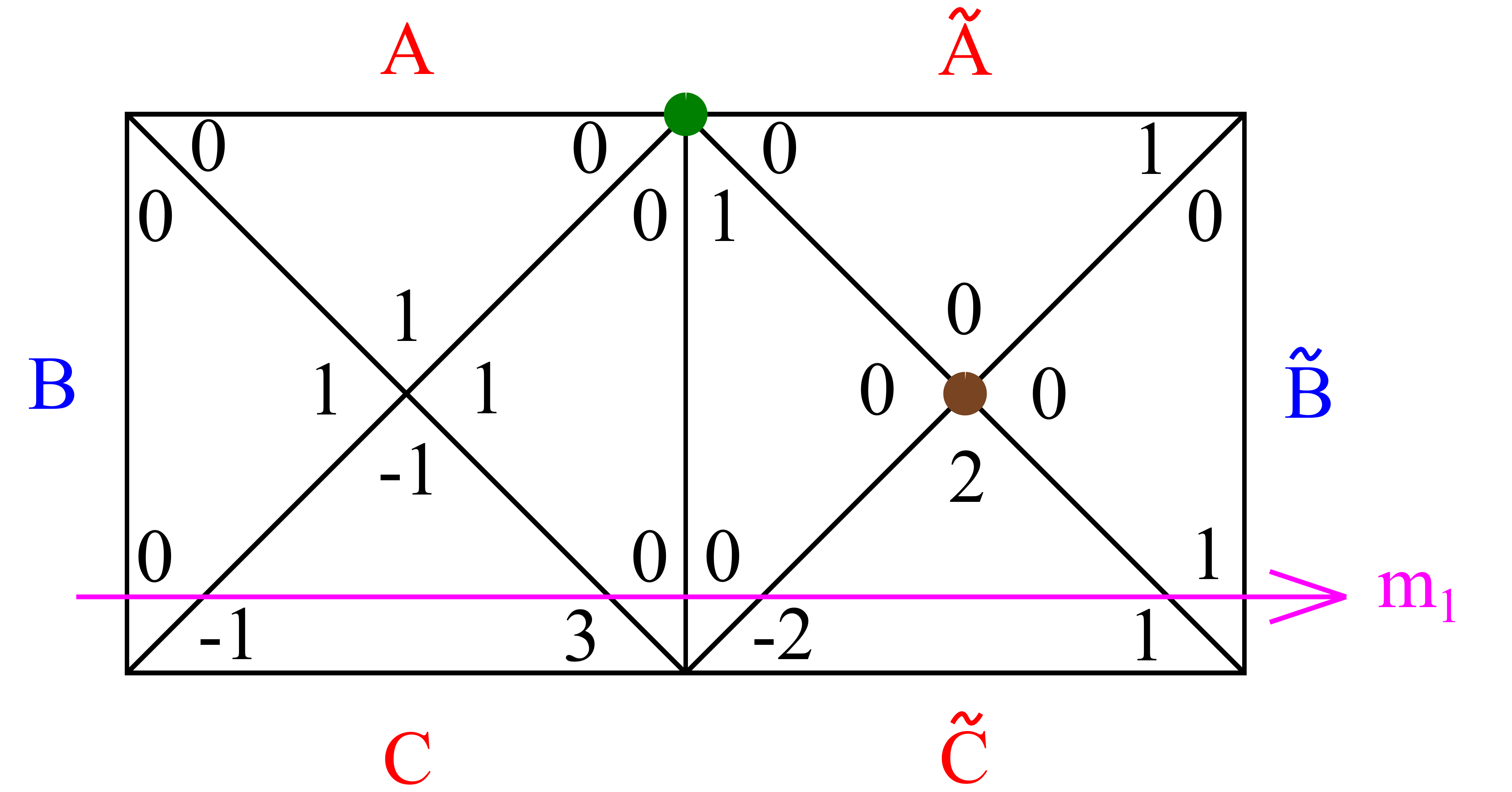}
\caption{The figures on the left and on the right show the triangulation of the truncated rectangle around vertex (1) and the contribution of the combinatorial flattening respectively. Note that at each interior vertex, such as the brown dot, the contribution is given by $2$. At each boundary vertex, such as the green dot, the contribution is given by $1$. Since the boundary vertex will be glued to another boundary vertex of some rectangle, around each vertex the contribution of the combinatorial flattening is $2$. Besides, along the curve $m_1$, the contribution from left to right is given by $-0 + (-1) -0 - 0 + 2 - 1 = 0$. }\label{combf2}
\end{figure}

\subsection{Computation on each building block}\label{compEB}
Recall that in the proof of Proposition \ref{explictsolnDB}, when $\mathrm{H}(m_l)=2\sqrt{-1}\theta_l$ for $l=1,\dots,6$, we define a change of variable (\ref{subs}) to simplifies the problem of solving the gluing equations on a $D$-block. In this section, we define a change of variable (\ref{defpsi}) that simplifies the computation of the determinant of the Jacobian matrix of the gluing map on each $D$-block. The main goal of this section is to prove Proposition \ref{1loopDB}, which relates the determinant of the Jacobian matrix with the determinant of the associated Gram matrix of the $D$-block defined in Section \ref{TFSL}.

Let $\psi:\CC^8 \to (\CC\cup\{\infty\})^8$ be the meomorphic function defined by
\begin{align}\label{defpsi}
\psi(\phi_1,\dots,\phi_4, \phi_1',\dots,\phi_4')
= (z_1,\dots, z_4, \tilde z_1,\dots, \tilde z_4),
\end{align}
where
$z_k = \frac{\sin \phi_k'}{\sin (\phi_k+\phi_k')}e^{\sqrt{-1} \phi_k}$ and 
$\tilde z_k = \frac{\sin (\phi_k+\phi_k')}{\sin \phi_k'}e^{\sqrt{-1} \phi_k}$ for $k=1,\dots,4$. For simplicity we write $\boldsymbol{\phi} = (\phi_1,\dots,\phi_4, \phi_1',\dots,\phi_4')$ and $\mathbf{z}=(z_1,\dots,z_4, \tilde{z_1},\dots,\tilde{z_4})$ such that $\psi(\boldsymbol{\phi}) = \mathbf{z}$. Then for the function $\mathcal{G}_{\mathcal{D}}$ defined in (\ref{defFDB}), we have
\begin{align}
D_{\boldsymbol{\phi}}(\mathcal{G}_{\mathcal{D}}\circ \psi) = D_{\mathbf{z}}\mathcal{G}_{\mathcal{D}} \circ D_{\boldsymbol{\phi}} \psi
\end{align}
and
\begin{align}\label{chainrule}
\det(D_{\boldsymbol{\phi}}(\mathcal{G}_{\mathcal{D}} \circ \psi)) = \det(D_{\mathbf{z}}\mathcal{G}_{\mathcal{D}} ) \cdot \det(D_{\boldsymbol{\phi}} \psi).
\end{align}
Note that for $\phi_k,\phi_k' \in (0,\pi)$ satisfing $\phi_k + \phi_k' < \pi$, $z_k$ and $\tilde{z}_k$ can be regarded as the shape parameters of ideal tetrahedra as shown in Figure \ref{Eucli}.
\begin{figure}[h]
\centering
\includegraphics[scale=0.12]{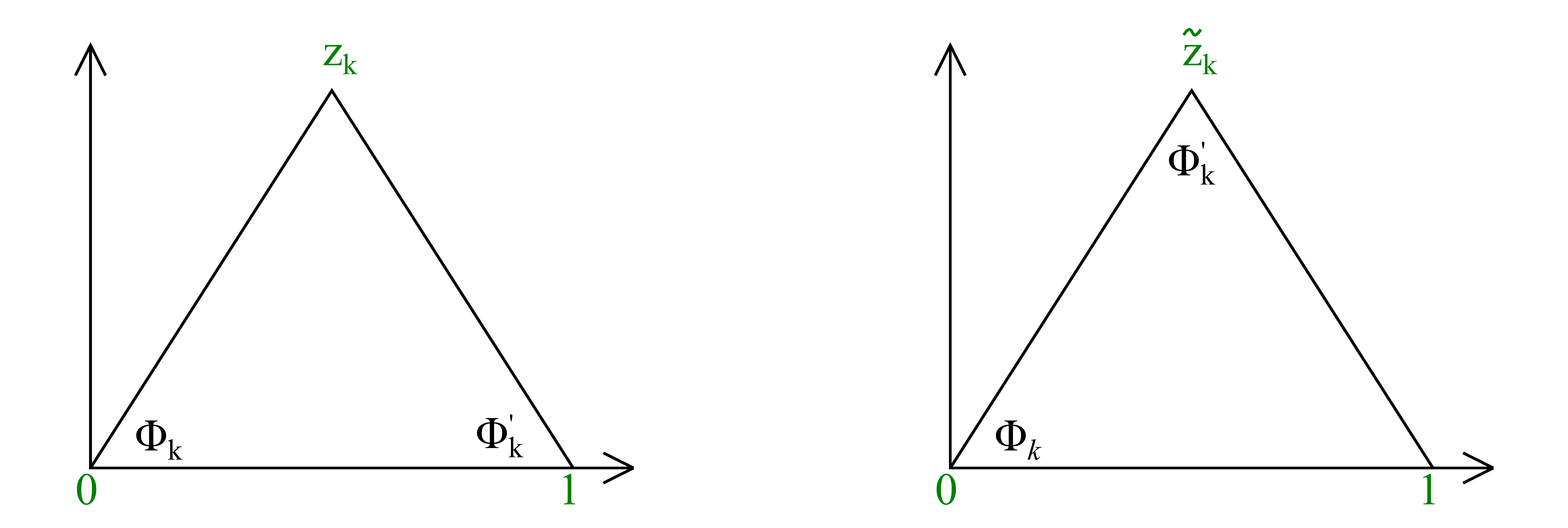}
\caption{When $\phi_k, \phi_k' \in (0,\pi)$ satisfy $\phi_k + \phi_k' < \pi$, one can construct a Euclidean triangle with angles $\phi_k, \phi_k'$ and $\pi - \phi_k - \phi_k'$. In particular, the norms $|z_k|, |1-z_k|, |\tilde{z}_k|$ and $|1-\tilde{z}_k|$ can be computed in terms of $\phi_k$ and $\phi_k'$ by using the Euclidean sine law on the triangles.}\label{Eucli}
\end{figure}
\begin{lemma}\label{DBcomp1}
When $\phi_1+\phi_2 + \phi_3 + \phi_4 = 2\pi$,
\begin{align*}
\det(D_{\mathbf{z}}\mathcal{G}_{\mathcal{D}} ) 
= -2^6 \sqrt{-1}  \sum_{i=1}^4 \frac{\sin \phi_i}{\sin \phi_i' \sin (\phi_i + \phi_i')} \Bigg(\prod_{i=1}^4 \frac{\sin\phi_i}{\sin\phi_i' \sin(\phi_i + \phi_i')}  \Bigg)^{-1}.
\end{align*}
\end{lemma}
\begin{proof}
Note that
\begin{align}
D_{\boldsymbol{\phi}} \psi
=
\begin{pmatrix}
\frac{\partial z_1}{\partial \phi_1} & 0 & 0 & 0 & \frac{\partial  z_1}{\partial \phi_1'} & 0 & 0 & 0 \\
0 & \frac{\partial z_2}{\partial \phi_2} & 0 & 0 & 0 & \frac{\partial  z_2}{\partial \phi_2'} & 0 & 0  \\
0 & 0 & \frac{\partial z_3}{\partial \phi_3} & 0 & 0 & 0 & \frac{\partial  z_3}{\partial \phi_3'} & 0  \\
0 & 0 & 0 & \frac{\partial z_4}{\partial \phi_4} & 0 & 0 & 0 & \frac{\partial  z_4}{\partial \phi_4'} \\
\frac{\partial \tilde z_1}{\partial \phi_1} & 0 & 0 & 0 & \frac{\partial  \tilde z_1}{\partial \phi_1'} & 0 & 0 & 0 \\
0 & \frac{\partial \tilde z_2}{\partial \phi_2} & 0 & 0 & 0 & \frac{\partial  \tilde z_2}{\partial \phi_2'} & 0 & 0  \\
0 & 0 & \frac{\partial \tilde z_3}{\partial \phi_3} & 0 & 0 & 0 & \frac{\partial \tilde z_3}{\partial \phi_3'} & 0  \\
0 & 0 & 0 & \frac{\partial \tilde z_4}{\partial \phi_4} & 0 & 0 & 0 & \frac{\partial \tilde z_4}{\partial \phi_4'}   
\end{pmatrix}
\end{align}
with determinant
\begin{align}
\det(D_{\boldsymbol{\phi}} \psi)
= \prod_{i=1}^4\Big(\frac{\partial z_i}{\partial \phi_i}\frac{\partial \tilde z_i}{\partial \phi_i'} - \frac{\partial z_i}{\partial \phi_i'}\frac{\partial \tilde z_i}{\partial \phi_i}\Big).
\end{align}
By a direct computation, for $i=1,\dots,4$,
\begin{empheq}[left = \empheqlbrace]{align*}
\frac{\partial z_i}{\partial \phi_i}
&= - \frac{e^{\sqrt{-1}\phi_i - \sqrt{-1} (\phi_i+\phi_i')}\sin\phi_i'}{\sin^2(\phi_i+\phi_i')}, \\
\frac{\partial z_i}{\partial \phi_i'}
&= \frac{e^{\sqrt{-1}\phi_i}\sin\phi_i}{\sin^2(\phi_i+\phi_i')}, \\
\frac{\partial \tilde z_i}{\partial \phi_i}
&= \frac{e^{\sqrt{-1}\phi_i + \sqrt{-1} (\phi_i+\phi_i')}}{\sin(\phi_i')}, \\
\frac{\partial z_i}{\partial \phi_i'}
&= -\frac{e^{\sqrt{-1}\phi_i}\sin\phi_i}{\sin^2(\phi_i')}.
\end{empheq}
Thus,
\begin{align}
\det(D_{\boldsymbol{\phi}} \psi)
=  16 \prod_{i=1}^4 \frac{e^{2\sqrt{-1}\phi_i}\sin\phi_i}{\sin\phi_i' \sin(\phi_i + \phi_i')}.
\end{align}
Especially, when $\phi_1+\dots+\phi_4= 2\pi$, we have
\begin{align}\label{dpsi}
\det(D_{\boldsymbol{\phi}} \psi)
=  16 \prod_{i=1}^4 \frac{\sin\phi_i}{\sin\phi_i' \sin(\phi_i + \phi_i')}.
\end{align}

Note that when $(\phi_1,\dots,\phi_4, \phi_1', \dots, \phi_4')\in\RR^8$, 
\begin{empheq}[left = \empheqlbrace]{align*}
F_{e_1}\circ \psi &= \log r_1 + \log r_2 + \log r_3 + \log r_4 + \sqrt{-1}( \phi_1 +  \phi_2 + \phi_3 + \phi_4) \\
F_{e_2}\circ \psi &= -\log r_1 - \log r_2 - \log r_3 - \log r_4 + \sqrt{-1}( \phi_1 +  \phi_2 +  \phi_3 +  \phi_4) \\
F_{m_1}\circ \psi &= 2\sqrt{-1}(\phi_4 - \phi_1' + \phi_3 + \phi_3')\\
F_{m_2}\circ \psi &= 2\sqrt{-1}(\phi_4' - \phi_1')\\
F_{m_3}\circ \psi &= 2\sqrt{-1}(\phi_2 - \phi_1 + \phi_2' - \phi_1') \\
F_{m_4}\circ \psi &= 2\sqrt{-1}(\phi_3 - \phi_2' + \phi_4 + \phi_4')\\
F_{m_5}\circ \psi &= 2\sqrt{-1}(\phi_2' - \phi_3')\\
F_{m_6}\circ \psi &= 2\sqrt{-1}(\phi_3 - \phi_4 + \phi_3' - \phi_4'),
\end{empheq}
where $r_i = \frac{\sin \phi_k'}{\sin (\phi_k+\phi_k')}$. 
This implies that 
\begin{align}\label{dK1}
\det D_{\boldsymbol\phi} (\mathcal{G}_{\mathcal{D}} \circ \psi) = (2\sqrt{-1})^7 \det K, 
\end{align}
where 
\begin{align}
K
=
\begin{pmatrix}
\frac{\partial \log r_1}{\partial \phi_1} & \frac{\partial \log r_2}{\partial \phi_2} & \frac{\partial \log r_3}{\partial \phi_3} & \frac{\partial \log r_4}{\partial \phi_4} &\frac{\partial \log r_1}{\partial \phi_1'} & \frac{\partial \log r_2}{\partial \phi_2'} & \frac{\partial \log r_3}{\partial \phi_3'} & \frac{\partial \log r_4}{\partial \phi_4'}  \\
1 & 1 & 1 & 1 & 0 & 0 & 0 & 0 \\
0 & 0 & 1 & 1 & -1 & 0 & 1 & 0 \\
0 & 0 & 0 & 0 & -1 & 0 & 0 & 1 \\
-1 & 1 & 0 & 0 & -1 & 1 & 0 & 0 \\
0 & 0 & 1 & 1 & 0 & -1 & 0 & 1 \\
0 & 0 & 0 & 0 & 0 & 1 & -1 & 0 \\
0 & 0 & 1 & -1 & 0 & 0 & 1 & -1 
\end{pmatrix}
\end{align}
with 
\begin{align}\label{detK}
\det K = 8 \sum_{i=1}^4 \frac{\partial \log r_i}{\partial \phi_i'} = 8 \sum_{i=1}^4 \frac{\sin \phi_i}{\sin \phi_i' \sin (\phi_i + \phi_i')}.
\end{align}
The result then follows from (\ref{chainrule}), (\ref{dpsi}), (\ref{dK1}) and (\ref{detK}).
\end{proof}

Let $\boldsymbol{\mathcal{F}}  = (\mathbf f,\mathbf f',\mathbf f'')$ be the generalized combinatorial flattening defined in (\ref{deff}).
\begin{lemma}\label{DBcomp2} When $\mathrm{H}(m_l) = 2\theta_l\sqrt{-1}$ for $l=1,\dots, 6$, at the solution $\mathbf{z}^*(\mathbf{H(m)})$ in Proposition \ref{explictsolnDB}, we have
$$\frac{1}{\prod_{i=1}^8 \xi_i^{f_i} \xi_i'^{f_i'} \xi_i''^{f_i''}}
= \pm \frac{\prod_{i=1}^4 \sin \phi_i}{\prod_{i=1}^4 \sin \phi_i'}
= \pm \frac{\prod_{i=1}^4 \sin \phi_i}{\prod_{i=1}^4 \sin (\phi_i + \phi_i')}.
$$
\end{lemma}
\begin{proof}
Note that 
$$\frac{1}{\prod_{i=1}^8 \xi_i^{f_i} \xi_i'^{f_i'} \xi_i''^{f_i''}}
= \frac{1}{ \Big(\frac{1}{z_1^*\dots z_4^* \tilde z_1^* \dots \tilde z_4^*}\Big)(z_1^{*'} \dots z_4^{*'} \tilde z_1^{*'} \dots \tilde z_4^{*'})^{\frac{1}{2}} }
= \frac{1}{(z_1^{*'} \dots z_4^{*'} \tilde z_1^{*'} \dots \tilde z_4^{*'})^{\frac{1}{2}}},
$$
where the last equality follows from $F_{e_1}(\mathbf{z}^*(\mathbf{H(m)}))= F_{e_2}(\mathbf{z}^*(\mathbf{H(m)})) = 0$. 

For $i=1,\dots, 4$, by the Euclidean sine law (Figure \ref{Eucli}), we have
$ |1-z_i^*| = \frac{\sin \phi_i}{\sin (\phi_i + \phi_i')}$, which implies 
$$ z_i^{*'} = \frac{1}{1-z_i^*} =  \frac{\sin (\phi_i + \phi_i')}{\sin \phi_i} e^{\sqrt{-1} \phi_i'}.$$
Recall from (\ref{cong2}) that $\Arg \tilde z_i^{*'} = \Arg z_i^{*''} = \pi - \phi_i - \phi_i'$. Moreover, by (\ref{cong2}),
$$
|\tilde z_i^{*'}| 
= \frac{1}{|z_i^{*''}|}
= \frac{|z_i^*|}{|1-z_i^*|}
= \frac{\sin \phi_i'}{\sin \phi_i}
$$
and
$$
\tilde z_i^{*'}
= \frac{\sin \phi_i'}{\sin \phi_i} e^{\sqrt{-1}(\pi - \phi_i - \phi_i')}.$$ 
Altogether,
$$
(z_1^{*'} \dots z_4^{*'} \tilde z_1^{*'} \dots \tilde z_4^{*'})^{\frac{1}{2}}
= \Bigg(\frac{\prod_{i=1}^4 \sin \phi_i' \sin (\phi_i+\phi_i')}{\prod_{i=1}^4 \sin^2 \phi_i}e^{\sqrt{-1}(4\pi - (\phi_1+\dots + \phi_4))}
 \Bigg)^{\frac{1}{2}} .$$
Recall that at the solution $\mathbf{z}^*(\mathbf{H(m)})$, we have
$$\phi_1+\dots + \phi_4 = 2\pi$$ and 
$$r_1 r_2 r_3 r_4 = \frac{\prod_{i=1}^4 \sin \phi_i'}{\prod_{i=1}^4  \sin (\phi_i + \phi_i')}=1.$$ Thus,
$$
\frac{1}{\prod_{i=1}^8 \xi_i^{f_i} \xi_i'^{f_i'} \xi_i''^{f_i''}}
= \frac{1}{(z_1^{*'} \dots z_4^{*'} \tilde z_1^{*'} \dots \tilde z_4^{*'})^{\frac{1}{2}}}
= \pm \frac{\prod_{i=1}^4 \sin \phi_i}{\prod_{i=1}^4 \sin \phi_i'}
= \pm \frac{\prod_{i=1}^4 \sin \phi_i}{\prod_{i=1}^4 \sin (\phi_i + \phi_i')}.
$$
\end{proof}

\begin{proposition}\label{1loopDB} When $\mathrm{H}(m_l) = 2\theta_l\sqrt{-1}$ for $l=1,\dots, 6$, at the solution $\mathbf{z}^*(\mathbf{H(m)})$ in Proposition \ref{explictsolnDB}, we have
\begin{align*}
\frac{\det(D_{\mathbf{z}}\mathcal{G}_{\mathcal{D}} (\mathbf{z}^* )) }{\prod_{i=1}^8 \xi_i^{f_i} \xi_i'^{f_i'} \xi_i''^{f_i''}}
= \pm 2^5 \sqrt{\det \mathbb{G}}
\end{align*}
where $\mathrm{H}(m_l)$ is the holonomy of the curve $m_l$ around the $l$-th ideal vertex and $\mathbb{G}$ is the associated Gram matrix defined in Section \ref{TFSL}.
\end{proposition}
\begin{proof}
From Lemmas \ref{DBcomp1} and \ref{DBcomp2},
\begin{align}\label{DB12}
\frac{\det(D_{\mathbf{z}}\mathcal{G}_{\mathcal{D}} (\mathbf{z}^* )) }{\prod_{i=1}^8 \xi_i^{f_i} \xi_i'^{f_i'} \xi_i''^{f_i''}}
=
\pm 2^6 \sqrt{-1}  \prod_{i=1}^4 \sin(\phi_i + \phi_i') \Bigg(\sum_{i=1}^4 \frac{\sin \phi_i}{\sin (\phi_i') \sin (\phi_i + \phi_i')} \Bigg).
\end{align}
Let $f(\phi_4') = \prod_{i=1}^4 \sin \phi_i'$ and $g(\phi_4') = \prod_{i=1}^4 \sin (\phi_i + \phi_i')$. Let $u_l = e^{\sqrt{-1} \theta_l}$ for $l=1,\dots, 6$ and $z=e^{2\sqrt{-1} \phi_4'}$. Then
\begin{align}\label{fgz}
f(\phi_4') - g(\phi_4') 
= \frac{1}{16}\Big(Az + B + \frac{C}{z}\Big),
\end{align}
where $A,B$ and $C$ are defined in (\ref{defABC}). Besides, by (\ref{disgram}),
\begin{align}\label{disdetG2}
B^2 - 4AC = 16 \det \mathbb{G}
\end{align}
where $\det \mathbb{G}$ is the determinant of the associated Gram matrix defined in Section \ref{TFSL}.
Let $z_\pm = \frac{-B\pm\sqrt{B^2 - 4AC}}{2A}$ and let $\phi_\pm$ such that $z_\pm = e^{2\sqrt{-1}\phi_\pm}$. From (\ref{fgz}), we have
$$
f(\phi_-) - g(\phi_-) = f(\phi_+) - g(\phi_+) = 0. 
$$
From (\ref{detK}) and (\ref{linsoln}),
\begin{align}
\sum_{i=1}^4 \frac{\sin \phi_i}{\sin(\phi_i') \sin (\phi_i + \phi_i')}
&=\sum_{i=1}^4 \frac{\partial \log r_i}{\partial \phi_i'} 
=\sum_{i=1}^4 \frac{\partial \log r_i}{\partial \phi_4'},
\end{align}
which implies that
\begin{align}\label{fgexp}
\sum_{i=1}^4 \frac{\sin \phi_i}{\sin(\phi_i') \sin (\phi_i + \phi_i')} 
= \frac{\partial}{\partial \phi_4'}( \log (r_1r_2r_3r_4))
= \frac{\partial}{\partial \phi_4'} \log \Bigg(\frac{f(\phi_4')}{g(\phi_4')}\Bigg)
= \frac{f'(\phi_4') - g'(\phi_4')}{g(\phi_4')}
\end{align}
for $\phi_4'$ satisfying $f(\phi_4') = g(\phi_4')$. Thus, from (\ref{DB12}) and (\ref{fgexp}),
\begin{align}
\frac{\det(D_{\mathbf{z}}\mathcal{G}_{\mathcal{D}} (\mathbf{z}^* )) }{\prod_{i=1}^8 \xi_i^{f_i} \xi_i'^{f_i'} \xi_i''^{f_i''}}
=
\pm 2^6 \sqrt{-1} (f'(\phi_4') - g'(\phi_4'))
= \pm 2^6 \sqrt{-1} \frac{\partial (f(z)-g(z))}{\partial z}\Big\vert_{z = z_-} \frac{\partial z}{\partial \phi_4'}\Big\vert_{\phi_4' = \phi_-}.
\end{align}
From (\ref{fgz}), we have
\begin{align*}
\frac{\det(D_{\mathbf{z}}\mathcal{G}_{\mathcal{D}} (\mathbf{z}^* )) }{\prod_{i=1}^8 \xi_i^{f_i} \xi_i'^{f_i'} \xi_i''^{f_i''}}
=& \pm 2^3 \Big(Az_- - \frac{C}{z_-}\Big) \\
=& \pm 2^3 A \Big(z_- - \frac{C}{Az_-}\Big) \\
=& \pm 2^3 A (z_- - z_+) \\
=& \pm 2^3 \sqrt{B^2-4AC} \\
=& \pm 2^5 \sqrt{\det \mathbb{G}},
\end{align*}
where the third equality follows from the product of root formula $z_+ z_- = C/A$ and  the last equality follows from (\ref{disdetG2}).
\end{proof}

\subsection{Proof of Theorem \ref{mainthm}}\label{pfmainthm}
We claim that when $\boldsymbol\alpha$ is the system of meridians, the 1-loop invariant $\tau(M,\boldsymbol\alpha, \mathbf{z^*}, \mathcal{T},\boldsymbol{\mathcal{F}})$ with respect to the generalized strong combinatorial flattening $\boldsymbol{\mathcal{F}}  = (\mathbf f,\mathbf f',\mathbf f'')$ defined in (\ref{deff}) is given by
\begin{align}\label{1loopmer}
\tau(M,\alpha, \mathbf{z^*}, \mathcal{T},\mathcal{F}) 
= \pm 2^{3c} \prod_{i=1}^c \sqrt{\det \mathbb{G}_i},
\end{align}
where $ \mathbb{G}_i$ is the associated Gram matrix of the $i$-th $D$-block defined in Section \ref{TFSL}. 

We first consider the case where $\mathrm{H}(m_l) = 2 \theta_l \sqrt{-1}$ for some $\theta_l \in \RR$ and $l=1,\dots, k$. 
By Proposition \ref{1loopDB} and Lemma \ref{detL}, 
\begin{align*}
\tau(M,\alpha, \mathbf{z^*}, \mathcal{T},\boldsymbol{\mathcal{F}}  )
=& \frac{1}{2}\frac{\mathrm{det}(D_{\mathbf z}\mathcal{G}(\mathbf{H(m)},\mathbf z^*))}{\prod_{i=1}^{8c} \xi_i^{f_i} \xi_i'^{f_i'} \xi_i''^{f_i''}}\\
=& 2^{-2c}\frac{\mathrm{det}(D_{\mathbf z}\mathcal{G}_0(\mathbf{H(m)},\mathbf z^*))}{\prod_{i=1}^{8c} \xi_i^{f_i} \xi_i'^{f_i'} \xi_i''^{f_i''}}\\
=&
  2^{-2c}\Bigg(\frac{\prod_{i=1}^c \mathrm{det}(D_{\mathbf z} \mathcal{G}_{\mathcal{D}_i}(\mathbf{H(m)},\mathbf z^*))}{ \prod_{i=1}^{8c} \xi_i^{f_i} \xi_i'^{f_i'} \xi_i''^{f_i''}}\Bigg)\\
=& \pm 2^{3c} \prod_{s=1}^c \sqrt{\det \mathbb{G}_s}.
\end{align*}
By Lemma \ref{diffg}, we have
\begin{align*}
\tau(M,\boldsymbol\alpha, \mathbf{z^*}, \mathcal{T}) 
= \pm 2^{3c} (\sqrt{-1})^g \prod_{s=1}^c \sqrt{\det \mathbb{G}_s}
\end{align*}
for some $g\in \{0,1\}$. Finally, to compute $g$, we compute $\tau(M,\boldsymbol\alpha, \mathbf{z}^*, \mathcal{T}) $ with respect to the combinatorial flattening $\hat{\boldsymbol{\mathcal{F}}  }$ defined in (\ref{defhatf}) at the holonomy representation of the complete hyperbolic structure, where the shape parameters are given by $\mathbf{z}^*(0,\dots,0) = (\sqrt{-1},\dots,\sqrt{-1})$. Note that in this case,
\begin{align}\label{com1}
\prod_{i=1}^{8c} \xi_i^{f_i} \xi_i'^{f_i'} \xi_i''^{f_i''}
= \Bigg(\Bigg(\frac{1}{\sqrt{-1}}\Bigg)^4 \Bigg(\frac{1}{\sqrt{-1}\big(\sqrt{-1}-1\big)}\Bigg)^4 \Bigg)^c
= 2^{-2c}.
\end{align}
Besides, 
\begin{align}\label{com2}
\mathrm{det}(D_{\mathbf z}\mathcal{G}(\mathbf{H(m)},\mathbf z^*))
= \pm 2^{1-2c} \prod_{i=1}^c  \mathrm{det}(D_{\mathbf z} \mathcal{G}_{\mathcal{D}_i}(\mathbf{H(m)},\mathbf z^*))
= \pm 2^{1-2c} \big(32\sqrt{-1}\big)^c ,
\end{align}
where the last equality follows from the (\ref{detcomplete}).
From (\ref{com1}) and (\ref{com2}), 
\begin{align}
\tau(M,\boldsymbol\alpha, \mathbf{z}^*, \mathcal{T}) 
= \pm  2^{3c} \prod_{s=1}^c \sqrt{\det \mathbb{G}_s}.
\end{align}
Thus, we have $g=0$ and 
\begin{align}\label{com3}
\tau(M,\boldsymbol\alpha, \mathbf{z}^*, \mathcal{T}) 
= \pm 2^{3c} \prod_{s=1}^c \sqrt{\det \mathbb{G}_s}
= \pm \mathbb T_{(M,\alpha)}([\rho_0]),
\end{align}
where $\rho_0$ is the discrete faithful representation of $M$ and the last equality follows from Theorem \ref{main1}(1).
This completes the proof in the case that $\mathrm{H}(m_l) = 2 \theta_l \sqrt{-1}$ for some $\theta_l \in \RR$ and $l=1,\dots, k$. Since both the 1-loop invariant and the torsion are locally analytic functions on $(\mathrm{H}(m_1),\dots, \mathrm{H}(m_k))$, by Lemma \ref{MCV}, there exists an open subset $U \subset \CC^k$ containing $(0,\dots,0)$ such that for any character $[\rho]$ sufficiently close to the discreta faithful character with $(\mathrm{H}(m_1),\dots, \mathrm{H}(m_k))\in U$, we have
\begin{align}
\tau(M,\boldsymbol\alpha, \mathbf{z}, \mathcal{T}) 
= \pm \mathbb T_{(M,\alpha)}(\rho_{\mathbf{z}}).
\end{align}
More generally, by Proposition \ref{toranaglu}, both the 1-loop invariant and the torsion are analytic functions on the gluing variety $\mathcal{V}_{\mathcal{T}}$. As a result, by analyticity, they agree on $\mathcal{V}_{\mathcal{T}}$.

\subsection{Combinatorics of the gluing equations}
In this section, we prove the technical lemmas about the combinatorics of the gluing equations. 
Let $\mathcal{G}: \CC^k \times (\CC\setminus\{0,1\})^{8c} \to  \CC^{8c}$ be the function defined in (\ref{gluingeq}) with respect to the triangulation $\mathcal{T}$. Note that the definition of the map $\mathcal{G}$ depends on a choice of a set of linearly independent edges. Recall that the function $\mathcal{G}_0: \CC^k \times (\CC\setminus\{0,1\})^{8c}  \to \CC^{8c}$ in (\ref{defF_0}) is defined by
\begin{align}
\mathcal{G}_0(\mathbf{H(m)}, \mathbf{z_1},\dots, \mathbf{z_c}) = \big(\mathcal{G}_{\mathcal{D}_1}((\mathbf{H_{\mathcal{D}_1}(m)}, \mathbf{z_1}), \dots, \mathcal{G}_{\mathcal{D}_c}((\mathbf{H_{\mathcal{D}_c}(m)}, \mathbf{z_c})\big),
\end{align}
where $\mathbf{z_1},\dots, \mathbf{z_c} \in \CC^8\setminus\{0,1\}$, $\mathcal{G}_{\mathcal{D}_i}: (\CC\setminus\{0,1\})^8 \to \CC^8$ is the function $\mathcal{G}_{\mathcal{D}}$ defined in (\ref{defFDB}) with respect to the $i$-th $D$-block $\mathcal{D}_i$. 

The main goal of the section is to prove Lemma \ref{detL}, which relates the determinants of the Jacobian matrices of $\mathcal{G}_0$ and $\mathcal{G}$.
The proof follows from elementary computation and careful checking of all possible cases. 
\begin{lemma}\label{detL}
There exists a set of linearly independent edges such that for the map $\mathcal{G}$ defined with respect to this set of edges, we can find an invertible matrix $L$ with $\det L =\pm 2^{2c-1}$ such that 
$$ 
\mathrm{det}(D_{\mathbf z}\mathcal{G}_0(\mathbf{H(m)},\mathbf z^*))
= \det (L)\mathrm{det}(D_{\mathbf z}\mathcal{G}(\mathbf{H(m)},\mathbf z^*))
= \pm 2^{2c-1}\mathrm{det}(D_{\mathbf z}\mathcal{G}(\mathbf{H(m)},\mathbf z^*)).
$$
\end{lemma}

\begin{proof}
Suppose the fundamental shadow link complement has $k$ boundary $T_1\coprod \dots \coprod T_k$. 
Recall that a fundamental shadow link complement is obtained by gluing the $D$-blocks together along the red faces in Figure \ref{idealocta}. We call the three edges of a red face a \emph{cycle}. Note that in the construction of a fundamental shadow link complement, cycles from different $D$-blocks are identified with each other. Besides, since each $D$-block has $4$ cycles and each cycle is shared by two $D$-blocks, there are altogether $2c$ cycles.

In the definition of $\mathcal{G}_0$, there are $2c$ edge equations $\{f^l_{1},f^l_{2}\}_{l=1}^{c}$ corresponding to the central edges and $6c$ holonomy equations. In the definition of the gluing map $F$, there are $8c-k$ edge equations and $k$ holonomy equations. In the following discussion, we will always choose the $2c$ central edges $\{f^l_{1},f^l_{2}\}_{l=1}^{c}$ and study the transformation from $6c$ holonomy equations to $k$ holonomy equations and $6c-k$ edges equations. The main goal is to show that the the transformation is given by a invertible matrix $L$ with $\det L =\pm 2^{2c-1}$. We will show that by choosing the dependent edges appropriately and applying row operations that do not change the determinant of $L$, we can transform $L$ into an upper triangular block matrix which contains $2c-1$ blocks with determinant $\pm 2$ and $1$ block with determinant $\pm1$.

First, given a fundamental shadow link complement, we construct a graph as follows. The vertices of the graph are the boundaries of the link complement. An edge between two vertices $T_p$ and $T_q$, where $p,q=1,\dots,k$, is a cycle $\mathbf{c}=(e_1,e_2,e_3)$ such that at least one pair of edges from that cycle $\mathbf{c}$ belong to the boundary $T_p$ and at least one pair of edges from that cycle $\mathbf{c}$ belong to the boundary $T_q$. For a pair of edges $(e_a,e_b)$ and for a boundary torus $T_p$, where $a,b\in \{1,2,3\}$ and $p=1,\dots,k$, we write $(e_a,e_b) \in T_p$ if the pair of edges $(e_a,e_b)$ lies on the boundary $T_p$.

Since the fundamental shadow link complement is connected, we can always find a maximal tree that contains all the vertices of the graph. In particular, the graph contains $k$ vertices and $k-1$ edges. Each edge of the maximal tree corresponds to a cycle. Observe that since each cycle $\mathbf{c}=(e_1,e_2,e_3)$ contains three pairs of edges $(e_1,e_2), (e_1,e_3)$ and $(e_2,e_3)$, each cycle corresponds as most two edges of the tree. We choose a dendrogram representation of the graph with a top degree one vertex and directed edges pointing from the top to the bottom. We call that vertex the \emph{distinguished vertex}. Since the distinguished vertex is degree one, there exists a unique edge that has the top vertex as one of its endpoints. We call the corresponding cycle the \emph{distinguished cycle} and denote it by $\mathbf{c}^1=(e_1^1,e_2^1,e_3^1)$. For each pair of vertices $T_p,T_q$ in the graph, where $p,q =1,\dots,k$, there exist a unique path of edges connecting that $T_p$ and $T_q$ without backtracking. We denote the number of edges in that path by $d(p,q)$. For adjacent vertices $T_p,T_q$ of the graph, we denote the edge connecting $T_p$ and $T_q$ by $\mathbf{c(p,q)}$. From each edge of the dendrogram labelled by some cycle $\mathbf{c}^i$, we pick a dependent edge from the bottom vertex of the edge as follows.

\begin{enumerate}
\item Suppose the cycle $\mathbf{c}^i=(e^i_1,e^i_2,e^i_3)$ corresponds to exactly one edge of the tree. We have the following possible positions of the three pairs of edges.
\begin{enumerate}[I)]
\item There is a branch 
$$\dots \to T_a  \xrightarrow{\mathbf{c^i}} T_b \to \dots,$$ 
where $a,b=1,\dots,k$, $(e^i_1,e^i_2) \in T_b$ is on the top layer of $T_b$,  $(e^i_1,e^i_3) \in T_a$ is not on the top layer of $T_a$ and $(e^i_2,e^i_3) \in T_a$ is not on the top layer of $T_a$.
\item There is a branch 
$$\dots \to T_a  \xrightarrow{\mathbf{c^i}} T_b \to \dots,$$ 
where $a,b=1,\dots,k$, $(e^i_1,e^i_2) \in T_b$ is on the top layer of $T_b$, $(e^i_1,e^i_3) \in T_a$ is not on the top layer of $T_a$ and $(e^i_2,e^i_3) \in T_b$ is not on the top layer of $T_b$. 
\item There is a branch 
$$\dots \to T_a \to \dots \to T_b \xrightarrow{\mathbf{c^i}} T_c \to \dots,$$ 
where $a,b,c =1,\dots,k$, $(e^i_1,e^i_2) \in T_c$ is on the top layer of $T_c$, $(e^i_1,e^i_3) \in T_b$ is not on the top layer of $T_b$ and $(e^i_2,e^i_3) \in T_a$ is not on the top layer of $T_a$. 
\item There is a branch 
$$\dots \to T_a \xrightarrow{\mathbf{c^i}} T_b  \to \dots \to  T_c \to \dots,$$ 
where $a,b,c =1,\dots,k$, $(e^i_1,e^i_2) \in T_b$ is on the top layer of $T_b$, $(e^i_1,e^i_3) \in T_a$ is not on the top layer of $T_a$ and $(e^i_2,e^i_3) \in T_c$ is not on the top layer of $T_c$. 
\item $T_a,T_b$ and $T_c$ are not in the same branch, $(e^i_1,e^i_2) \in T_b$ is on the top layer of $T_b$, $(e^i_1,e^i_3) \in T_a$ is not on the top layer of $T_a$ and $(e^i_2,e^i_3) \in T_c$ is not on the top layer of $T_c$. 
\end{enumerate}
We choose $e^i_1$ to be a dependent edge in all these cases.
\item Suppose the cycle $\mathbf{c}^i=(e^i_1,e^i_2,e^i_3)$ corresponds to two edges of the tree. Then there is a branch
$$ \dots \to T_a \to T_b \to T_c \to \dots,$$
where $(e_1^i,e_2^i)\in T_c$ is on the top layer of $T_c$, $(e_1^i,e_3^i)\in T_b$ is on the top layer of $T_b$ and $(e_2^i,e_3^i)\in T_a$ is not on the top layer of $T_a$. We choose $e^i_2$ and $e^i_3$ to be dependent edges.
\end{enumerate}

This gives us altogether $k-1$ dependent edges. Next, we choose one more edge from the distinguished cycle $\mathbf{c}^1=(e_1^1,e_2^1,e_3^1)$ to be a dependent edge as follows. 

\begin{enumerate}
\item Suppose two edges of $\mathbf{c^1}$ have been chosen to be dependent edges in the previous step. Then we choose the remaining edge to be a dependent edge.

\item Up to renaming $\{T_i\}_{i=1}^k$ if necessary, suppose we have a branch of the form 
$$T_1 \xrightarrow{\mathbf{c^1}} T_2 \xrightarrow{\mathbf{c(2,3)}} T_3 \xrightarrow{\mathbf{c(3,4)}} \dots,$$ 
where $\mathbf{c^1} = (e_1^1,e_2^1,e_3^1)$, $(e_1^1, e_3^1)\in T_1$ is on the top layer of $T_1$, $(e_1^1, e_2^1) \in T_2$ is on the top layer of $T_2$ and $(e_2^1, e_3^1) \in T_2$ is not on the top layer. Then we choose $e_1^1$ and $e_2^1$ to be dependent edges.

\item Up to renaming $\{T_i\}_{i=1}^k$ if necessary, suppose we have a branch of the form 
$$T_1 \xrightarrow{\mathbf{c^1}} T_2 \xrightarrow{\mathbf{c(2,3)}} T_3 \xrightarrow{\mathbf{c(3,4)}} \dots,$$ 
where $\mathbf{c^1} = (e_1^1,e_2^1,e_3^1)$, $(e_1^1, e_3^1)\in T_1$ is on the top layer of $T_1$, $(e_2^1, e_3^1) \in T_1$ is not on the top layer and $(e_1^1, e_2^1) \in T_2$ is on the top layer of $T_2$. Then we choose $e_1^1$ and $e_3^1$ to be dependent edges.

\item Up to renaming $\{T_i\}_{i=1}^k$ if necessary, suppose for some $l\geq 2$, we have a branch of the form 
$$T_1 \xrightarrow{\mathbf{c(1,2)}} T_2 \xrightarrow{\mathbf{c(2,3)}} \dots \xrightarrow{\mathbf{c(l,l+1)}} T_{l+1} \to\dots, $$
where $ (e_1^1, e_3^1)\in T_1$ is on the top layer of $T_1$, $(e_1^1, e_2^1) \in T_2$ is on the top layer of $T_2$ and $(e_2^1, e_3^1) \in T_{l+1}$ is not on the top layer of $T_{l+1}$. We choose $e^1_1$ and $e^1_2$ to be dependent edges if $l$ is odd and choose $e^1_1$ and $e^1_3$ to be dependent edges if $l$ is even. 
\end{enumerate}

We give an order $\{\mathbf{c}^1,\mathbf{c}^2,\dots, \mathbf{c}^{2c}\}$ to the set of cycles, where $\mathbf{c}^1$ is the distinguished cycle. As shown in Figure (\ref{ndetl3}), for each boundary torus of the fundamental shadow link component, we find a fundamental domain of the torus such that the dependent edge we picked from that boundary component is at the top layer of the fundamental domain. We will keep the holonomy equation in the top rectangle. By applying row operations that does not affect the determinant of a matrix, we can transform the $6c$ holonomy equations into and $n$ holonomy equations and $6c-k$ equations of the form $e^i_{a}+e^i_{b}$, where $\{a,b\}\in \{ \{1,2\}, \{1,3\}, \{2,3\}\}$, $i\in \{1,\dots,2c\}$ and $\mathbf{c}^i= (e^i_1,e^i_2,e^i_3)$ is a cycle.  Note that the top layer will not appear in the matrix $L$. By Remark \ref{edgeid}, it suffices to consider the edge equations modulo $\pi\sqrt{-1}$. For simplicity, for two linear combinations of edge equations $s_1$ and $s_2$, we write $s_1 \sim s_2$ if $s_1-s_2$ is a linear combination of the central edges $\{f^l_{1},f^l_{2}\}_{l=1}^{c}$ modulo $\pi \sqrt{-1}$. \\

\begin{figure}[h]
\centering
\includegraphics[scale=0.1]{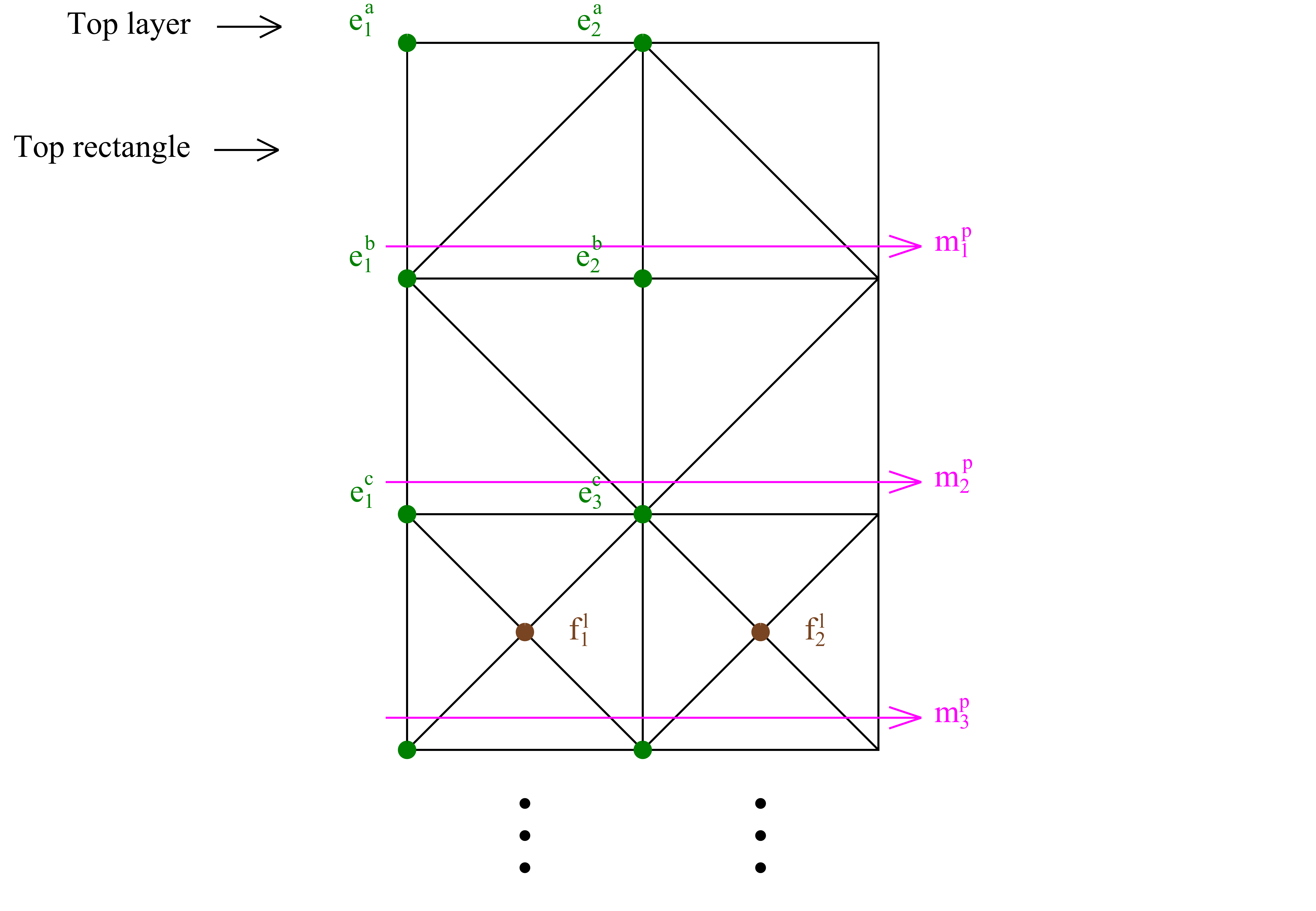}
\caption{The figure shows a possible configuration of the fundamental domain of a boundary torus $T_p$. The brown dots represent the central edges and the green dots represent the non-central edges. $m^p_1,m^p_2$ and $m^p_3$ represent three different meridian paths. In this figure, we keep the holonomy equation of $m^p_1$ in the top rectangle. Note that $m^p_{1}-m^p_{2} \equiv e^b_{1}+e^b_{2} \pmod{\pi\sqrt{-1}}$ and $m^p_{2}-m^p_{3} \equiv e^c_{1}+e^c_{2} + f^l_{1} + f^l_{2}\pmod{\pi\sqrt{-1}}$. Besides, the top layer $e^a_1+e^a_2$ will not appear in the matrix $L$.}\label{ndetl3}
\end{figure}

\noindent{\underline{\bf{Contribution from the cycle $\mathbf{c}^1$:}}}

\begin{enumerate}
\item Suppose all three edges from the cycle $\mathbf{c}^1$ are dependent edges. Then $\mathbf{c}^1$ does not contribute to the determinant of $L$.

\item Up to renaming $\{T_i\}_{i=1}^k$ if necessary, suppose we have a branch of the form 
$$T_1 \xrightarrow{\mathbf{c^1}} T_2 \xrightarrow{\mathbf{c(2,3)}} T_3 \xrightarrow{\mathbf{c(3,4)}} \dots,$$ 
where $\mathbf{c^1} = (e_1^1,e_2^1,e_3^1)$, $(e_1^1, e_3^1)\in T_1$ is on the top layer of $T_1$, $(e_1^1, e_2^1) \in T_2$ is on the top layer of $T_2$ and $(e_2^1, e_3^1) \in T_2$ is not on the top layer of $T_2$. By summing up all the edges in $T_1$, we have the equation
\begin{align}\label{layer1c2}
e_1^1 + e_3^1 + R_1 \sim 0,
\end{align}
where $R_1$ is the sum of all the edges in $T_1$ except $e_1^1$ and $e_3^1$. Similarly, by summing up all the edges in $T_2$, we have the equation
\begin{align}\label{layer2c2}
e_1^1 + 2e_2^1 + e_3^1 + R_2 \sim 0,
\end{align}
where $R_2$ is the sum of all the edges in $T_2$ except $e_1^1,e_2^1$ and $e_3^1$. Note that $R_1$ and $R_2$ are independent of $e_1^1,e_2^1$ and $e_3^1$. From (\ref{layer1c2}) and (\ref{layer2c2}), we can see that
\begin{align*}
e_2^1  \sim \frac{1}{2}(R_1-R_2)
\end{align*}
is independent of $e_1^1$ and $e_3^1$. As a result, the row corresponding to $e^1_{2}+e^1_{3}$ is of the form 
\begin{align*}
\begin{pNiceMatrix}[first-row, first-col]
 &  e^1_{3} & * & \dots & *\\
 e^1_{2}+e^1_{3} & 1& * & \dots & *
\end{pNiceMatrix}.
\end{align*}

\item
Up to renaming $\{T_i\}_{i=1}^k$ if necessary, suppose we have the configuration 
$$T_1 \xrightarrow{\mathbf{c^1}} T_2 \xrightarrow{\mathbf{c(2,3)}} T_3 \xrightarrow{\mathbf{c(3,4)}} \dots,$$ 
where $\mathbf{c^1} = (e_1^1,e_2^1,e_3^1)$, $(e_1^1, e_3^1)\in T_1$ is on the top layer of $T_1$, $(e_2^1, e_3^1) \in T_1$ is not on the top layer of $T_1$ and $(e_1^1, e_2^1) \in T_2$ is on the top layer of $T_2$.  By summing up all the edges in $T_1$, we have the equation
\begin{align}\label{layer1c3}
e_1^1 + e_2^1 + 2e_3^1 + R_1 \sim 0,
\end{align}
where $R_1$ is the sum of all the edges in $T_1$ except $e_1^1,e_2^1$ and $e_3^1$. Similarly, by summing up all the edges in $T_2$, we have the equation
\begin{align}\label{layer2c3}
e_1^1 + e_2^1 + R_2 \sim 0,
\end{align}
where $R_2$ is the sum of all the edges in $T_2$ except $e_1^1$ and $e_2^1$. Note that $R_1$ and $R_2$ are independent of $e_1^1,e_2^1$ and $e_3^1$. From (\ref{layer1c3}) and (\ref{layer2c3}), we can see that
\begin{align*}
e_3^1  \sim \frac{1}{2}(-R_0+R_1)
\end{align*}
is independent of $e_1^1$ and $e_2^1$. As a result, the row corresponding to $e^1_{2}+e^1_{3}$ is of the form 
\begin{align*}
\begin{pNiceMatrix}[first-row, first-col]
 &  e^1_{2} & * & \dots & *\\
 e^1_{2}+e^1_{3} & 1& * & \dots & *
\end{pNiceMatrix}.
\end{align*}

\item Up to renaming $\{T_i\}_{i=1}^k$ if necessary, suppose for some level $l\geq 2$, we have the configuration 
$$T_1 \xrightarrow{\mathbf{c(1,2)}} T_2 \xrightarrow{\mathbf{c(2,3)}} \dots \xrightarrow{\mathbf{c(l,l+1)}} T_{l+1}, $$
where $ (e_1^1, e_3^1)\in T_1$ is on the top layer of $T_1$, $(e_1^1, e_2^1) \in T_2$ is on the top layer of $T_2$ and $(e_2^1, e_3^1) \in T_{l+1}$ is not on the top layer of $T_{l+1}$. 
 For $j=2,\dots,k$, we let $p_j,q_j,r_j \in \{1,2,3\}$ such that for the cycle $\mathbf{c(j,j+1)}=(e_1^j,e_2^j,e_3^j)$, the edges pair $(e_{p_j}^j,e_{q_j}^j) \in T_{j+1}$ is on the top layer of $T_{j+1}$ and the edges pair $(e_{p_j}^j,e_{r_j}^j)\in T_{j}$ is not on the top layer of $T_{j}$. 

First, by summing up all the edges in $T_1$, we have the equation
\begin{align}\label{layer1}
e_1^1 + e_3^1 + R_1 \sim 0,
\end{align}
where $R_1$ is the sum of all the edges in $T_1$ except $e_1^1$ and $e_3^1$. Similarly, by summing up all the edges in $T_2$, we have the equation
\begin{align}\label{layer2} 
e_1^1 + e_2^1 + e_{p_2}^2 + e_{r_2}^2 +  R_2 \sim 0,
\end{align}
where $R_2$ is the sum of all the edges in $T_2$ except $e_1^1, e_2^1, e_{p_2}^2$ and $e_{r_2}^2$. By the same reason, for $j=3,\dots,l$, we have the equation
\begin{align}\label{layerj}
e_{p_{j-1}}^{j-1} + e_{q_{j-1}}^{j-1} + e_{p_{j}}^{j} + e_{r_{j}}^{j} + R_j = 0 ,\end{align}
where $R_j$ is the sum of all the other edges in $T_j$ except $e_{p_{j-1}}^{j-1} , e_{q_{j-1}}^{j-1} , e_{p_{j}}^{j}$ and $e_{r_{j}}^{j}$. Lastly, we have
\begin{align}\label{layerk}
e_{p_{l}}^{l} + e_{q_{l}}^{l} + e_{2}^{1} + e_{3}^1 + R_{l+1} = 0 ,
\end{align}
where $R_l$ is the sum of all the other edges in $T_{l+1}$ except $e_{p_{l}}^{l}, e_{q_{l}}^{l} , e_{2}^{1}$ and $e_{3}^1$.  
By successive substitutions using (\ref{layerj}) and (\ref{layerk}), we have
$$
e^2_{p_2} \sim - \sum_{j=2}^{l-1} (-1)^j (e_{p_j}^j + e_{r_{j+1}}^{j+1} + R_j) + (-1)^{l+1}(e_{q_l}^l + e_2^1 + e_3^1 + R_{l+1}).
$$
When $l$ is odd, from (\ref{layer2}), we have
\begin{align}\label{layertogetodd}
e_1^1 + 2e_2^1 + e_3^1
\sim
- e_{r_2}^2 - R_1 - e_{q_l}^l - R_l + \sum_{j=2}^{l-1} (-1)^j (e_{p_j}^j + e_{r_{j+1}}^{j+1} + R_j).
\end{align}
When $l$ is even,  from (\ref{layer2}), we have
\begin{align}\label{layertogeteven}
e_1^1 - e_3^1
\sim - e_{r_2}^2 - R_1 + e_{q_l}^l + R_l + \sum_{j=2}^{l-1} (-1)^j (e_{p_j}^j + e_{r_{j+1}}^{j+1} + R_j).
\end{align}
As a result, when $l$ is odd, from (\ref{layer1}) and (\ref{layertogetodd}), we have
\begin{align}\label{layercombodd}
e_2^1 \sim \frac{1}{2}\Big(R_0 - e_{r_2}^2 - R_1 - e_{q_l}^l - R_l + \sum_{j=2}^{l-1} (-1)^j (e_{p_j}^j + e_{r_{j+1}}^{j+1} + R_j) \Big).
\end{align}
When $l$ is even, from (\ref{layer1}) and (\ref{layertogeteven}), we have
\begin{align}\label{layercombeven1}
e_1^1 \sim \frac{1}{2}\Big(-R_0 - e_{r_2}^2 - R_1 - e_{q_l}^l - R_l + \sum_{j=2}^{l-1} (-1)^j (e_{p_j}^j + e_{r_{j+1}}^{j+1} + R_j) \Big), 
\end{align}
which by (\ref{layer1}) implies that
\begin{align}\label{layercombeven2}
e_3^1 \sim \frac{1}{2}\Big(-R_0 - e_{r_2}^2 - R_1 - e_{q_l}^l - R_l + \sum_{j=2}^{l-1} (-1)^j (e_{p_j}^j + e_{r_{j+1}}^{j+1} + R_j) \Big), 
\end{align}
Note that the expressions on the right hand sides of (\ref{layercombodd}) and (\ref{layercombeven2}) are independent of $e_1^1,e_2^1,e_3^1$. As a result, when $l$ is odd, the row corresponding to $e^1_{2}+e^1_{3}$ is of the form 
\begin{align}\label{B1a}
\begin{pNiceMatrix}[first-row, first-col]
 &  e^1_{3} & * & \dots & *\\
e^1_{2}+e^1_{3} & 1 & * & \dots & *
\end{pNiceMatrix}.
\end{align}
Similarly, when $l$ is even, the row corresponding to $e^1_{2}+e^1_{3}$ is of the form 
\begin{align}\label{B1a}
\begin{pNiceMatrix}[first-row, first-col]
 &  e^1_{2} & * & \dots & *\\
 e^1_{2}+e^1_{3} & 1& * & \dots & *
\end{pNiceMatrix}.
\end{align}
\end{enumerate}

\noindent{\underline{\bf{Contribution of the cycles $\mathbf{c}^2,\dots,\mathbf{c}^n$:}}}\\

Next, we claim that for each cycle $\mathbf{c}^i=(e_1^i,e_2^i,e_3^i)$, each of them contributes a factor of $\pm2$ to the determinant. Observe that each of them falls into one of the following possibilities:

\begin{enumerate}
\item Suppose the cycle $\mathbf{c}^i=(e^i_1,e^i_2,e^i_3)$ corresponds to exactly one edge of the tree. 
\begin{enumerate}[I)]
\item
For configuration in 1)I), by summing up all the edges in $T_b$, we have
\begin{align}
e^i_1+e^i_2  + R_b \sim 0 
\end{align}
where $R_b$ is the sum of edges in $T_b$ except $e^i_1,e^i_2$. This implies that
\begin{align}
e^i_1 + R_b \sim -e^i_2 
\end{align}
Note that $R_b$ is independent of $e^i_1,e^i_2,e^i_3$. Locally we get a block
\begin{align}
\begin{pNiceMatrix}[first-row, first-col]
 &  e^i_{2} & e^i_{3}\\
e^i_{1}+e^i_{3} & -1 & 1\\
e^i_{2}+e^i_{3} & 1 & 1
\end{pNiceMatrix}
\end{align}
with determinant $-2$. 

\item 
For configuration 1)II), by summing up all the edges in $T_b$, we have
\begin{align}
e^i_1+2e^i_2  + e^i_3 + R_b \sim 0 
\end{align}
where $R_b$ is the sum of edges in $T_b$ except $e^i_1,e^i_2,e^i_3$. This implies that
\begin{align}
e^i_1 + e^i_3 + R_b \sim -2e^i_2.
\end{align}
Note that $R_b$ is independent of $e^i_1,e^i_2,e^i_3$. Locally we get a block
\begin{align}
\begin{pNiceMatrix}[first-row, first-col]
 &  e^i_{2} & e^i_{3}\\
e^i_{1}+e^i_{3} & -2 & 0\\
e^i_{2}+e^i_{3} & 1 & 1
\end{pNiceMatrix}
\end{align}
with determinant $-2$. 

\item For configuration 1)III), by summing up all the edges in $T_c$, we have
$$
e^i_1+e^i_2+R_c \sim 0,
$$
where $R_c$ is the sum of edges in $T_c$ except $e^i_1$ and $e^i_2$. This implies
$$
e^i_1+R_c \sim -e^i_2,
$$
Note that $R_c$ does not depend on $e^i_1,e^i_2$ and $e^i_3$. Locally we get a block
\begin{align}
\begin{pNiceMatrix}[first-row, first-col]
 &  e^i_{2} & e^i_{3}\\
e^i_{1}+e^i_{3} & -1 & 1\\
e^i_{2}+e^i_{3} & 1 & 1
\end{pNiceMatrix}
\end{align}
with determinant $-2$.

\item 
For configuration 1)IV), by a similar computation as case 4, we have
$$
e^i_1+e^i_2 \sim (-1)^{d(b,c)}(e^i_2+e^i_3).
$$
When $d(b,c)$ is even, we have
$$ e^i_1 \sim e^i_3.$$
Locally we get a block
\begin{align}
\begin{pNiceMatrix}[first-row, first-col]
 &  e^i_{2} & e^i_{3}\\
e^i_{1}+e^i_{3} & 0 & 2\\
e^i_{2}+e^i_{3} & 1 & 1
\end{pNiceMatrix}
\end{align}
with determinant $-2$. 
When $d(b,c)$ is odd, we have
$$ e^i_1 + e^i_3 \sim -2e^i_2.$$
Locally we get a block
\begin{align}
\begin{pNiceMatrix}[first-row, first-col]
 &  e^i_{2} & e^i_{3}\\
e^i_{1}+e^i_{3} & -2 & 0\\
e^i_{2}+e^i_{3} & 1 & 1
\end{pNiceMatrix}
\end{align}
with determinant $-2$. 

\item For configuration 1)V), by summing up all the edges in $T_c$, we have
$$
e^i_1+e^i_2+R_c \sim 0,
$$
where $R_c$ is the sum of edges in $T_c$ except $e^i_1$ and $e^i_2$. This implies
$$
e^i_1+R_c \sim -e^i_2,
$$
Note that $R_c$ does not depend on $e^i_1,e^i_2$ and $e^i_3$. Locally we get a block
\begin{align}
\begin{pNiceMatrix}[first-row, first-col]
 &  e^i_{2} & e^i_{3}\\
e^i_{1}+e^i_{3} & -1 & 1\\
e^i_{2}+e^i_{3} & 1 & 1
\end{pNiceMatrix}
\end{align}
with determinant $-2$. 
\end{enumerate}
\item By summing up all the edges in $T_c$, 
we have
$$
e^i_1+e^i_2+R_c \sim 0,
$$
where $R_c$ is the sum of edges in $T_c$ except $e^i_1$ and $e^i_2$. 
 By summing up all the edges in $T_b$, 
$$
e^i_1+e^i_3+R_b \sim 0,
$$
where $R_b$ is the sum of edges in $T_b$ except $e^i_1$ and $e^i_3$. As a result,
$$ e^i_2 + e^i_3 \sim -2e^i_1 -R_b-R_c.$$
Note that $R_b$ and $R_c$ do not depend on $e^i_1,e^i_2$ and $e^i_3$. Locally we get a block
\begin{align}
\begin{pNiceMatrix}[first-row, first-col]
 &  e^i_{1} \\
e^i_{2}+e^i_{3} & -2
\end{pNiceMatrix}
\end{align}
with determinant $-2$. 
\item Suppose no edge from the cycle is a dependent edge. Locally, we get a block 
\begin{align*}
\begin{pNiceMatrix}[first-row, first-col]
 & e^i_{1} & e^i_{2} & e^i_{3} \\
e^i_{1}+e^i_{2} & 1 & 1 & 0\\
e^i_{1}+e^i_{3} & 1 & 0 & 1\\
e^i_{2}+e^i_{3} & 0 & 1 & 1
\end{pNiceMatrix}
\end{align*}
with determinant $-2$.
\end{enumerate}
Altogether, by applying row operations that do not change the determinant of $L$, we transform $L$ into an upper triangular block matrix. Moreover, $2c-1$ cycles give blocks with determinant $\pm 2$ and 1 cycle gives a block with determinant $\pm 1$. As a result, we have $\det L = \pm 2^{2c-1}$.
\end{proof}

\subsection{Proof of Theorem \ref{COC1loop}}
Assume that $\tau(M,\boldsymbol\alpha, \mathbf{z}, \mathcal{T})\neq 0$. By the definition of the 1-loop invariant, we know that 
\begin{align*}
\frac{\tau(M,\boldsymbol\alpha', \mathbf{z}, \mathcal{T})}{\tau(M,\boldsymbol\alpha, \mathbf{z}, \mathcal{T})}
=  \frac{\mathrm{det}(D_{\mathbf z}\mathcal{G}_{\boldsymbol\alpha'}(\mathbf z))}{\mathrm{det}(D_{\mathbf z}\mathcal{G}_{\boldsymbol\alpha}(\mathbf z))},
\end{align*}
where $\mathcal{G}_{\boldsymbol\alpha}$ and $\mathcal{G}_{\boldsymbol\alpha'}$ are the function $\mathcal{G}$ defined in (\ref{defG1loop}) with respect to the system of simple closed curves $\boldsymbol\alpha$ and $\boldsymbol\alpha'$ respectively. By the inverse function theorem, since $\mathrm{det}(D_{\mathbf z}\mathcal{G}_{\boldsymbol\alpha}(\mathbf z))\neq 0$ by assumption, locally $\mathcal{G}_{\boldsymbol\alpha}$ defines a biholomorphism around $\mathbf{z}$. Recall that the first $n-k$ entries of $\mathcal{G}_{\boldsymbol\alpha}$ and $\mathcal{G}_{\boldsymbol\alpha'}$ correspond to the $n-k$ edge equations of the ideal triangulation. Besides, the last $k$ entries of $\mathcal{G}_{\boldsymbol\alpha}$ and $\mathcal{G}_{\boldsymbol\alpha'}$ correspond to the $k$ (logarithmic) holonomy equations along the system of simple closed curves $\boldsymbol\alpha$ and $\boldsymbol\alpha'$ respectively. In particular, the Jacobian of the holomorphic function $\mathcal{G}_{\boldsymbol\alpha'} \circ \mathcal{G}_{\boldsymbol \alpha}^{-1}$ is of the form
\begin{align*}
D\Big(\mathcal{G}_{\boldsymbol\alpha'} \circ \mathcal{G}_{\boldsymbol\alpha}^{-1} \Big) 
=
\begin{pmatrix}
\text{I}_{n-k} & 0 \\
* & \bigg( \frac{\partial \mathrm H(\alpha'_i)}{\partial \mathrm H(\alpha_j)}\bigg)_{ij}
\end{pmatrix}
\end{align*}
with 
$$\det \Big(D\Big(\mathcal{G}_{\alpha'} \circ \mathcal{G}_{\boldsymbol\alpha}^{-1} \Big) \Big) = \det \bigg( \frac{\partial \mathrm H(\alpha'_i)}{\partial \mathrm H(\alpha_j)}\bigg)_{ij}
,$$
where $\text{I}_{n-k}$ is the $(n-k)\times (n-k)$ identity matrix. Moreover, by chain rule, we have 
$$\det \Big(D\Big(\mathcal{G}_{\boldsymbol\alpha'} \circ \mathcal{G}_{\boldsymbol\alpha}^{-1} \Big) (\mathcal{G}_{\boldsymbol\alpha}(\mathbf{z}) ) \Big) =\frac{\mathrm{det}(D_{\mathbf z}\mathcal{G}_{\boldsymbol\alpha'}(\mathbf z))}{\mathrm{det}(D_{\mathbf z}\mathcal{G}_{\boldsymbol\alpha}(\mathbf z))}.$$
Altogether, we have
$$
\frac{\tau(M,\boldsymbol\alpha', \mathbf{z}, \mathcal{T})}{\tau(M,\boldsymbol\alpha, \mathbf{z}, \mathcal{T})}
=  \det \bigg( \frac{\partial \mathrm H(\alpha'_i)}{\partial \mathrm H(\alpha_j)}\bigg)_{ij}.
$$

\section{Topological invariance of the 1-loop invariant}\label{Sinv02}
\subsection{Invariance under 0-2 move}
The main goal of this section is to prove the following proposition.
\begin{proposition}\label{inv02move}
Let $\mathcal{T}$ and $\mathcal{T}'$ be two $\rho$-regular triangulations that are related by a 0-2 Pachner move. 
Suppose there exists $\mathbf{z}\in \mathcal{V}_{\mathcal{T}}$ and $\mathbf{z'}\in \mathcal{V}_{\mathcal{T}'}$ such that $\mathcal{P}_{\mathcal{T}} (\mathbf{z}) = \mathcal{P}_{\mathcal{T}'} (\mathbf{z}')$ and $\mathbf{z}$ is obtained from $\mathbf{z'}$ by removing the shape parameters corresponding to the two tetrahedra in $\mathcal{T}' \setminus \mathcal{T}$.
Then 
$$\tau(M,\boldsymbol\alpha, \mathbf{z}, \mathcal{T}) = \tau(M,\boldsymbol\alpha, \mathbf{z'}, \mathcal{T}').$$
\end{proposition}

\begin{proof}
\begin{figure}[h]
\centering
\includegraphics[scale=0.16]{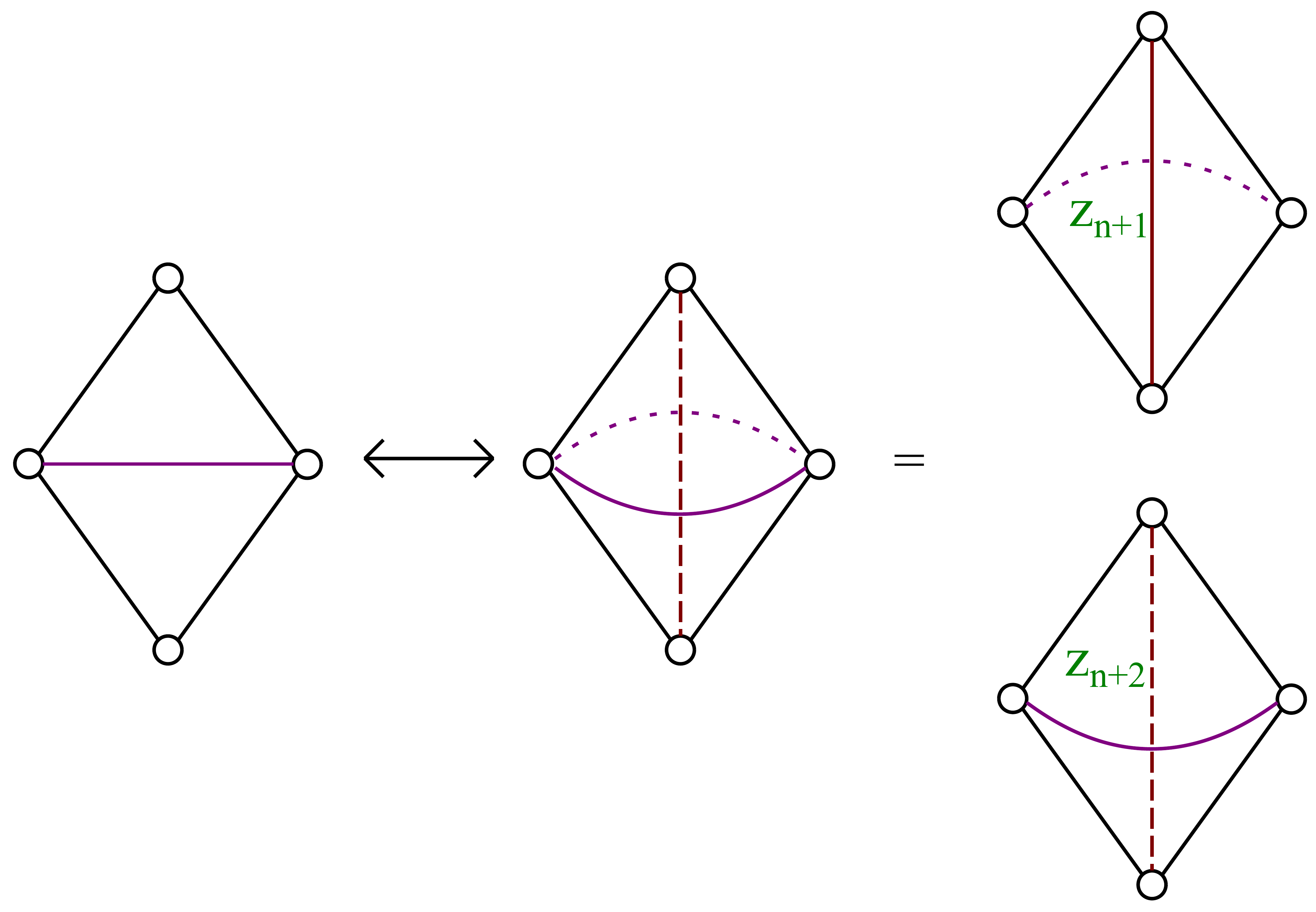}
\caption{The figure shows a 0-2 move that changes the triangulation $\mathcal{T}$ to $\mathcal{T}'$. We assign shape parameters $z_{n+1}$ and $z_{n+2}$ to the new edge as shown on the right.}\label{02move3}
\end{figure}

\begin{figure}[h]
\centering
\includegraphics[scale=0.16]{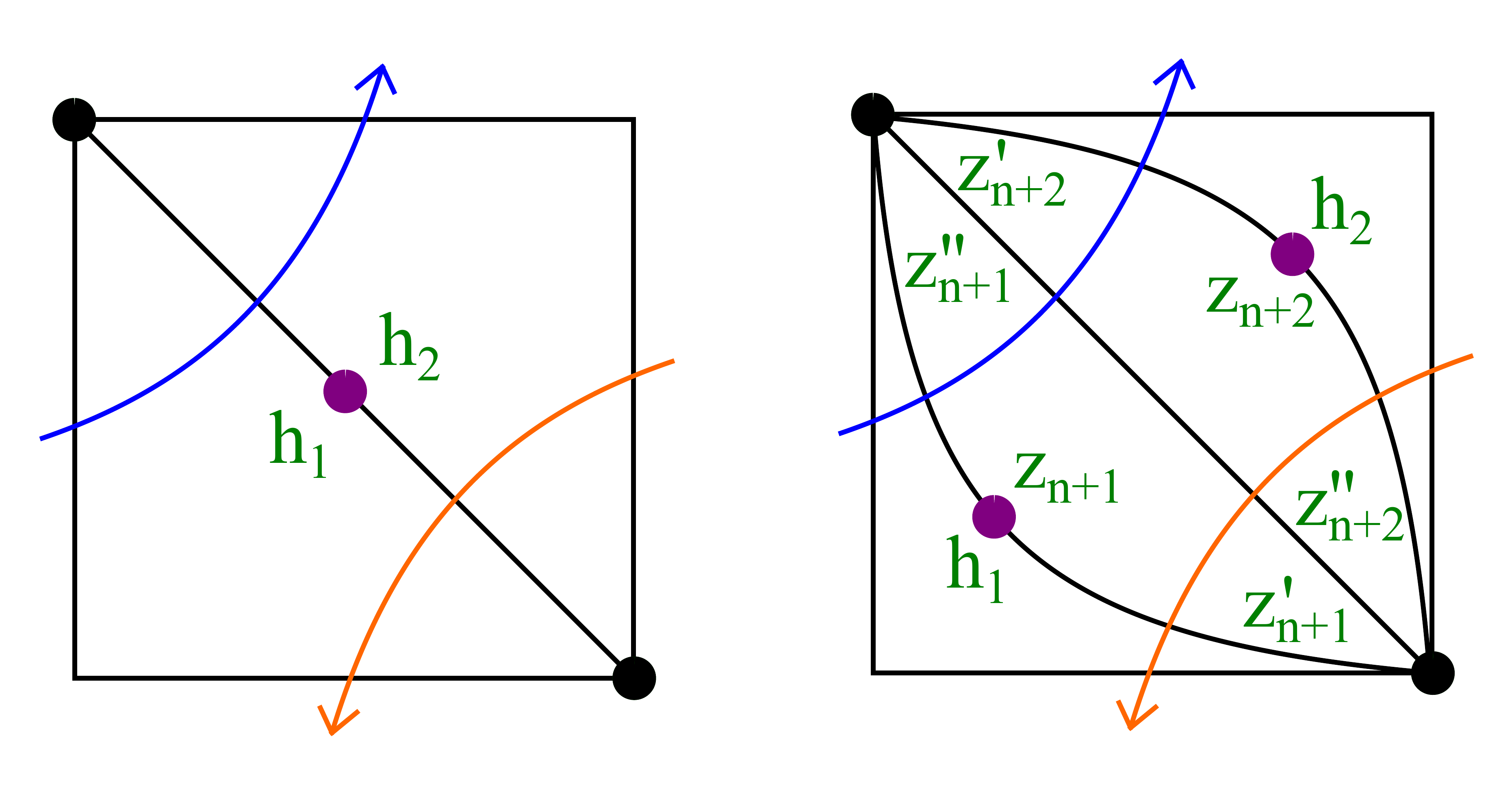}
\caption{This figure illustrates how the triangulations of the boundary torus corresponding to the left and right ideal vertices in Figure \ref{02move3} change under a 0-2 move. The holonomies of the blue and orange curves change by a multiples of $\log z_{n+1}' + \log z_{n+2}''$ and $\log z_{n+1}'' + \log z_{n+2}'$ respectively, where the multiples depend on the number of times the curves crosses the diagonal.}\label{02move2}
\end{figure}

\begin{figure}[h]
\centering
\includegraphics[scale=0.16]{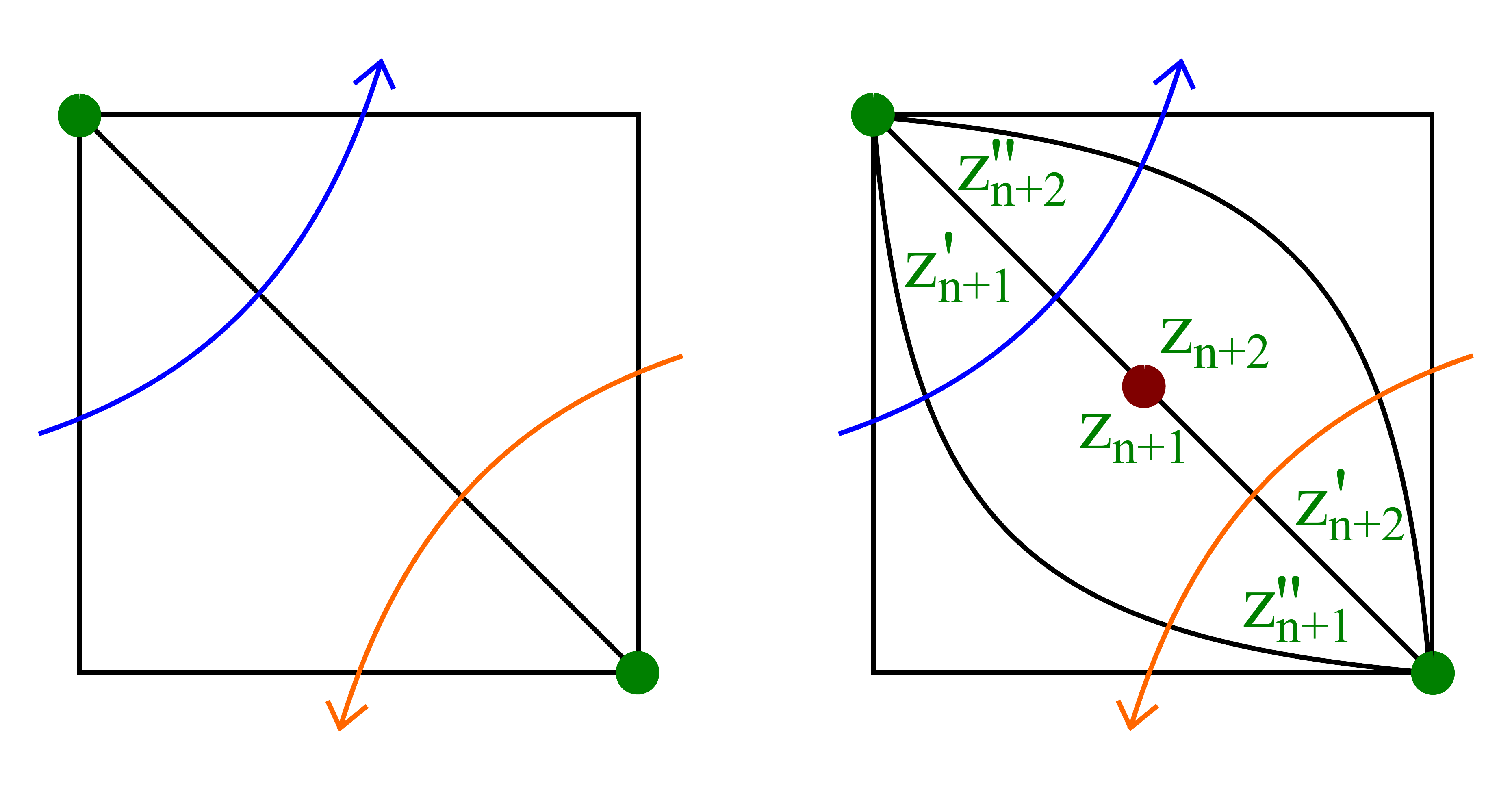}
\caption{This figure illustrates how the triangulations of the boundary torus corresponding to the top and bottom ideal vertices in Figure \ref{02move3} change under a 0-2 move. The holonomies of the blue and orange curves change by a multiples of $\log z_{n+1}' + \log z_{n+2}''$ and $\log z_{n+1}'' + \log z_{n+2}'$ respectively, where the multiples depend on the number of times the curves crosses the diagonal.}\label{02move1}
\end{figure}

Let $M$ be a 3-manifold with $\partial M = T^1 \coprod \dots \coprod T^k$ and let $\rho: \pi_1(M) \to \mathrm{PSL}(2;\CC)$. 
Let $\mathcal{T}=\{\Delta_i\}_{i=1}^n$ and $\mathcal{T}'=\{\Delta_i\}_{i=1}^n \cup \{\Delta_{n+1}, \Delta_{n+2}\}$ be two $\rho$-regular triangulations that are related by a 0-2 move as shown in Figure \ref{02move3}. Let $z_i$ be the shape parameter assigned to $\Delta_i$ for $i=1,2,\dots,n$. Denote the new edge in $\mathcal{T}'$ by $e_{n-k+2}$. Assign shape parameters $z_{n+1}$ and $z_{n+2}$ as shown in Figure \ref{02move3}. Note that the edge equation around $e_{n-k+2}$ is given by 
$$e_{n-k+2}: \qquad \log z_{n+1} + \log z_{n+2} = 2\pi\sqrt{-1},$$
which corresponds to a row
\begin{align}\label{e_{n-k+2}}
\begin{pNiceMatrix}[first-row, first-col]
 &  z_1& \dots &z_n  & z_{n+1} & z_{n+2}  \\
 e_{n-k+2} &0  &\dots & 0 & \frac{1}{z_{n+1}} & \frac{1}{z_{n+2}} \\
 \end{pNiceMatrix}
\end{align}
in the computation of the 1-loop invariant. 
In particular, we have
\begin{align}\label{zprodeq1}
z_{n+1}z_{n+2} = 1.
\end{align}
Besides, consider the triangulations of the boundary torus corresponding to the left and right ideal vertices in Figure \ref{02move3}. As shown in Figure \ref{02move2}, the edge equation around the purple edge, denoted by $e_{n-k}$, is changed from
$$e_{n-k}:  h_1 + h_2 = 2\pi\sqrt{-1} $$
into two edge equations $e_{n-k}^1$ and $e_{n-k}^2$ given by
$$e_{n-k}^1: \qquad h_1 + \log z_{n+1} = 2\pi\sqrt{-1}$$
and 
$$e_{n-k}^2: \qquad h_2 + \log z_{n+2} = 2\pi\sqrt{-1},$$
where $h_1$ and $h_2$ are the sums of the logarithms of the shape parameters on the left and right hand sides of the diagonal respectively.
Note that $e_{n-k}^1$ and $e_{n-k}^2$ corresponds to the row
\begin{align}
 \begin{pNiceMatrix}[first-row, first-col]
 &  z_1& \dots &z_n  & z_{n+1} & z_{n+2}  \\
 e_{n-k}^1 &*  &\dots & * & \frac{1}{z_{n+1}} & 0 \\
 \end{pNiceMatrix}
\end{align}
and
\begin{align}
\begin{pNiceMatrix}[first-row, first-col]
 &  z_1& \dots &z_n  & z_{n+1} & z_{n+2}  \\
 e_{n-k}^2 &*  &\dots & * & 0 & \frac{1}{z_{n+2}} \\
 \end{pNiceMatrix}
\end{align}
respectively. Moreover, as shown in Figures \ref{02move2} and \ref{02move1}, 
each of the remaining edge equation $e_i$ is either unchanged, or is changed by adding a multiple of $\pm ( \log z_1' + \log z_2'')$ or $\pm( \log z_1'' + \log z_2') $. In particular, each of them correspond to a row of the form
\begin{align}
\begin{pNiceMatrix}[first-row, first-col]
 &  z_1& \dots &z_n  & z_{n+1} & z_{n+2}  \\
 e_{i} &*  &\dots & * & 0 & 0 \\
 \end{pNiceMatrix},
\end{align}
\begin{align}\label{e_i1}
 \begin{pNiceMatrix}[first-row, first-col]
 &  z_1& \dots &z_n  & z_{n+1} & z_{n+2}  \\
 e_{i} &*  &\dots & * & \frac{k}{1-z_1} &  \frac{k}{z_2(z_2+1)} \\
 \end{pNiceMatrix},
\end{align}
or
\begin{align}\label{e_i2}
\begin{pNiceMatrix}[first-row, first-col]
 &  z_1& \dots &z_n  & z_{n+1} & z_{n+2}  \\
 e_{i} &*  &\dots & * &  \frac{k}{z_1(z_1+1)} &  \frac{k}{1-z_2} \\
 \end{pNiceMatrix}.
\end{align}
for some $k\in \ZZ$.
Similarly, each meridian equation $\alpha_i$ correspond to a row of the form
\begin{align}
\begin{pNiceMatrix}[first-row, first-col]
 &  z_1& \dots &z_n  & z_{n+1} & z_{n+2}  \\
 \alpha_{i} &*  &\dots & * & 0 & 0 \\
 \end{pNiceMatrix},
\end{align}
\begin{align}\label{alpha_i1}
 \begin{pNiceMatrix}[first-row, first-col]
 &  z_1& \dots &z_n  & z_{n+1} & z_{n+2}  \\
 \alpha_{i} &*  &\dots & * &  \frac{l}{1-z_1} &  \frac{l}{z_2(z_2+1)} \\
 \end{pNiceMatrix},
\end{align}
or
\begin{align}\label{alpha_i2}
\begin{pNiceMatrix}[first-row, first-col]
 &  z_1& \dots &z_n  & z_{n+1} & z_{n+2}  \\
 \alpha_{i} &*  &\dots & * &  \frac{l}{z_1(z_1+1)} &  \frac{l}{1-z_2} \\
 \end{pNiceMatrix}
\end{align}
for some $l\in\ZZ$.
Note that by (\ref{zprodeq1}), we have
$$ \det \begin{pmatrix} \frac{1}{z_{n+1}} & \frac{1}{z_{n+2}} \\ \frac{1}{1-z_{n+1}} & \frac{1}{z_{n+2}(z_{n+2}-1)} \end{pmatrix}
= \det \begin{pmatrix} \frac{1}{z_{n+1}} & \frac{1}{z_{n+2}} \\ \frac{1}{z_{n+1}(z_{n+1}-1)} & \frac{1}{1-z_{n+2}} \end{pmatrix} 
= 0.$$
In particular, by using row operations, we can use (\ref{e_{n-k+2}}) to eliminate the $z_{n+1}$ and $z_{n+2}$ coordinates of (\ref{e_i1}), (\ref{e_i2}), (\ref{alpha_i1}) and (\ref{alpha_i2}). 

As a result, for the triangulation $\mathcal{T}'$, if we let $\{e_1,e_2,\dots, e_{n-k-1}, e_{n-k}^1, e_{n-k}^2, e_{n-k+2}\}$ be a set of independent edges and $\{\alpha_1,\dots,\alpha_k\}$ be a system of boundary curves, then $\tau(M,\boldsymbol \alpha,\rho, \mathcal{T}')$ is given by
\begin{align}\label{02eq1}
\pm \frac{\det \left(
\begin{pNiceMatrix}[first-row, first-col]
 &  z_1& \dots &z_n  & z_{n+1} & z_{n+2}  \\
\alpha_1 &*  &\dots & * & 0 & 0 \\
\vdots  & \vdots  &\dots & \vdots & \vdots & \vdots  \\
\alpha_{k} &*  &\dots & * & 0 & 0 \\
e_1 &*  &\dots & * & 0 & 0 \\
\vdots  & \vdots  &\dots & \vdots & \vdots & \vdots  \\
e_{n-k-1} &*  &\dots & * & 0 & 0 \\
e_{n-k}^1 & * & \dots & * & \frac{1}{z_{n+1}} & 0\\
e_{n-k}^2 & * & \dots & * & 0 & \frac{1}{z_{n+2}} \\
e_{n-k+2} & 0 & \dots & 0 & \frac{1}{z_{n+1}} & \frac{1}{z_{n+2}}
\end{pNiceMatrix}
\right)}{\Big(\prod_{i=1}^{n} \xi_i^{f_i} \xi_i'^{f_i'} \xi_i''^{f_i''}
\Big)
\Big(\prod_{i=n+1}^{n+2} \xi_i^{f_i} \xi_i'^{f_i'} \xi_i''^{f_i''}
\Big)}
=&\ \ \pm \frac{\det \left(
\begin{pNiceMatrix}[first-row, first-col]
 &  z_1& \dots &z_n  & z_{n+1} & z_{n+2}  \\
 \alpha_1 &*  &\dots & * & 0 & 0 \\
\vdots  & \vdots  &\dots & \vdots & \vdots & \vdots  \\
\alpha_{k} &*  &\dots & * & 0 & 0 \\
e_1 &*  &\dots & * & 0 & 0 \\
\vdots  & \vdots  &\dots & \vdots & \vdots & \vdots  \\
e_{n-k-1} &*  &\dots & * & 0 & 0 \\
e_{n-k} & * & \dots & * & 0 & 0\\
e_{n-k}^2 & * & \dots & * & 0 & \frac{1}{z_{n+2}} \\
e_{n-k+2} & 0 & \dots & 0 & \frac{1}{z_{n+1}} & \frac{1}{z_{n+2}}
\end{pNiceMatrix}
\right)}{\Big(\prod_{i=1}^{n} \xi_i^{f_i} \xi_i'^{f_i'} \xi_i''^{f_i''}
\Big)
\Big(\prod_{i=n+1}^{n+2} \xi_i^{f_i} \xi_i'^{f_i'} \xi_i''^{f_i''}
\Big)},
\end{align}
where the equality follows by adding $e_{n-k-1}^2$ to $e_{n-k-1}^1$ and removing the $z_{n+1},z_{n+2}$ entries by using $e_{n-k+2}$. 
\begin{figure}[h]
\centering
\includegraphics[scale=0.16]{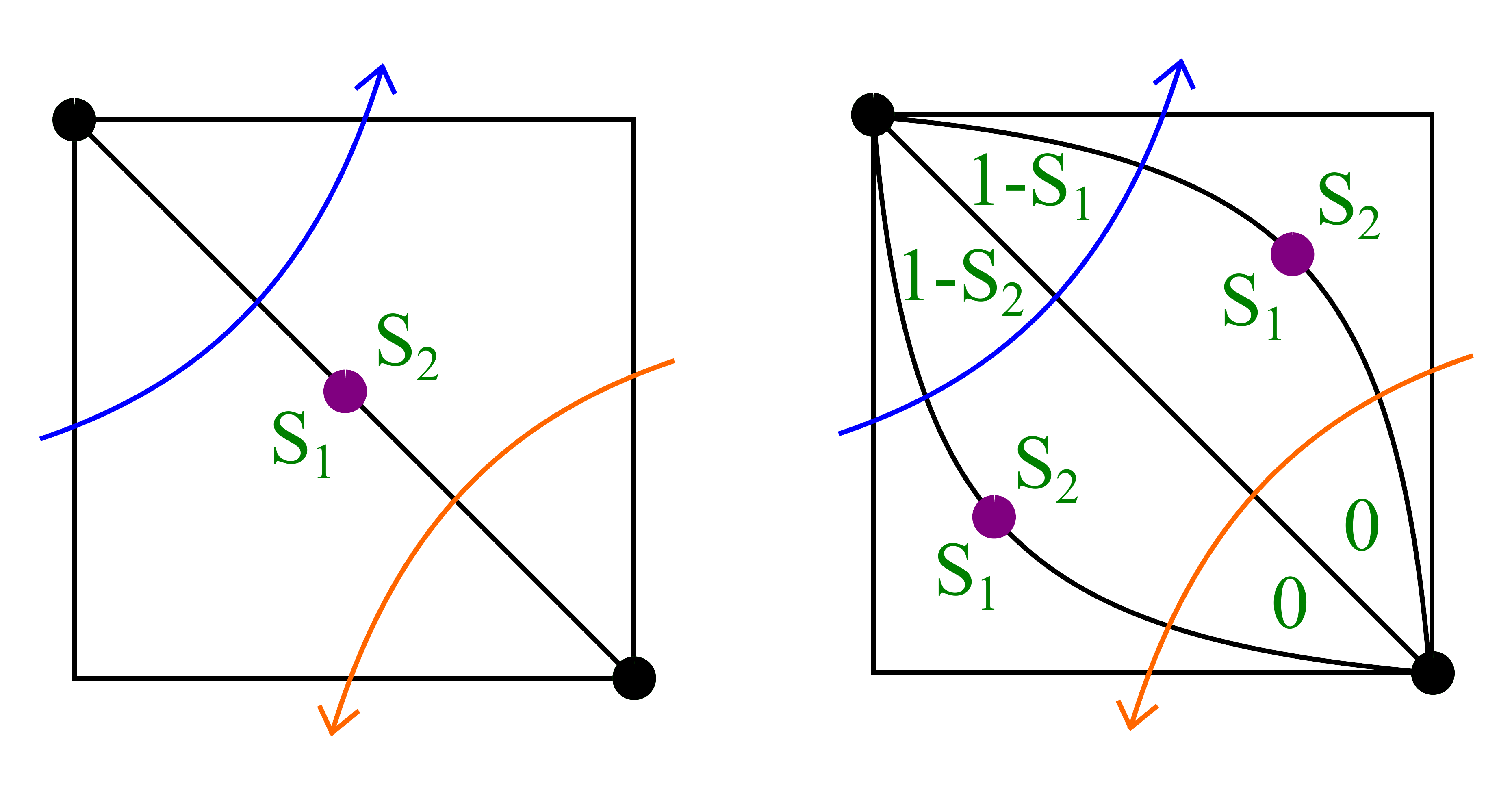}
\includegraphics[scale=0.16]{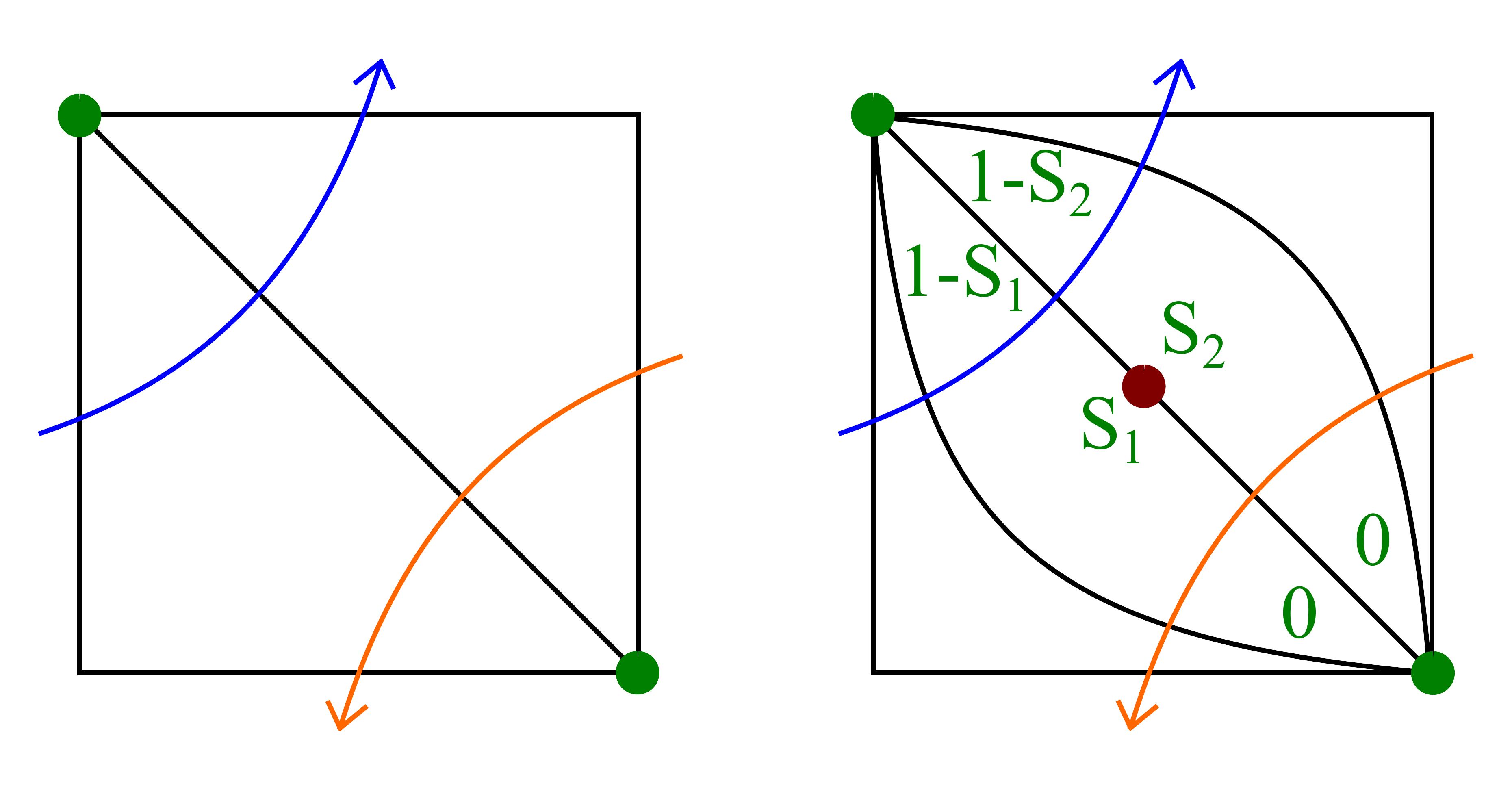}
\caption{This figure shows how the combinatorial flattening ${\boldsymbol{\mathcal{F}}}$ of $\mathcal{T}$ is extended to a combinatorial flattening ${\boldsymbol{\mathcal{F}}}' $ of $\mathcal{T}'$. Note that $S_1+S_2=2$ and $(1-S_1)+(1-S_2) = 2 - S_1 - S_2 = 0$.}\label{02move2combflat}
\end{figure}
Besides, let ${\boldsymbol{\mathcal{F}}} = (\mathbf{f}, \mathbf{f}', \mathbf{f''})$ be a strong combinatorial flattening of $\mathcal{T}$. As shown in Figure \ref{02move2combflat}, we can let $S_1, S_2$ be the sum of the combinatorial flattenings on each side of the diagonal, and extend ${\boldsymbol{\mathcal{F}}} $ to a strong combinatorial flattening  ${\boldsymbol{\mathcal{F}}}' $ of $\mathcal{T}'$ by defining $(f_{n+1},f_{n+1}', f_{n+1}'') = (S_2, 0, 1-S_2)$ and $(f_{n+2},f_{n+2}', f_{n+2}'') = (S_1, 1-S_1, 0)$. In particular, by using $z_{n+1}z_{n+2}=1$ and $S_1+S_2=2$,
\begin{align*}
\prod_{i=n+1}^{n+2} \xi_i^{f_i} \xi_i'^{f_i'} \xi_i''^{f_i''}
&= \bigg(\frac{1}{z_{n+1}}\bigg)^{S_2} \bigg(\frac{1}{z_{n+1}(z_{n+1}-1)} \bigg)^{1-S_2} \bigg(\frac{1}{z_{n+2}}\bigg)^{S_1} \bigg(\frac{1}{1-z_{n+2}} \bigg)^{1-S_1} \notag\\
&= \bigg(\frac{1}{z_{n+1}-1} \bigg)^{1-S_2}\bigg(\frac{z_{n+2}}{z_{n+2}-1} \bigg)^{1-S_1} \notag \\
&= \bigg(\frac{1}{z_{n+1}-1} \bigg)^{1-S_2} \bigg(\frac{\frac{1}{z_{n+1}}}{1-\frac{1}{z_{n+1}}} \bigg)^{1-S_1}\notag\\
&= \bigg(\frac{1}{z_{n+2}-1} \bigg)^{2-S_1-S_2} \notag\\
&=1,
\end{align*}
which implies that
\begin{align}\label{combin02}
\prod_{i=1}^{n+2} \xi_i^{f_i} \xi_i'^{f_i'} \xi_i''^{f_i''}
=
\prod_{i=1}^{n} \xi_i^{f_i} \xi_i'^{f_i'} \xi_i''^{f_i''}.
\end{align}
Altogether, by using (\ref{zprodeq1}), (\ref{02eq1}) and (\ref{combin02}), we have
\begin{align*}
\tau(M,\boldsymbol \alpha,\mathbf{z}, \mathcal{T}')
=&\ \ \pm \frac{\det \left(
\begin{pNiceMatrix}[first-row, first-col]
 &  z_1& \dots &z_n  & z_{n+1} & z_{n+2}  \\
 \alpha_1 &*  &\dots & * & 0 & 0 \\
\vdots  & \vdots  &\dots & \vdots & \vdots & \vdots  \\
\alpha_{k} &*  &\dots & * & 0 & 0 \\
e_1 &*  &\dots & * & 0 & 0 \\
\vdots  & \vdots  &\dots & \vdots & \vdots & \vdots  \\
e_{n-k-1} &*  &\dots & * & 0 & 0 \\
e_{n-k} & * & \dots & * & 0 & 0\\
e_{n-k-1}^2 & * & \dots & * & 0 & \frac{1}{z_{n+2}} \\
e_{n-k+1} & 0 & \dots & 0 & \frac{1}{z_{n+1}} & \frac{1}{z_{n+2}}
\end{pNiceMatrix}
\right)}{\Big(\prod_{i=1}^{n} \xi_i^{f_i} \xi_i'^{f_i'} \xi_i''^{f_i''}
\Big)
\Big(\prod_{i=n+1}^{n+2} \xi_i^{f_i} \xi_i'^{f_i'} \xi_i''^{f_i''}
\Big)}\\
=&\ \ \pm \frac{\det \left(
\begin{pNiceMatrix}[first-row, first-col]
 &  z_1& \dots &z_n   \\
 \alpha_1 &*  &\dots & *  \\
\vdots  & \vdots  &\dots & \vdots   \\
\alpha_{k} &*  &\dots & *   \\
e_1 &*  &\dots & *  \\
\vdots  & \vdots  &\dots & \vdots   \\
e_{n-k-1} &*  &\dots & *   \\
e_{n-k} & * & \dots & *  
\end{pNiceMatrix}
\right)}{\prod_{i=1}^{n} \xi_i^{f_i} \xi_i'^{f_i'} \xi_i''^{f_i''}
}\\
=& \tau(M,\boldsymbol \alpha,\mathbf{z}, \mathcal{T}).
\end{align*}
This completes the proof.
\end{proof}

\subsection{Proof of Theorem \ref{1loopreallyinv}}\label{p1loopreallyinv}
By Proposition \ref{biratiso}, we can regard $\mathbf{z}_2$ as a rational map defined on a Zariski open subset $W \subset \mathcal{V}_{\mathcal{T}_1}$. 
By \cite[Theorem 1.4, 4.1]{DG} and Proposition \ref{inv02move}, for any $\mathbf{z_1} \in W$, we have
$$\tau(M,\boldsymbol \alpha, \mathbf{z_1}, \mathcal{T}_1) = \tau(M,\boldsymbol \alpha, \mathbf{z_2}(\mathbf{z_1}), \mathcal{T}_2).$$
By continuity of both sides, we have the desired result.

\subsection{Proof of Corollary \ref{1loopreallyinvcor}}
Note that Conjecture \ref{1loopconjstatementori} follows immediately from Theorem \ref{mainthm} and \ref{1loopreallyinv}. 
Suppose Conjecture \ref{1loopconjstatement} is true for some $\rho_0$-regular ideal triangulation $\mathcal{T}_1$. Let $\mathcal{T}_2$ be another $\rho_0$-regular ideal triangulation. By Proposition \ref{biratiso}, we can regard $\mathbf{z}_1$ as a rational map defined on a Zariski open subset $W \subset \mathcal{V}_0(\mathcal{T}_2)$. 
By assumption and Theorem \ref{1loopreallyinv}, for any $\mathbf{z_2} \in W$ and $\rho_{\mathbf{z}_2} = \mathcal{P}_{\mathcal{T}_2}(\mathbf{z}_2)$, we have
$$\tau(M,\boldsymbol \alpha, \mathbf{z_2}, \mathcal{T}_2) = \tau(M,\boldsymbol \alpha, \mathbf{z_1}(\mathbf{z_2}), \mathcal{T}_1)
= \pm \mathbb T_{(M,\boldsymbol\alpha)}([\rho_{\mathbf{z}}]).$$
By continuity of both sides, we have the desired result.

\section{Surgery formula with respect to nice triangulations}\label{surgforS}
\subsection{Proof of Theorem \ref{1loopsurg}} 
We first consider the case where $l=1$, i.e. we only fill the first torus boundary $T_1$. The general case follows from the same argument by doing the Dehn-fillings one by one. In the context of Theorem \ref{1loopsurg}, for all the ideal triangulations in the rest of this section, we always assume that solutions of gluing equations exist.

Let $\widehat{\mathcal{T}'}= \{\Delta_i\}_{i=1}^n$ and $\widehat{\mathcal{T}}= \widehat{\mathcal{T}} \cup \{\Delta_1^1, \Delta_2^1\}$ be respectively the triangulations of $M'$ and $M$  described in Proposition \ref{niceidealintro}, where $\Delta_1^1, \Delta_2^1$ are the only two ideal tetrahedra intersecting $T_1$. Let $z_1,\dots,z_n$ be an assignment of shape parameters to $\Delta_1,\dots,\Delta_n$. Let $f,g,h$ be the edges of the base triangles and assign shape parameters $z_{n+1},z_{n+2}$ to $\Delta_1^1, \Delta_2^1$ as shown in Figure \ref{niceideal2}. We first compare the determinants of the Jacobian of the gluing maps before and after the Dehn-filling. For the representation $\rho = \rho'|_{\pi_1(M)}$, since the two bottom faces of $\Delta_1^1, \Delta_2^1$ are glued together in $M'$, by considering the product of the shape parameters above those two faces, we have 
\begin{align}
z_{n+1}z_{n+2}=1.
\end{align}

\begin{figure}[h]
\centering
\includegraphics[scale=0.2]{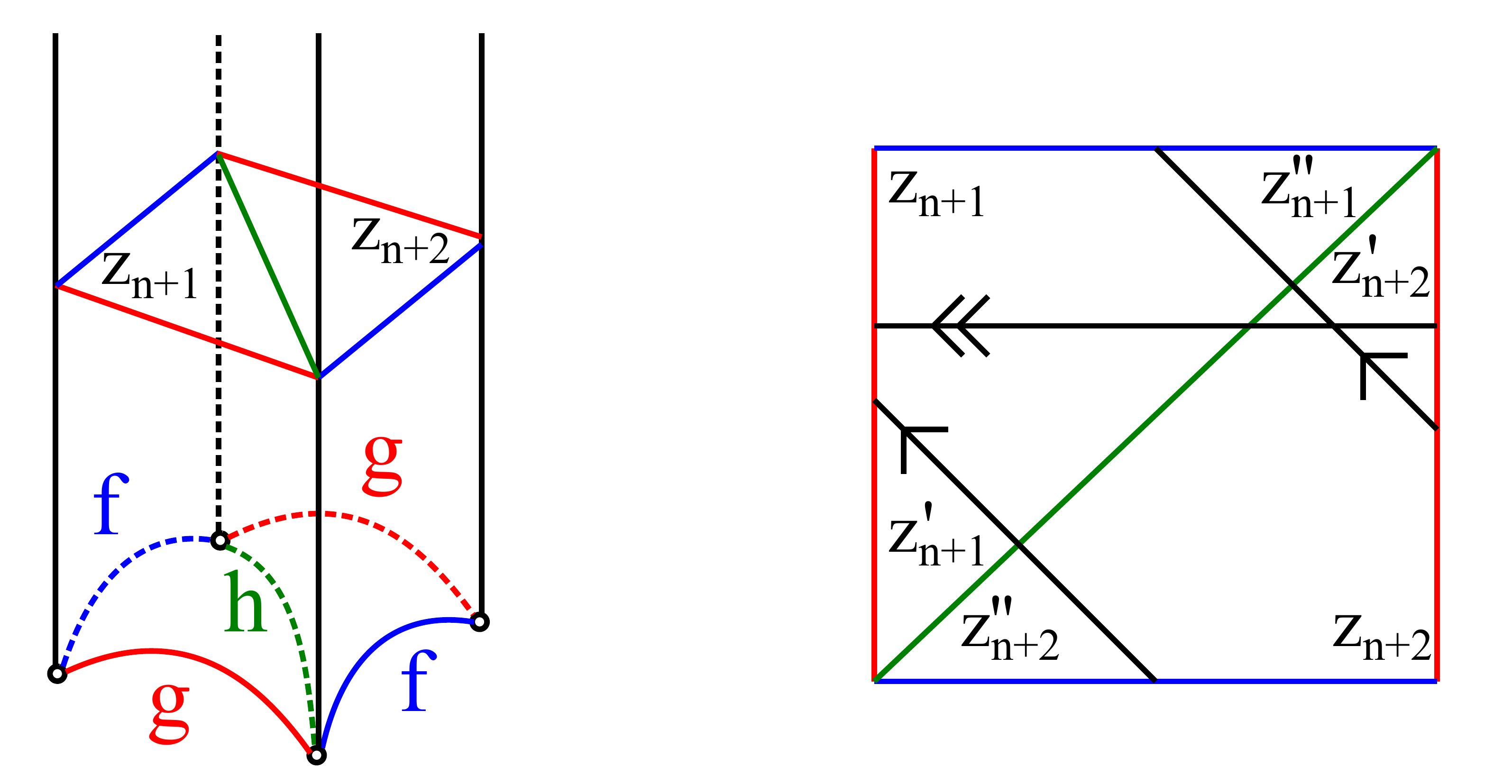}
\caption{We assign shape parameters $z_{n+1},z_{n+2}$ to $\Delta_1^1$, $\Delta_2^1$ as shown on the left of this figure. On the right, the curve with a single arrow represents the curves $\alpha=\alpha_1$. The curve $\gamma=\gamma_1$ with a double arrow represents the core curve of the Dehn-filled solid torus after doing the Dehn-filling to homotopically ``kill" $\alpha$. Note that the intersection number of $\alpha$ and $\gamma$ is $1$.}\label{niceideal2}
\end{figure}

By Remark \ref{edgeid}, we ignore all possible multiples of $\pi \sqrt{-1}$ in the following computation.
Consider the edge equations around edges $f$ and $g$. Note that they can be written in the form
\begin{align}
f &= \log z_{n+1}' + \log z_{n+2}' + \text{sum of logarithms of shape parameters of tetrahedra in $\{\Delta_i\}_{i=1}^n$},\label{compufandg1}\\
g &= \log z_{n+1}'' + \log z_{n+2}'' + \text{sum of logarithms of shape parameters of tetrahedra in $\{\Delta_i\}_{i=1}^n$}.\label{compufandg2}
\end{align}
Note that in Figure \ref{niceideal2}, the curve with a single arrow represents $\alpha=\alpha_1$ and the curve with a double arrow represents the core curve $\gamma=\gamma_1$. Moreover, the logarithmic holonomies of $\alpha$ and $\gamma$ are given by
\begin{align}
\alpha &= \log z_{n+1}' + \log z_{n+2}'' - \log z_{n+1}'' - \log z_{n+2}', \label{compualpha}\\
\gamma &= \log z_{n+1}' - \log z_{n+2}' \label{compugamma}.
\end{align}
Locally, by (\ref{compufandg1}), (\ref{compufandg2}) and (\ref{compualpha}), the Jacobian of the gluing map is given by
\begin{align}
\begin{pNiceMatrix}[first-row, first-col]
 & z_{n+1} & z_{n+2}  \\
f & \frac{1}{1-z_{n+1}} & \frac{1}{1-z_{n+2}} \\
g & \frac{1}{z_{n+1}(z_{n+1}-1)} & \frac{1}{z_{n+2}(z_{n+2}-1)} \\
\alpha & \frac{1}{1-z_{n+1}} - \frac{1}{z_{n+1}(z_{n+1}-1)} & \frac{1}{z_{n+2}(z_{n+2}-1)}  - \frac{1}{1-z_{n+2}}  
\end{pNiceMatrix}.
\end{align}
By adding the second row to the first row, we have
\begin{align}\label{comf+g}
\begin{pNiceMatrix}[first-row, first-col]
 & z_{n+1} & z_{n+2}  \\
f+g & -\frac{1}{z_{n+1}} & -\frac{1}{z_{n+2}} \\
g & \frac{1}{z_{n+1}(z_{n+1}-1)} & \frac{1}{z_{n+2}(z_{n+2}-1)} \\
\alpha & \frac{1}{1-z_{n+1}} - \frac{1}{z_{n+1}(z_{n+1}-1)} & \frac{1}{z_{n+2}(z_{n+2}-1)}  - \frac{1}{1-z_{n+2}}  
\end{pNiceMatrix}
\end{align}

The following elementary computations will be used to study the change of the determinants of the Jacobians of the gluing maps.
\begin{lemma}\label{1loopsurglem}
When $z_{n+1}z_{n+2}=1$, there exists some constant $C\in \CC$ such that 
$$
\Big(\frac{1}{z_{n+1}},\frac{1}{z_{n+2}}\Big)
= C\Big(\frac{1}{1-z_{n+1}} - \frac{1}{z_{n+1}(z_{n+1}-1)}, \frac{1}{z_{n+2}(z_{n+2}-1)}  - \frac{1}{1-z_{n+2}} \Big) \in \CC^2.
$$
Besides, we have
\begin{align*}
\det
\begin{pmatrix}
\frac{1}{z_{n+1}(z_{n+1}-1)} & \frac{1}{z_{n+2}(z_{n+2}-1)} \\
\frac{1}{1-z_{n+1}} - \frac{1}{z_{n+1}(z_{n+1}-1)} & \frac{1}{z_{n+2}(z_{n+2}-1)}  - \frac{1}{1-z_{n+2}}  
\end{pmatrix}
= \frac{z_{n+1}+z_{n+2}+2}{(z_{n+1}-1)(z_{n+2}-1)}
\end{align*}
and
\begin{align*}
4\sinh^2\frac{\mathrm{H}(\gamma)}{2} 
= -(z_{n+1}+z_{n+2}+2).
\end{align*}
\end{lemma}
\begin{proof}
Note that
\begin{align*}
&\det
\begin{pmatrix}
\frac{1}{z_{n+1}} & \frac{1}{z_{n+2}} \\
\frac{1}{1-z_{n+1}} - \frac{1}{z_{n+1}(z_{n+1}-1)} & \frac{1}{z_{n+2}(z_{n+2}-1)}  - \frac{1}{1-z_{n+2}}  
\end{pmatrix}\\
=&\ \
\frac{1}{z_{n+1}z_{n+2}(z_{n+2}-1)}-\frac{1}{z_{n+1}(1-z_{n+2})}+
\frac{1}{z_{n+1}z_{n+2}(z_{n+1}-1)}-\frac{1}{z_{n+2}(1-z_{n+1})}\\
=& \ \
\frac{1}{\frac{1}{z_{n+1}}-1} - \frac{1}{z_{n+1}(1-\frac{1}{z_{n+1}})}
+ \frac{1}{z_{n+1}-1} - \frac{1}{\frac{1}{z_{n+1}}(1-z_{n+1})}\\
=&\ \ 0,
\end{align*}
where the second last equality follows from $z_{n+1}z_{n+2}=1$. This proves the first equality.
The second equality follows from direct computation using the equation $z_{n+1}z_{n+2}=1$. For the last equality, by (\ref{compugamma}),
\begin{align*}
4\sinh^2\frac{\mathrm{H}(\gamma)}{2} 
=e^{\mathrm{H}(\gamma)}+e^{-\mathrm{H}(\gamma)}-2
=\frac{z_{n+1}'}{z_{n+2}'}+\frac{z_{n+2}'}{z_{n+1}'} - 2
=\frac{1-z_{n+2}}{1-z_{n+1}}+\frac{1-z_{n+1}}{1-z_{n+2}}-2.
\end{align*}
Since $z_{n+1}z_{n+2}=1$,
\begin{align*}
4\sinh^2\frac{\mathrm{H}(\gamma)}{2} 
=\frac{1-z_{n+2}}{1-\frac{1}{z_{n+2}}}+\frac{1-z_{n+1}}{1-\frac{1}{z_{n+1}}}-2
=-(z_{n+1}+z_{n+2}+2).
\end{align*}
\end{proof}

Next, we compare the contribution of the correction terms coming from the combinatorial flattenings. Figure \ref{niceideal3} shows how the tetrahedra $\Delta_1^1, \Delta_2^1$ meets the cusp $T_1$. Note that by the second condition of Definition \ref{SCF} and \ref{GSCF}, a (generalized) strong combinatorial flattening must satisfy 
\begin{align}\label{combinflat}
f_1=f_2, \quad f_1'=f_2' \quad\text{ and }\quad f_1''=f_2''.
\end{align} 
Given a combinatorial flattening $\widehat{\boldsymbol{\mathcal{F'}}}=(\mathbf{f},\mathbf{f'},\mathbf{f''})$ of $\widehat{\mathcal{T}}'$, we extend $\widehat{\boldsymbol{\mathcal{F'}}}$ to a generalized combinatorial flattening $\widehat{\boldsymbol{\mathcal{F}}  }=(\widehat{\mathbf{f}},\widehat{\mathbf{f}}',\widehat{\mathbf{f}}'')$ of $\widehat{\mathcal{T}}$ as shown in Figure \ref{niceideal4}, where $b$ is unique number so that the second condition of Definition \ref{SCF} for the edges $f$ and $g$ are satisfied. Note that $b$ is either an integer or a half integer.

\begin{figure}[h]
\centering
\includegraphics[scale=0.16]{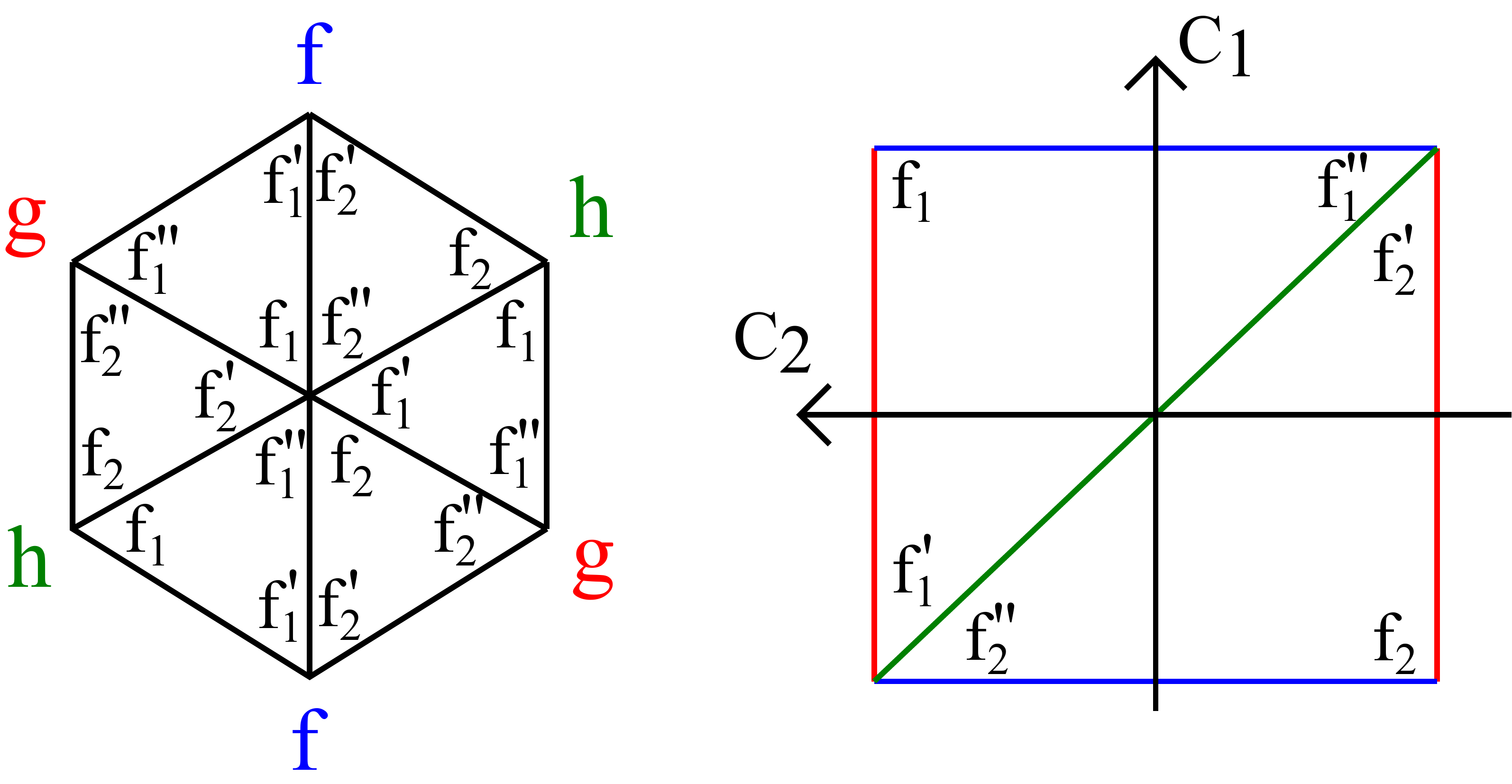}
\caption{A combinatorial flattening restricted on $T_1$. The figure on the left shows the triangles around the edge that goes into $T_1$. From the figure on the right, by considering the second condition of Definition \ref{SCF} and \ref{GSCF} for the simple closed curves $C_1, C_2$, we have $f_1=f_2, f_1'=f_2'$ and $f_1''=f_2''$.}\label{niceideal3}
\end{figure}

\begin{figure}[h]
\centering
\includegraphics[scale=0.16]{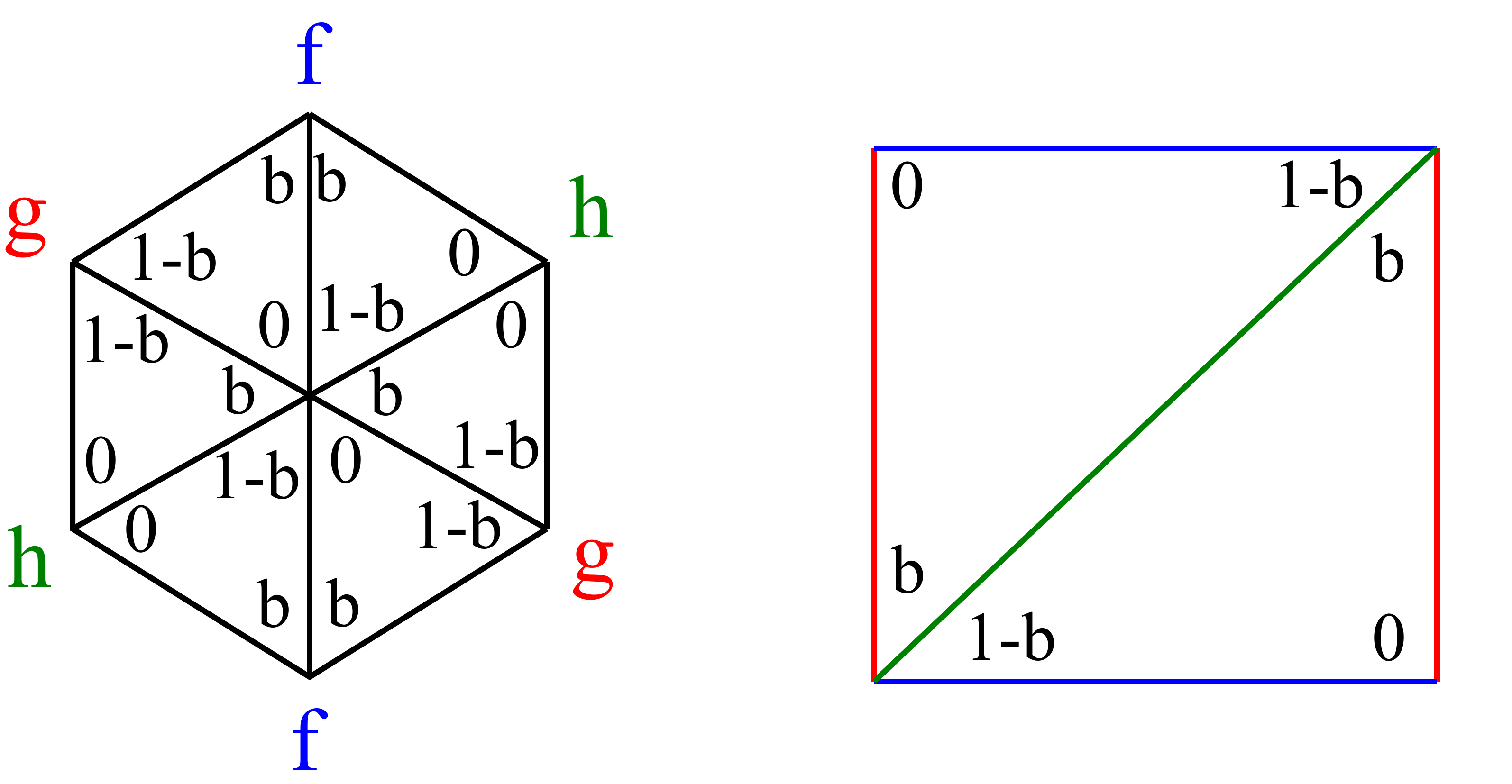}
\caption{A generalized strong combinatorial flattening restricted on $T_1$.}\label{niceideal4}
\end{figure}

\begin{lemma}\label{changegencom}
Let $\tau(M,\boldsymbol \alpha,\rho,\widehat{\mathcal{T}},\widehat{\boldsymbol{\mathcal{F}}  })$ be the 1-loop invariants defined with respect to the generalized strong combinatorial flattenings $\widehat{\boldsymbol{\mathcal{F}}  }=(\widehat{\mathbf{f}},\widehat{\mathbf{f}}',\widehat{\mathbf{f}}'')$. Then we have
$$\tau(M,\boldsymbol \alpha,\rho,\widehat{\mathcal{T}},\widehat{\boldsymbol{\mathcal{F}}  })
=
\tau(M,\boldsymbol \alpha,\rho,\widehat{\mathcal{T}}).$$
Furthermore, with respect to $\widehat{\boldsymbol{\mathcal{F}}  }=(\widehat{\mathbf{f}},\widehat{\mathbf{f}}',\widehat{\mathbf{f}}'')$,
\begin{align*}
\prod_{i=n+1}^{n+2} \xi_i^{f_i} \xi_i'^{f_i'} \xi_i''^{f_i''}
= \pm \frac{1}{(1-z_{n+1})(1-z_{n+2})}.
\end{align*}
\end{lemma}
\begin{proof}
We first specify the quad type and the set of edge equations we used for the computation of the 1-loop invariants. First, we let $\hat{E}'=\{e_1,\dots,e_{n-k-1},f+g\}$ be a set of linearly independent edges of the ideal triangulation $\widehat{\mathcal{T}}'$ of $M'$ and let $\alpha_2,\dots,\alpha_k$ be a system of simple closed curves on $T_2\coprod \dots \coprod T_k$. By \cite[Lemma A.3]{DG}, we may choose a quad type for $\widehat{\mathcal{T}}'$ so that the corresponding matrix $B_{\widehat{\mathcal{T}}'} = G_{\widehat{\mathcal{T}}'}'' - G_{\widehat{\mathcal{T}}'}'$ is invertible, where $G_{\widehat{\mathcal{T}}'}''$ and $G_{\widehat{\mathcal{T}}'}'$ are the Neumann-Zagier matrices defined in Section \ref{NZD} with respect to $\widehat{\mathcal{T}}'$. By abuse of notations, $\{e_1,\dots,e_{n-k-1}\}$ and $\{\alpha_2,\dots,\alpha_k\}$ are also edges and boundary curves of $\widehat{\mathcal{T}}$. Let $\hat{E}=\{e_1,\dots,e_{n-k-1},f,g\}$. Extend the quad type of $\widehat{\mathcal{T}}'$ to a quad type of $\widehat{\mathcal{T}}$ by assigning shape parameters to $\Delta_1^1, \Delta_2^1$ as shown in Figure \ref{niceideal2}. Let $B_{\widehat{\mathcal{T}}} =  G_{\widehat{\mathcal{T}}}'' - G_{\widehat{\mathcal{T}}}'$. Note that by (\ref{compufandg1}), (\ref{compufandg2}) and (\ref{compualpha}), by switching the rows of $B_{\widehat{\mathcal{T}}}$, the matrix $B_{\widehat{\mathcal{T}}}$ is of the form
\begin{align*}
\begin{pNiceMatrix}[first-row, first-col]
 &  z_1& \dots &z_n  & z_{n+1} & z_{n+2}  \\
\alpha_2 &*  &\dots & * & 0 & 0 \\
\vdots  & \vdots  &\dots & \vdots & \vdots & \vdots  \\
\alpha_{k} &*  &\dots & * & 0 & 0 \\
e_1 &*  &\dots & * & 0 & 0 \\
\vdots  & \vdots  &\dots & \vdots & \vdots & \vdots  \\
e_{n-k-1} &*  &\dots & * & 0 & 0 \\
f & * & \dots & * & -1 & -1 \\
g & * & \dots & * & 1 & 1 \\
\alpha & 0 & \dots & 0 & -2 & 2 
\end{pNiceMatrix},
\end{align*}
which has the same determinant as the matrix
\begin{align*}
\begin{pNiceMatrix}[first-row, first-col]
 &  z_1& \dots &z_n  & z_{n+1} & z_{n+2}  \\
\alpha_2 &*  &\dots & * & 0 & 0 \\
\vdots  & \vdots  &\dots & \vdots & \vdots & \vdots  \\
\alpha_{k} &*  &\dots & * & 0 & 0 \\
e_1 &*  &\dots & * & 0 & 0 \\
\vdots  & \vdots  &\dots & \vdots & \vdots & \vdots  \\
e_{n-k-1} &*  &\dots & * & 0 & 0 \\
f+g & * & \dots & * & 0 & 0 \\
g & * & \dots & * & 1 & 1 \\
\alpha & 0 & \dots & 0 & -2 & 2 
\end{pNiceMatrix}.
\end{align*}
By Proposition \ref{niceidealintro} (3), we may assume that the Neumann-Zagier datum for $\alpha_2,\dots,\alpha_k$ remains unchanged under the Dehn-filling. This implies that $\det B_{\widehat{\mathcal{T}}} = \pm 4 \det B_{\widehat{\mathcal{T}}}' \neq 0$. As a result, $\hat{E}$ is a set of linearly independent edges of $\widehat{\mathcal{T}}$ and $B_{\widehat{\mathcal{T}}}$ is invertible. 

Now we can prove the lemma as follows. Let ${\boldsymbol{\mathcal{F}}  }=(\widehat{\mathbf{f}},\widehat{\mathbf{f}}',\widehat{\mathbf{f}}'')$ be a strong combinatorial flattening of $\widehat{\mathcal{T}}$. 
Similar to the proof of Lemma \ref{diffg}, since $\det B_{\widehat{\mathcal{T}}}\neq 0$, by the argument in \cite[Section 4.5]{DG}, we have
$$
\frac{\prod_{i=1}^{n+2} \Big(z_i^{\hat{f_i}''} z_i''^{-\hat{f_i}}\Big)}{\prod_{i=1}^{n+2} \Big(z_i^{f_i''} z_i''^{-f_i}\Big)}
= e^{(\mathbf f'' \cdot \hat{\mathbf f} - \mathbf f \cdot \hat{\mathbf f}'')\pi\sqrt{-1}}.
$$ 
We consider $(\mathbf f'' \cdot \hat{\mathbf f} - \mathbf f \cdot \hat{\mathbf f}'')$ inside and outside hexagon shown in Figure \ref{niceideal4}. 
Note that outside the hexagon shown on the left of Figure \ref{niceideal4}, since $f,f'',\hat{f},\hat{f''}\in \ZZ$, we have $(f'' \cdot \hat{f} - f \cdot \hat{f}'')\in \ZZ$. Inside the hexagon, by (\ref{combinflat}), we also have $(f'' \cdot \hat{f} - f \cdot \hat{f}'') \in \ZZ$. Altogether, we have
$$
\frac{\prod_{i=1}^{n+2} \Big(z_i^{\hat{f_i}''} z_i''^{-\hat{f_i}}\Big)}{\prod_{i=1}^{n+2} \Big(z_i^{f_i''} z_i''^{-f_i}\Big)}
= e^{(\mathbf f'' \cdot \hat{\mathbf f} - \mathbf f \cdot \hat{\mathbf f}'')\pi\sqrt{-1}} = \pm 1.
$$ 
This proves the first claim. For the second claim, note that 
\begin{align*}
\prod_{i=n+1}^{n+2} \xi_i^{f_i} \xi_i'^{f_i'} \xi_i''^{f_i''}
= \prod_{i=n+1}^{n+2} \left(\frac{1}{z_i}\right)^0 \left(\frac{1}{z_i-1}\right)^\frac{1}{2} \left(\frac{1}{z_i(z_i-1)}\right)^\frac{1}{2}
=  \pm\frac{1}{(1-z_{n+1})(1-z_{n+2})},
\end{align*}
where the last equality follows from $z_{n+1}z_{n+2}=1$.
\end{proof}

\begin{proof}[Proof of Theorem \ref{1loopsurg}]
By Lemma \ref{changegencom}, we can compute the 1-loop invariant by using the generalized combinatorial flattening $\widehat{\boldsymbol{\mathcal{F}}  }=(\widehat{\mathbf{f}},\widehat{\mathbf{f}}',\widehat{\mathbf{f}}'')$. Using the same notations in the proof of Lemma \ref{changegencom}, by (\ref{comf+g}) and Lemma \ref{1loopsurglem}, the 1-loop invariant $\tau(M,\boldsymbol \alpha,\rho,\widehat{\mathcal{T}})$ of $M$ is up to a sign equal to 
\begin{align*}
&\frac{\det \left(
\begin{pNiceMatrix}[first-row, first-col]
 &  z_1& \dots &z_n  & z_{n+1} & z_{n+2}  \\
 \alpha_2 &*  &\dots & * & 0 & 0 \\
\vdots  & \vdots  &\dots & \vdots & \vdots & \vdots  \\
\alpha_{k} &*  &\dots & * & 0 & 0 \\
e_1 &*  &\dots & * & 0 & 0 \\
\vdots  & \vdots  &\dots & \vdots & \vdots & \vdots  \\
e_{n-k-1} &*  &\dots & * & 0 & 0 \\
f & * & \dots & * & \frac{1}{1-z_{n+1}} & \frac{1}{1-z_{n+2}} \\
g & * & \dots & * & \frac{1}{z_{n+1}(z_{n+1}-1)} & \frac{1}{z_{n+2}(z_{n+2}-1)} \\
\alpha & 0 & \dots & 0 & \frac{1}{1-z_{n+1}} - \frac{1}{z_{n+1}(z_{n+1}-1)} & \frac{1}{z_{n+2}(z_{n+2}-1)}  - \frac{1}{1-z_{n+2}}  
\end{pNiceMatrix}
\right)}{\Big(\prod_{i=1}^{n} \xi_i^{f_i} \xi_i'^{f_i'} \xi_i''^{f_i''}
\Big)
\Big(\prod_{i=n+1}^{n+2} \xi_i^{f_i} \xi_i'^{f_i'} \xi_i''^{f_i''}
\Big)}\\
=&\ \ \frac{\det \left(
\begin{pNiceMatrix}[first-row, first-col]
 &  z_1& \dots &z_n  & z_{n+1} & z_{n+2}  \\
 \alpha_2 &*  &\dots & * & 0 & 0 \\
\vdots  & \vdots  &\dots & \vdots & \vdots & \vdots  \\
\alpha_{k} &*  &\dots & * & 0 & 0 \\
e_1 &*  &\dots & * & 0 & 0 \\
\vdots  & \vdots  &\dots & \vdots & \vdots & \vdots  \\
e_{n-k-1} &*  &\dots & * & 0 & 0 \\
f+g+C\alpha & * & \dots & * & 0 & 0 \\
g & * & \dots & * & \frac{1}{z_{n+1}(z_{n+1}-1)} & \frac{1}{z_{n+2}(z_{n+2}-1)} \\
\alpha & 0 & \dots & 0 & \frac{1}{1-z_{n+1}} - \frac{1}{z_{n+1}(z_{n+1}-1)} & \frac{1}{z_{n+2}(z_{n+2}-1)}  - \frac{1}{1-z_{n+2}}  
\end{pNiceMatrix}
\right)}{\Big(\prod_{i=1}^{n} \xi_i^{f_i} \xi_i'^{f_i'} \xi_i''^{f_i''}
\Big)
\Big(\prod_{i=n+1}^{n+2} \xi_i^{f_i} \xi_i'^{f_i'} \xi_i''^{f_i''}
\Big)},
\end{align*}
where $C$ is the constant in Lemma \ref{1loopsurglem}.
Note that the $(z_1,\dots,z_n)$ components of $f+g+C\mu$ and $f'+g'$ are the same. By Lemmas \ref{1loopsurglem} and \ref{changegencom}, we have 
\begin{align*}
\tau(M,\boldsymbol \alpha,\rho,\widehat{\mathcal{T}})
=&\pm\frac{\det \left(
\begin{pNiceMatrix}[first-row, first-col]
 & z_1 & \dots & z_n    \\
 \alpha_2 &*  &\dots & *  \\
\vdots  & \vdots  &\dots & \vdots   \\
\alpha_{k} &*  &\dots & *  \\
e_1 &*  &\dots & *  \\
\vdots  & \vdots  &\dots & \vdots \\
e_{n-k-1} &*  &\dots & * \\
f'+g' & * & \dots & * 
\end{pNiceMatrix}
\right)}{
\prod_{i=1}^{n} \xi_i^{f_i} \xi_i'^{f_i'} \xi_i''^{f_i''}}
\times (z_1+z_2+2)\\
=& \pm \tau(M',\boldsymbol \alpha',\rho',\widehat{\mathcal{T}}')
\Bigg(4\sinh^2 \frac{\mathrm{H}(\gamma)}{2}\Bigg).
\end{align*}
\end{proof}

\subsection{Proof of Corollary \ref{corsufflong}}
Let $M'$ be a hyperbolic manifold obtained by doing sufficiently long Dehn-fillings on the boundary components of a fundamental shadow link complement. Let $\rho': \pi_1(M') \to \mathrm{PSL}(2;\CC)$ be a representation that is sufficiently close to the discrete faithful representation of $M'$ such that $\rho= \rho'|_{\pi_1(M)} \in U_M$. Let $\widehat{\mathcal{T}}$ be the ideal triangulation of $M$ in Theorem \ref{1loopsurg}. We want to apply Theorem \ref{mainthm} and \ref{1loopreallyinv}. However, it is not clear whether the triangulation $\mathcal{T}$ in Theorem \ref{idealtri} is $\rho$-regular. Nevertheless, since $\widehat{\mathcal{T}}$ is $\rho$-regular and $[\rho]\in \mathcal{Z}_{\boldsymbol \alpha}$, it is also $\tilde\rho$-regular for $\tilde\rho$ sufficiently close to $\rho$. Moreover, generically the triangulation $\mathcal{T}$ in Theorem \ref{idealtri} is $\tilde\rho$-regular. By Theorem \ref{mainthm}, \cite[Theorem 1.4, 4.1]{DG}, Proposition \ref{inv02move} and \cite[Theorem A.1]{KSS}, for a generic $\tilde\rho$ sufficiently close to $\rho$ with $\tilde\rho = \mathcal{P}_{\widehat{\mathcal{T}}}(\mathbf{\tilde z})$ for some $\mathbf{\tilde z} \in \mathcal{V}_0(\widehat{\mathcal{T}})$, we have
$$
\tau(M,\boldsymbol\alpha, \mathbf{\tilde z}, \widehat{\mathcal{T}}) 
=  \pm \mathbb T_{(M,\boldsymbol\alpha)}([\tilde\rho]).
$$
By continuity if necessary, we have
$$
\tau(M,\boldsymbol\alpha, \mathbf{\widehat{z}}, \widehat{\mathcal{T}}) 
=  \pm \mathbb T_{(M,\boldsymbol\alpha)}([\rho]).
$$
By Theorem \ref{1loopsurg}, Theorem \ref{funT} (iii), Remark \ref{rmksurg}, we have
\begin{align*}\label{pfsufflonglast}
\tau(M',\boldsymbol\alpha', \mathbf{\widehat z'}, \widehat{\mathcal{T}}') 
= \pm \mathbb T_{(M',\boldsymbol\alpha')}([\rho']).
\end{align*}
By Proposition \ref{toranaglu}, both the 1-loop invariant and the torsion are analytic functions on the gluing variety $\mathcal{V}_0(\mathcal{T})$. As a result, by analyticity, they agree on $\mathcal{V}_0(\mathcal{T'})$. By Corollary \ref{1loopreallyinvcor}, we have the desired result.

\noindent
Tushar Pandey\\
Department of Mathematics\\  Texas A\&M University\\
College Station, TX 77843, USA\\
(tusharp@tamu.edu)\\

\noindent
Ka Ho Wong\\
Department of Mathematics\\  Yale University\\
 New Haven, CT 06511, USA\\
(kaho.wong@yale.edu)
\\

\end{document}